\documentclass[10pt]{amsproc}
\usepackage{fullpage}
\usepackage{graphicx}
\usepackage{amssymb}
\usepackage{amsfonts}
\usepackage{amsthm}
\usepackage{amsmath}
\usepackage[shortalphabetic]{amsrefs}
\usepackage{xypic}
\usepackage{eucal}
\usepackage{pdfsync}
\usepackage{hyperref}
\usepackage{enumerate}
\usepackage[margin=1.25in]{geometry}

\def\ra{\rightarrow}
\def\Z{{\mathbb Z}}

\def\R{{\mathbb R}}
\def\C{{\mathbb C}}

\def\K{\mathbf{k}}
\def\w{\mathcal{W}}
\def\f{\mathcal{F}}

\def\e{\epsilon}
\def\cc{\mathcal{C}}
\def\mod{\mathrm{-mod}}

\def\e{\epsilon}

\def\b{\beta}
\def\r#1{\mathrm{#1}}

\def\c#1{\mathcal{#1}}
\def\mc#1{\mathcal{#1}}
\def\mb#1{\mathbf{#1}}
\def\w{\mathcal{W}}

\def\ainf{A_\infty}
\def\f{\c{F}}

\def\z2{\Z / 2\Z}

\def\lra{\longrightarrow}
\def\ra{\rightarrow}

\def\ob{\mathrm{ob\ }}

\newtheorem{lem}{Lemma}
\newtheorem{prop}{Proposition}
\newtheorem{thm}{Theorem}
\newtheorem{cor}{Corollary}

\newtheorem{defn}{Definition}

\newtheorem{assumption}{Assumption}

\theoremstyle{remark}
\newtheorem{rem}{Remark}

\numberwithin{equation}{section}
\begin{document}

\begin{abstract}
    We construct geometric maps from the cyclic homology groups of the (compact or wrapped) Fukaya category to the corresponding $S^1$-equivariant (Floer/quantum or symplectic) cohomology groups, which are natural with respect to all Gysin and periodicity exact sequences and are isomorphisms whenever the (non-equivariant) open-closed map is.  These {\em cyclic open-closed maps} give (a) constructions of geometric smooth and/or proper Calabi-Yau structures on Fukaya categories (which in the proper case implies the Fukaya category has a cyclic $\ainf$ model in characteristic 0) and (b) a purely symplectic proof of the non-commutative Hodge-de Rham degeneration conjecture for smooth and proper subcategories of Fukaya categories of compact symplectic manifolds.  Further applications of cyclic open-closed maps, to counting curves in mirror symmetry and to comparing topological field theories, are the subject of joint projects with Perutz-Sheridan \cites{GPS1:2015,GPS2:2015} and Cohen \cite{CohenGanatra:2015}.
\end{abstract}

\def\oc{\mc{OC}}
\def\co{\mc{CO}}
\def\b{\mc{B}}
\def\z{ { \mathbb Z}}
\title{Cyclic homology, $S^1$-equivariant Floer cohomology, and Calabi-Yau structures}
\author{Sheel Ganatra}
\thanks{The author was partially supported by the National Science Foundation through a postdoctoral fellowship --- grant number DMS-1204393 --- and agreement number DMS-1128155. Any opinions, findings and conclusions or recommendations expressed in this material are those of the author(s) and do not necessarily reflect the views of the National Science Foundation.}
\maketitle
\section{Introduction}

This paper concerns the compatibility between 
chain level $S^1$ actions arising in two different types of Floer theory on a
symplectic manifold $M$.  The first of these $C_{-*}(S^1)$\footnote{We use a
cohomological grading convention in this paper, so singular chain complexes
are {\em negatively graded}.} 
actions is induced geometrically on the {\em Hamiltonian Floer homology chain
complex} $CF^*(M)$, formally a type of Morse complex for an action functional
on the free loop space, through rotating free loops.  The homological action of
$[S^1]$ is known as the \emph{BV operator} $[\Delta]$, and the $C_{-*}(S^1)$
action can be used to define {\em $S^1$-equivariant Floer homology theories}
--- see e.g., \cite{Seidel:2010fk,
Bourgeois:2012fk}\footnote{Sometimes $S^1$-equivariant Floer theory is instead
defined as Morse theory of an action functional on the $S^1$-Borel
construction of the loop space. For a comparison between these two definitions,
see \cite{Bourgeois:2012fk}.}.
The second $C_{-*}(S^1)$ action lies on the {\em Fukaya category} of $M$, and has
discrete or combinatorial origins, coming from the hierarchy of compatible
cyclic $\Z/k\Z$ actions on cyclically composable chains of morphisms between
Lagrangians.  A (categorical analogue of a) fundamental observation of Connes',
Tsygan, and Loday-Quillen that such a structure, which exists on any category $\cc$,
can be packaged into a $C_{-*}(S^1)$ action on the {\it Hochschild homology
chain complex} $\r{CH}_*(\cc)$ of the category
(see e.g., \cites{Connes:1985aa, Tsygan:1983aa,Loday:1984aa, McCarthy:1994aa, Keller:1999aa})
The associated operation of multiplcation by (a cycle representing) $[S^1]$ is
frequently called the {\em Connes' $B$ operator} $B$, and the corresponding
$S^1$-equivariant homology theories are called {\em cyclic homology groups}. 

A relationship between the Hochschild homology of the Fukaya category $\mc{F}$ and Floer
homology on $M$ is provided by the so-called {\em open-closed string map}
\cite{Abouzaid:2010kx}
\begin{equation}
    \label{openclosedmap}
    \oc: \r{CH}_*(\mathcal{F}) \ra CF^{*+n}(M).
\end{equation}
Our main result is about the compatibility of $\oc$ with $C_{-*}(S^1)$ actions.
Namely, we prove  --- under technical hypotheses detailed below the main
result --- that $\oc$ can be made (coherently homotopically)
$C_{-*}(S^1)$-equivariant: 
\begin{thm}\label{thm:mainresult1}
Suppose $M$, its Fukaya category, and $CF^*(M)$ satisfy the technical
assumptions $(\star)$.  Then the map $\oc$ admits a geometrically defined `{\em
$S^1$-equivariant enhancement}', to an $\ainf$ homomorphism of
$C_{-*}(S^1)$-modules, $\widetilde{\oc} \in
\mathrm{RHom}_{C_{-*}(S^1)}^n(\r{CH}_*(\f),
    \r{CF}^*(M))$.
\end{thm}
\begin{rem}
    Theorem \ref{thm:mainresult1} implies (but is not implied by) the statement
    (see Theorem \ref{thm:homology}) that $[\oc]$ intertwines homological
    actions of $[S^1]$.
\end{rem}
\begin{rem}
    In the geometric settings considered here $\oc$ does not a priori strictly intertwine
    the $C_{-*}(S^1)$ actions (due to a priori non-equivariant perturbations made to moduli spaces to define operations, and further due to the potential non-triviality of $\Delta$, which --- as $\Delta$ is defined using moduli spaces but $B$ is defined using algebra --- imply that $\oc \circ B$ and $\Delta \circ \oc$ involve moduli spaces of maps from differing domains).
    In particular, the homomorphism
    $\widetilde{\oc}$ involves extra data recording coherently higher
    homotopies between the two $C_{-*}(S^1)$ actions. This explains our use
    of the term `enhancement'.
\end{rem}
\begin{rem}
It can be shown using usual invariance arguments that the enhancement
$\widetilde{\oc}$ we define in this paper is uniquely determined up to homotopy: 
while the geometric chain-level construction requires a number of auxiliary choices (of
perturbation data on moduli spaces), any two sets of such choices produce homotopic
enhancements.
\end{rem}
To explain the consequences of Theorem \ref{thm:mainresult1} to cyclic homology
and equivariant Floer homology, recall that there are a variety of {\em
$S^1$-equivariant homology} chain complexes (and homology groups) that one can
associate functorially to an $\ainf$ $C_{-*}(S^1)$ module $P$.
For instance, denote by 
\begin{equation}\label{orbitfixedpointtate}
    P_{hS^1},\ P^{hS^1},\ P^{Tate}
\end{equation}
the {\em homotopy orbit complex}, {\em homotopy fixed point complex}, and {\em
Tate complex} constructions of $P$, described in \S \ref{sec:equivariantgroups}. When
applied to the Hochschild complex $\r{CH}_*(\cc)$, the constructions
\eqref{orbitfixedpointtate} by definition recover complexes computing
{\em (positive) cyclic homology}, {\em negative cyclic homology}, and {\em
periodic cyclic homology} groups of $\cc$ respectively (see \S
\ref{subsec:cyclic}).
Similarly the group $H^*(CF^*(M)_{hS^1})$ is the {\em $S^1$-equivariant Floer
cohomology} studied (for the symplectic homology Floer chain complex); see e.g.,
\cite{Viterbo:1999fk, Seidel:2010fk,  Bourgeois:2012fk}.
The groups
$H^*(CF^*(M)^{hS^1})$ and $H^*(CF^*(M)^{Tate})$ have also been studied in recent work in Floer theory
\cites{Seidel:2018aa, Zhao:2019aa, Albers:2016aa}.
Functoriality of the constructions \eqref{orbitfixedpointtate} and homotopy
invariance properties of $C_{-*}(S^1)$ modules (see Cor.
\ref{S1homotopyinvariance} and Prop. \ref{functorialitysequences}) immediately
imply the result announced in the abstract: 
\begin{cor}\label{cor:cyclicOC}
    Let $HF^{*, +/-/\infty}_{S^1}(M)$ denote the (cohomology of the) homotopy
    orbit complex, fixed point complex, and Tate complex construction applied
    to $CF^*(M)$, and let $\r{HC}^{+/-/\infty}(\cc)$ denote the corresponding
    positive/negative/periodic cyclic homology groups.
    Then under the hypotheses $(\star)$ of Theorem \ref{thm:mainresult1},
    $\widetilde{\oc}$ induces \emph{cyclic open-closed maps}
    \begin{equation}
        [\widetilde{\oc}^{+/-/\infty}]: \mathrm{HC}^{+/-/\infty}_*(\mathcal{F}) \ra HF^{*+n,+/-/\infty}_{S^1}(M),
    \end{equation}
    naturally compatible with respect to the various periodicity/Gysin exact
    sequences, which are isomorphisms whenever $\oc$ is.\qed
\end{cor}
The map \eqref{openclosedmap} is frequently an isomorphism, allowing one to
recover in these cases closed string Floer/quantum homology groups from open
string, categorical ones \cite{BEE1published, ganatra1_arxiv,
GPS1:2015, afooo}. In such cases, Theorem \ref{thm:mainresult1} and Corollary
\ref{cor:cyclicOC} allow one to further categorically recover the
$C_{-*}(S^1)$ as well as the associated equivariant homology groups (in terms
of the cyclic homology groups of the Fukaya category).
\begin{rem}
    There are other $S^1$-equivariant homology functors, to which our results apply
    tautologically as well. For instance, consider the contravariant
    functor $P \mapsto (P_{hS^1})^{\vee}$; when applied to $\r{CH}_*(\cc)$ this
    produces the {\em cyclic cohomology} chain complex of $\cc$.
\end{rem}
We have been deliberately vague about which Fukaya category and which
Hamiltonian Floer homology groups Theorem \ref{thm:mainresult1} applies to, as
it applies in several different geometric (compact and non-compact) settings.  To keep this paper a manageable
length, we implement the map $\widetilde{\oc}$ and prove Theorem
\ref{thm:mainresult1} in the technically simplest of such settings --- our technical hypotheses are detailed in $(\star)$ below --- for which the
moduli spaces appearing in the constructions can be shown to be well-behaved by
classical methods.  That being said, we should remark that our methods and
arguments are orthogonal to the usual analytic difficulties faced in
constructing Fukaya categories and open-closed maps in more general contexts,
and we expect they should extend relatively directly to other settings. For instance,
in the setting of relative Fukaya categories of compact projective Calabi-Yau
manifolds (not considered here), an adapted version of our construction will appear in joint work
with Perutz-Sheridan \cite{GPS2:2015}.

\begin{center}
    \noindent {\bf $(\star)$ Assumptions on $M$, $\mc{F}$, and $CF^*(M)$:}
\end{center}
In our main results we make technical assumptions, explained in detail in \S
\ref{subsec:fukaya} for $M$ and its Fukaya category and in \S
\ref{subsubsec:sh}-\ref{subsubsec:relh} for
the corresponding Hamiltonian Floer homology chain complexes, which broadly
encapsulate the following situations: 
\begin{enumerate}
    \item If $M$ is compact and satisfies suitable technical hypotheses such as being monotone or symplectically aspherical (see \S
        \ref{subsec:assumptions}), one could take $\mathcal{F}$ to be the usual Fukaya category (or a summand thereof)
        of those compact Lagrangians also satisfying suitable technical
        hypotheses such as being monotone or not bounding disks with symplectic
        area. In this case $CF^*(M)$, the Hamiltonian
        Floer complex of any (sufficiently generic) Hamiltonian, is
        quasi-isomorphic to the {\em quantum cohomology} ring with its trivial
        $C_{-*}(S^1)$ action.

    \item If $M$ is non-compact and Liouville, one could take $\mathcal{F} = \mathcal{W}$ to
        be the {\em wrapped Fukaya category} and $CF^*(M) = SC^*(M)$ to be the {\em
        symplectic cohomology co-chain complex} with its (typically highly non-trivial)
        $C_{-*}(S^1)$ action.

    \item If $M$ is non-compact and Liouville, one could take $\mathcal{F} \subset \mathcal{W}$ to
        be the {\em Fukaya category of compact exact Lagrangians}. When restricted to $\r{CH}_*(\mc{F})$, the map
        $\oc$ to $SC^*(M)$ of (2) factors through $H^*(M, \partial^{\infty} M)$, the
        {\em relative (or compactly supported) cohomology group} with its
        trivial $C_{-*}(S^1)$ action.  In fact, as reviewed in \S
        \ref{sec:compactopenclosed} $\oc$ further factors through the
        symplectic {\em homology} chain
        complex $SC_*(M) \cong (SC^*(M))^{\vee}[-2n]$. One could take any of these groups ($SC^*(M)$, $H^*(M, \partial^{\infty} M)$, or $SC_*(M)$) to be $CF^*(M)$ here. For the main portion of the paper we use $CF^*(M):= H^*(M, \partial^{\infty} M)$. 
\end{enumerate}
For example, in case (2) above, when the relevant $[\oc]$ map is an
isomorphism, Corollary \ref{cor:cyclicOC} computes various $S^1$-equivariant
symplectic cohomology groups\footnote{and in particular, the usual equivariant symplectic
cohomology $SH^*_{S^1}(M) = H^*(SC^*(M)_{hS^1})$, 
see \cite{Bourgeois:2012fk} but note differing
conventions regarding e.g., homology vs. cohomology} in terms of cyclic
homology groups of the wrapped Fukaya category.
\begin{rem}
For the Fukaya subcategory of a single Lagrangian in a compact symplectic
manifold $M$ over a characteristic zero (Novikov) field containing $\R$, a
variant of the (positive) cyclic open-closed map has also been constructed by
Fukaya-Oh-Ohta-Ono \cite{foooasterisque} (and will be generalized to multiple
Lagrangians in \cite{afooo}). Their construction, which requires the target
group ($H^*(M)$) to have trivial $C_{-*}(S^1)$ action, uses Connes' small
(``coinvariants of cyclic group action bar'') complex for (in characteristic
zero only) positive cyclic homology, along with cyclically symmetric
(necessarily virtual) perturbations of all moduli spaces (building on work of
Fukaya \cite{Fukaya:2010aa} described in Remark \ref{fukayacyclic}), to
directly construct a geometric map bypassing the higher $\ainf$
$C_{-*}(S^1)$-action homotopies constructed here.  It does not seem possible to
generalize the methods of \cite{foooasterisque} to the (possibly non-compact
$M$ with arbitrary coefficients e.g., integral/rational/finite characteristic)
settings considered here, see for instance the discussion in Remark
\ref{rem:cyclic}.  Also, the perspective of $C_{-*}(S^1)$-modules taken here
makes it simpler to talk about (and describe) all cyclic homology theories at
once, as well to study the compatibility of additional structures (e.g., exact
sequences, semi-infinite/noncommutative Hodge structures).
\end{rem}

\begin{rem}\label{partiallywrapped1}
There are other settings in which Fukaya categories are now well-studied, for
instance Fukaya categories of Lefschetz fibrations (and more general LG
models), or more generally partially wrapped Fukaya categories (such as
wrapped Fukaya categories of {\em Liouville sectors}).  We do not
discuss these situations in our paper, but expect suitable versions of Theorem
\ref{thm:mainresult1} to hold in such settings too. We do note however that the
target of the open-closed map from Hochschild homology in such settings is
usually more subtle than in the cases discussed here, e.g., it does not
typically have the structure of a unital ring.  
\end{rem}

\begin{rem} \label{rmk:thmvariations}
    One can consider variations on Theorem \ref{thm:mainresult1}. 
    As a notable example, let $M$ denote a (noncompact) Liouville manifold, and
    $\mathcal{F}$ the Fukaya category of compact exact Lagrangians in $M$. Then there
    is a non-trivial refinement of the map 
    $\r{HH}_*(\mathcal{F}) \ra H^*(M, \partial^{\infty} M)$, which can be viewed as a pairing $\r{HH}_*(\mathcal{F}) \times H^*(M) \to \K$, to a pairing
    \[
        \r{CH}_*(\f) \otimes SC^*(M) \ra \K.
    \]
    (note symplectic cohomology does not satisfy Poincar\'{e} duality, 
    so this is {\em not} equivalent to a map to symplectic cohomology).
    Our methods also imply that this pairing admits an $S^1$-equivariant enhancement,
    with respect to the diagonal $C_{-*}(S^1)$ action on the left and the
    trivial action on the right.
    Passing to adjoints, we obtain cyclic open closed maps from
    $S^1$-equivariant symplectic cohomology to cyclic {\em cohomology} groups
    of $\f$, and from cyclic homology of $\f$ to equivariant symplectic {\em
    homology}.
    See \S \ref{sec:compactopenclosed} for more details.
\end{rem}

Beyond computing equivariant Floer cohomology groups in terms of cyclic
homology theories, we describe in the following subsection two applications of
Theorem \ref{thm:mainresult1} to the structure of Fukaya categories.
\begin{rem}  
    We anticipate additional concrete applications of Theorem
    \ref{thm:mainresult1} and its homological shadow, Theorem
    \ref{thm:homology}.  For instance, one can
    study the compatibility of open-closed maps with \emph{dilations} in the
    sense of \cite{Seidel:2010uq}, which are elements $B$ in 
    $SH^*(M)$ satisfying $[\Delta] B = 1$; the existence of dilations strongly
    constrains intersection properties of embedded Lagrangians
    \cite{Seidel:2014aa}. 
    Theorem \ref{thm:homology},
    or rather the variant discussed in Remark \ref{rmk:thmvariations}, implies 
    {\em if there exists a dilation, e.g., an element $x \in SH^1(M)$ with
    $[\Delta] x = 1$, then on the Fukaya category of compact Lagrangians
    $\mc{F}$, there exists $x' \in (\r{HH}_{n+1}(\mc{F}))^{\vee}$ with $x'
    \circ [B] = [tr]$}, where
    $tr$
    is the geometric weak proper Calabi-Yau structure on the
    Fukaya category (see \S \ref{intro:cystructures}).
\end{rem}

\subsection{Calabi-Yau structures on the Fukaya category}\label{intro:cystructures}
Calabi-Yau structures are a type of cyclically symmetric duality structure on a
dg or $\ainf$ category $\cc$ generalizing the notion of a nowhere vanishing
holomorphic volume form 
on a complex algebraic variety $X$ in the case $\cc = perf(X)$.  As is well understood, there are two (in
some sense dual) types of Calabi-Yau structures on $\ainf$ categories: 
\begin{enumerate}
    \item {\em proper Calabi-Yau structures} \cite{Kontsevich:2009ab} can be associated to {\em
        proper} categories $\cc$ (those which have cohomologically finite-dimensional
        morphism spaces), abstract and refine the notion of integration against
        a nowhere vanishing holomorphic volume form. For $\cc = perf(X)$ with
        $X$ a proper $n$-dimensional variety, 
        the resulting structure in particular induces the Serre duality pairing
        with trivial canonical sheaf $\r{Ext}^*(\mathcal{E}, \mathcal{F})
        \times \r{Ext}^*(\mathcal{F}, \mathcal{E}) \to \K[-n]$,
        Roughly, a proper Calabi-Yau structure on $\cc$ (of dimension $n$) is a
    map $[\widetilde{tr}]: \r{HC}^{+}_*(\cc) \to \K[-n]$ satisfying a non-degeneracy
        condition.

    \item {\em smooth Calabi-Yau structures} \cite{Kontsevich:uq, Kontsevich:CY} can be
        associated to {\em smooth} categories $\cc$ (those with perfect diagonal
        bimodule), and abstract the notion of the nowhere vanishing holomorphic
        volume form itself, along with the induced identification (by
        contraction against the volume form) of polyvectorfields with differential forms.
        Roughly, a smooth Calabi-Yau structure on $\cc$ (of dimension $n$) is a
        map $[\widetilde{cotr}]: \K[n] \to \r{HC}^-_*(\cc)$, or equivalently an element
        $[\widetilde{\sigma}]$ or ``$[vol_{\cc}]$'' in  $\r{HC}^-_{-n}(\cc)$, satisfying a
        non-degeneracy condition. 
\end{enumerate} 
In both cases, the non-degeneracy condition can be phrased purely in terms of
the underlying non-equivariant shadow of the map, e.g. in the first case on the induced map
$[tr]: \r{HH}_{*}(\cc) \to \r{HC}^+(\cc) \stackrel{[\widetilde{tr}]}{\to}
\K[-n]$.  Precise definitions are reviewed in \S \ref{section:cystructures}.
When $\cc$
is simultaneously smooth and proper, it is a folk result that the notions are
equivalent; see \cite{GPS1:2015}*{Prop. 6.10}.

In general, Calabi-Yau structures may not exist and when they do, there may be
a non-trivial space of choices (see \cite{Menichi:2009aa} for an example). 
A Calabi-Yau structure in either form induces non-trivial identifications between
Hochschild invariants of the underlying category $\cc$.\footnote{In the
proper case, there is an induced isomorphism between Hochschild cohomology and
the linear dual of Hochschild homology. In the smooth case, there is an
isomorphism between Hochschild cohomology and homology without taking duals.}
Moreover, categories with Calabi-Yau structures (should) carry induced 2-dimensional
chain level TQFT operations on their Hochschild homology chain complexes,
associated to moduli spaces of Riemann surfaces with marked points
\cite{Costello:2006vn, Kontsevich:2009ab}
(in the smooth case, this has not yet been established, but is ongoing work
\cite{Kontsevich:uq, Kontsevich:CY}); if the category is proper and
non-smooth (respectively smooth non-proper) the resulting TQFT is incomplete in
that every operation must have at least one input (respectively output).  In
the smooth and proper case in particular, Calabi-Yau structures play a central role in the
mirror symmetry-motivated question of recovering
Gromov-Witten invariants from the Fukaya category and to the related question
of categorically recovering Hamiltonian Floer homology with all of its
(possibly higher homotopical) operations.  See \cite{Costello:2006vn,
Costello:2009aa, Kontsevich:2008aa} for work around these questions in the
setting of abstract topological field theories and
\cite{GPS1:2015} for applications of Calabi-Yau structures to recovering
genus-0 Gromov-Witten invariants from the Fukaya category.
\begin{rem}\label{rem:cyclic}
    A closely related to (1), and well studied, notion is that of a {\em cyclic
    $\ainf$ category}: this is an $\ainf$ category $\cc$ equipped with a chain
    level perfect pairing \[\langle -, - \rangle : \hom(X,Y) \times \hom(Y,X) \to
    \K[-n] \]
    such that the induced correlation functions
    \[
       \langle \mu^d(-, -, \ldots, -), - \rangle
    \]
    are strictly (graded) cyclically symmetric, for each $d$
    see e.g.,
    \cite{Costello:2006vn, Fukaya:2010aa, Cho:2012aa}. 
    Although the property of being a cyclic $\ainf$ structure is not a homotopy
    invariant notion (i.e., not preserved under $\ainf$ quasi-equivalences), 
    cyclic $\ainf$ categories and proper Calabi-Yau structures turn out to be
    weakly equivalent {\em in characteristic 0}, in the following sense. 
    Any cyclic $\ainf$ category carries a canonical proper Calabi-Yau
    structure,
    and Kontsevich-Soibelman proved that a proper Calabi-Yau structure on any
    $\ainf$ category $\cc$ determines a (canonical up to quasi-equivalence)
    quasi-isomorphism between $\cc$ and a cyclic $\ainf$ category
    $\widetilde{\cc}$ \cite[Thm.  10.7]{Kontsevich:2009ab}.  
    When $\r{char}(\K)
    \neq 0$, the two notions of proper Calabi-Yau and cyclic $\ainf$ differ in
    general, due to group cohomology obstructions to imposing cyclic symmetry.
    In such instances, it seems that the notion of a proper Calabi-Yau structure is the
    ``correct'' one (as it is a homotopy invariant notion and, by Theorem
    \ref{mainthm3}, the compact Fukaya category always has one).  
\end{rem}
As a first application of Theorem
\ref{thm:mainresult1}, we verify the longstanding expectation that various
compact Fukaya categories possess geometrically defined canonical Calabi-Yau
structures:
\begin{thm}\label{mainthm3} 
The Fukaya category of compact Lagrangians has --- under technical hypotheses
$(\star)$ --- a canonical geometrically defined proper Calabi-Yau structure over
any ground field $\K$ (over which the Fukaya category and $\widetilde{\oc}$ are
defined).
\end{thm}
In fact, this proper Calabi-Yau structure is easy to describe in terms of the
cyclic open-closed map (c.f., Cor. \ref{cor:cyclicOC}):
it is the composition of the map $\widetilde{\oc}^+: \r{HC}^+_*(\mc{F}) \ra
H^{*+n}(M, \partial M)( ( u ) ) / u H^{*+n} (M,\partial M)[ [ u ] ]$\footnote{Recall that $C^*(M, \partial M)$ has the trivial $C_{-*}(S^1)$ module structure; the homology of the associated homotopy orbit complex is $H^{*+n}(M, \partial M)( ( u ) ) / u H^{*+n} (M,\partial M)[ [ u ] ]$ where $|u| = 2$, as described in \S \ref{section:s1action}.} with the
linear map  to $\K$ 
which sends the top class $PD(pt) \cdot u^0 \in H^{2n}(M, \partial M)$ to $1$,
and all other generators $\alpha \cdot u^{-i}$ to 0. See \S \ref{section:cystructures} for more details.

As a consequence of the discussion in Remark \ref{rem:cyclic}, specifically
\cite[Thm.  10.7]{Kontsevich:2009ab}, we deduce that 
\begin{cor}
    If $\r{char}(\K) = 0$, then any Fukaya category of compact Lagrangians
    satisfying $(\star)$ admits a (canonical up to equivalence) cyclic $\ainf$
    (minimal) model.
\end{cor}

\begin{rem}\label{fukayacyclic}
    In the case of compact symplectic manifolds and over $\K = $ a Novikov
    field containing $\R$, Fukaya \cite{Fukaya:2010aa} constructed a cyclic
    $\ainf$ model of the Floer cohomology algebra of a single compact
    Lagrangian, which will be extended to multiple objects by
    Abouzaid-Fukaya-Oh-Ohta-Ono \cite{afooo}. 
\end{rem}

\begin{rem}
In order to construct (chain level) 2d-TFTs on the Hochschild chain complexes
of categories, Kontsevich-Soibelman \cite{Kontsevich:2009ab} partly  show (on
the closed sector) that a proper Calabi-Yau structure can be
used instead of the (weakly equivalent in characteristic 0) cyclic $\ainf$
structures considered in \cite{Costello:2006vn}. One might similarly hope that,
for applications of cyclic $\ainf$ structures to disc counting/open
Gromov-Witten invariants developed in \cite{Fukaya:2011aa}, a proper
Calabi-Yau structure is in fact sufficient. See \cite{Cho:2012aa} for related
work.
\end{rem}

Turning to smooth Calabi-Yau structures, in \S \ref{subsec:smoothCY}, we will
establish the following existence result for smooth Calabi-Yau structures, which
applies to wrapped Fukaya categories of non-compact (Liouville) manifolds as
well as to Fukaya categories of compact manifolds:
\begin{thm}\label{thm:smoothCY}
    Under the technical hypotheses $(\star)$, suppose further that our
    symplectic manifold $M$ is {\em non-degenerate} in the sense of
    \cite{ganatra1_arxiv}, meaning that the map
    $[\oc]: \r{HH}_{*-n}(\f) \ra HF^*(M)$ hits the unit $1 \in HF^*(M)$.  Then,
    its (compact or wrapped) Fukaya category $\f$ possesses a canonical,
    geometrically defined {\em strong
    smooth Calabi-Yau structure}.
\end{thm}
Once more, the cyclic open-closed map gives an efficient description of this structure: it is the unique element $\r{HC}^-_{-n}(\f)$ mapping via $\widetilde{\oc}^-$ to the geometrically canonical lift $\widetilde{1}\in H^*(CF^*(M)^{hS^1})$ of the unit $1 \in CF^*(M)$ described in \S \ref{sec:interior}.\footnote{As shown in \cite{ganatra1_arxiv, GPS1:2015}, if $[\oc]$ hits $1$, then $[\oc]$ is an isomorphism, and hence by Corollary \ref{cor:cyclicOC}, $[\widetilde{\oc}^-]$ is too. Hence one can speak about the unique element.}

\begin{rem}
    In contrast to compact Fukaya categories or wrapped Fukaya categories of
    Liouville manifolds, the Fukaya categories of non-compact Lagrangians
    discussed in Remark \ref{partiallywrapped1} are
    typically not Calabi-Yau in either sense\footnote{One manifestation of this
is the failure of the target of the open-closed map to have a distinguished unit element,
as also discussed in Remark \ref{partiallywrapped1}.}, even if they are smooth or
proper categories; indeed they typically arise as homological mirrors to
perfect/coherent complexes on non-Calabi-Yau varieties. Instead, one might
expect such
categories to admit {\em pre-Calabi-Yau structures} in the sense of
\cite{Kontsevich:uq, Kontsevich:CY,Yeung:2018aa} (e.g., see
\cite{Seidel:2018ab} for a construction of related structures) or {\em relative
Calabi-Yau structures} in the sense of
\cite{brav_dyckerhoff_2019}.  
\end{rem}
The notion of a smooth Calabi-Yau structure, or sCY structure, will be studied
further in forthcoming joint work with R. Cohen \cite{CohenGanatra:2015}, and
used to compare the wrapped Fukaya category of a cotangent bundle and string
topology category of its zero section as {\em categories with sCY structures}
(in order to to deduce a comparison of topological field theories on both
sides).

\subsection{Noncommutative Hodge-de-Rham degeneration for smooth and proper Fukaya categories}
For a $C_{-*}(S^1)$ module $P$, there is a canonical Tor spectral
sequence converging to $H^*(P_{hS^1})$ with first page
$H^*(P) \otimes_{\K} H^*(\K_{hS^1}) \cong H^*(P)
\otimes_{\K} H_*(\mathbb{CP}^\infty)$.  When applied to  the Hochschild complex $P =
\r{CH}_*(\cc)$ of a (dg/$\ainf$) category $\cc$, the resulting spectral sequence, from
(many copies of) $\r{HH}_*(\cc)$ to $\mathrm{HC}^+(\cc)$ is called the {\em
Hochschild-to-cyclic} or {\em noncommutative Hodge-de-Rham (ncHDR) spectral
sequence.} The latter name comes from the fact that when $\cc = perf(X)$ is
perfect complexes on a complex variety $X$, this spectral sequence is equivalent (via
Hochschild-Kostant-Rosenberg (HKR) isomorphisms) to the usual Hodge-to-de-Rham
spectral sequence from Hodge cohomology to de Rham cohomology
\[
    H^*(X, \Omega^*_X) \Rightarrow H^*_{dR}(X),
\]
which degenerates (as we are in characteristic 0) whenever $X$ is smooth and proper.
Motivated by this, Kontsevich formulated the {\em noncommutative Hodge-de-Rham (ncHDR)
degeneration conjecture} \cite{Kontsevich:2009ab, Kontsevich:2008aa}: 
for any smooth and
proper category $\cc$ in characteristic 0, its ncHDR spectral sequence
degenerates. A general proof of this fact for $\Z$-graded categories was recently given by Kaledin, following earlier work establishing it in the
coconnective case \cites{Kaledin:2017aa, Kaledin:2008aa}.

Using the cyclic open-closed map, we can give a purely symplectic proof of the
nc-HdR degeneration property for those smooth and proper $\cc$ arising as
Fukaya categories, including in non-$\Z$-graded cases:
\begin{thm}\label{thm:hdrfuk}
    Let $\mathcal{A} \subset \mathcal{F}(M)$ be a smooth and proper
    subcategory of any Fukaya category of any compact symplectic manifold
    satisfying the technical assumptions $(\star)$, over any field $\K$
    (over which the Fukaya category and the cyclic open-closed map are
    defined). Then, the nc Hodge-de-Rham spectral sequence for $\mathcal{A}$
    degenerates.
\end{thm}

\begin{proof}
The nc Hodge-de-Rham spectral sequence for $\mc{A}$ degenerates at
page 1 if and only if $P=CH_*(\mathcal{A})$ is isomorphic (in the category of
$C_{-*}(S^1)$ modules) to a trivial
$C_{-*}(S^1)$-module, e.g., if the $C_{-*}(S^1)$ action is trivializable \cite{Dotsenko:2015aa}*{Thm. 2.1}.
For compact symplectic manifolds $M$, recall that $CF^*(M) \cong H^*(M)$
has a canonically trivial(izable) $C_{-*}(S^1)$ action (see Corollary
\ref{trivialconstantloops}; this comes from, e.g.,  the fact that we can choose
a $C^2$ small Hamiltonian to compute the complex,  all of the orbits  of which
are constant loops on which geometric rotation acts trivially. Or more
directly, we can modify the definition of $\widetilde{\oc}$ to give a map
directly to $H^*(M)$ with its trivial $C_{-*}(S^1)$ action, as described in \S
\ref{pseudocycles}). 

By earlier work \cite{GPS1:2015, Gautogen_arxiv}, whenever $\mathcal{A}$ is
smoooth, $\oc|_{\mathcal{A}}$ is an isomorphism from
$\r{HH}_{*-n}(\mathcal{A})$ onto a non-trivial summand $S$ of $HF^*(M) \cong
\r{QH}^*(M)$; the $C_{-*}(S^1)$ action on this summand is trivial too. Theorem
\ref{thm:mainresult1} shows that $\widetilde{\oc}|_{\mathcal{A}}$ induces an
isomorphism in the category of $C_{-*}(S^1)$ modules between
$\r{CH}_{*}(\mathcal{A})$ and $S[n]$ with its trivial action, so we are done.
\end{proof}

\begin{rem}
        Theorem \ref{thm:hdrfuk} holds for a field $\K$ of any
        characteristic over which the Fukaya category and relevant
        structures (satisfy $(\star)$ and) are defined, for any
        grading structure that can be defined on the given Fukaya category
        (for instance, it holds for the $\Z/2$-graded Fukaya category of a
        monotone symplectic manifold over a field of any characteristic).
        In contrast, for an arbitrary smooth and proper $\Z/2$-graded dg
        category in characteristic zero, the noncommutative Hodge-de-Rham
        degeneration is not yet established (though it is
        expected).  And it is not always true in finite characteristic.

        An incomplete explanation for the degeneration property holding for finite
        characteristic smooth and proper Fukaya categories 
        may be that the Fukaya category over a characteristic $p$ field $\K$
        (whenever Lagrangians are monotone or tautologically unobstructed at
        least) may always admit a lift to
        second Witt vectors $W_2(\K)$.\footnote{The author wishes to thank Mohammed
        Abouzaid for discussions regarding this point.}
\end{rem}

As is described in joint work (partly ongoing) with T. Perutz and N. Sheridan
\cite{GPS1:2015, GPS2:2015}, the cyclic open-closed map $\widetilde{\oc}^-$ can
further be shown to be a {\em morphism of semi-infinite Hodge structures}, a
key step (along with the above degeneration property and construction of
Calabi-Yau structure) in recovering Gromov-Witten invariants from the Fukaya
category and enumerative mirror predictions from homological mirror theorems.

\subsection{Outline of Paper}
In \S \ref{section:s1action}, we recall a
convenient model for the category of $\ainf$ modules over $C_{-*}(S^1)$ and
various equivariant homology functors from this category.
In \S\ref{section:open}, we review the (compact and wrapped) Fukaya category
along with $C_{-*}(S^1)$ action on its (and more generally, any cohomologically
unital $\ainf$ category's) {\em non-unital Hochschild chain complex} (a variant on usual cyclic bar complex that has usually appeared in the symplectic literature e.g., in \cite{Abouzaid:2010kx}).
In \S \ref{section:closed}, we recall the construction of the $\ainf$
$C_{-*}(S^1)$ module structure on the (Hamiltonian) Floer chain complex,
following \cite{Bourgeois:2012fk, Seidel:2010fk} (note that our technical setup
is slightly different, though equivalent).
Then we prove our main results in \S \ref{section:openclosed1}.
Some technical and conceptual variations on the construction of
$\widetilde{\oc}$ (including Remark \ref{rmk:thmvariations}) are discussed at
the end of this section, see \S \ref{sec:OCvariants}.  Finally, in \S
\ref{section:cystructures} we apply our results to construct proper and smooth
Calabi-Yau structures, proving Theorems \ref{mainthm3} and \ref{thm:smoothCY}.

\subsection{Conventions}\label{sec:conventions}
We work over a ground field $\K$ of arbitrary characteristic (though we note
that all of our geometric constructions are valid over an arbitrary ring, e.g.,
$\Z$). All chain complexes will be graded {\it cohomologically}, including
singular chains of any space, which hence have negative the homological grading
and are denoted by $C_{-*}(X)$. All gradings are either in $\Z$ or $\Z/2$ (in
the latter case, degrees of maps are implicitly mod 2).

\subsection*{Acknowledgements}
I'd like to thank Paul Seidel for a very helpful conversation and Nick Sheridan
for several helpful discussions about technical aspects of this paper such as
signs.  I would also like to thank Zihong Chen for
catching a typo in the construction of Floer data. Part of this work was revised during a visit at the Institut
Mittag-Leffler in 2015, who I'd like to thank for their hospitality. 
Finally, I'd like to thank an anonymous referee for a number of helpful suggestions, comments, and
corrections which improved the exposition of this article.

\section{Complexes with circle action}\label{section:s1action}
In this section, we review a convenient model for the category of $\ainf$
$C_{-*}(S^1)$ modules, for which the $\ainf$ $C_{-*}(S^1)$ action can be
described by a single hierarchy of maps satisfying equations. 
We also describe various equivariant homology complexes in this language in
terms of simple formulae.  This model appears elsewhere in the literature as
{\em $\infty$-mixed complexes} or {\em $S^1$-complexes} or {\em multicomplexes} (we will sometimes
adopt the second term); see e.g., \cite{Bourgeois:2012fk, Zhao:2019aa, Dotsenko:2015aa} (but
note that the first and third references use homological grading conventions).

\subsection{Definitions}\label{circleactionsubsection}
Let $C_{-*}(S^1)$ denote the dg algebra of chains on the circle with
coefficients in $\K$, graded cohomologically, with multiplication induced by the
Pontryagin product $S^1 \times S^1 \ra S^1$. This algebra is {\it formal}, or
quasi-isomorphic to its homology, an exterior algebra on one generator
$\Lambda$ of degree -1 with no differential. Henceforth, by abuse of notation
we take this exterior algebra as our working model for $C_{-*}(S^1)$
\begin{equation}
    C_{-*}(S^1):= \K[\Lambda]/\Lambda^2, \ \ \ \ |\Lambda| = -1,
\end{equation}
and use the terminology $C_{-*}^{sing}(S^1)$ to refer to usual singular chains
on $S^1$.  
\begin{defn}\label{s1strict}
    A {\em strict $S^1$-complex}, or a {\em chain complex with strict/dg $S^1$
    action}, is a unital differential graded module over
    $\K[\Lambda]/\Lambda^2$. 
\end{defn}
Let $(M,d)$ be a strict $S^1$-complex; by definition $(M,d)$ is a co-chain
complex (recall our conventions for complexes from \S \ref{sec:conventions})
and the unital dg $\K[\Lambda]/\Lambda^2$ module structure is equivalent to
the data of the single additional operation of multiplying by $\Lambda$
\begin{equation}
    \Delta = \Lambda \cdot -: M_* \ra M_{*-1},
\end{equation}
which must square to zero and anti-commute with $d$. In other words,
$(M,d,\Delta)$ is what is known as a {\it mixed complex}, see e.g.,
\cite{Burghelea:1986aa,
Kassel:1987aa, Loday:1992fk}.

We will need to work with the weaker notion of an $\ainf$ action, or rather an
$\ainf$ module structure over $C_{-*}(S^1) = \K[\Lambda]/\Lambda^2$.
Recall that a {\em (left) {\em $\ainf$ module $M$}} \cite{Keller:2006ab, Seidel:2008zr, Seidel:2008cr, ganatra1_arxiv} over the associative
graded algebra $A = \K[\Lambda]/\Lambda^2$ is a graded $\K$-module $M$
equipped with maps
\begin{equation}\label{ainfmodulemaps}
    \mu^{k|1}: A^{\otimes k} \otimes M \ra M,\ \ \ k \geq 0
\end{equation}
of degree $1-k$, satisfying the $\ainf$ module equations described in
\cite{Seidel:2008cr} or \cite{ganatra1_arxiv}*{(2.35)}.
Since $A = \K[\Lambda]/\Lambda^2$ is unital, we can work with modules that are
also {\em strictly unital} (see \cite{Seidel:2008cr}*{(2.6)}); this implies
that all multiplications by a sequence with at least one unit element is
completely specified,\footnote{More precisely $\mu^{1|1}(1,\mathbf{m}) =
\mathbf{m}$ and $\mu^{k|1}(\ldots, 1, \ldots, \mathbf{m})
= 0$ for $k > 1$.} 
and hence
the only non-trivial structure maps to define are the operators 
\begin{equation}\label{bvinfmaps}
        \delta_{k} := \mu^{k|1}_{M}(\underbrace{\Lambda, \ldots, \Lambda}_{k\textrm{ copies}}, -): M \ra M [1-2k],\ k \geq 0.
\end{equation}
The $\ainf$ module equations are equivalent to the following relations for \eqref{bvinfmaps}
for each $s \geq 0$,
\begin{equation}\label{inftycomplexequations}
    \sum_{i=0}^s \delta_i \delta_{s-i} = 0.
\end{equation}
We summarize the discussion so far with the following definition:
\begin{defn} \label{homotopycircleaction}
    An {\em $S^1$-complex}, or a {\em chain complex with a $\ainf$ $S^1$ action}, 
    is a stricly unital (left) $\ainf$ module $M$ over
    $\K[\Lambda]/\Lambda^2$. Equivalently, it is a graded $\K$-module $M$ equipped
    with operations $\{\delta_k: M \to M[1-2k]\}_{k \geq 0}$ satisfying, for
    each $s \geq 0$, the hierarchy of equations \eqref{inftycomplexequations}.
\end{defn}

\begin{rem}\label{spaceS1actionS1complex}
    If $X$ is a topological space with $S^1$ action, then 
    $C_{-*}(X)$ carries a dg $C_{-*}^{sing}(S^1)$ module structure, with module
    action induced by the action $S^1 \times X \to X$. Under the $\ainf$
    equivalence $C_{-*}^{sing}(S^1) \cong \K[\Lambda]/\Lambda^2$, it follows
    that $C_{-*}(X)$ carries an $\ainf$ (not necessarily dg)
    $\K[\Lambda]/\Lambda^2$ module structure, which can further be made strictly unital 
    (by \cite{Lefevre02}*{Thm. 3.3.1.2} or by
    passing to normalized chains). If one wishes, one can then appeal to abstract
    strictification results to produce a dg $\K[\Lambda]/\Lambda^2$ module
    which is quasi-isomorphic as $\ainf$ $\K[\Lambda]/\Lambda^2$ modules to
    $C_{-*}(X)$. More directly, it turns out \cite{CohenGanatra:2015} one can find an equivalent dg
    $\K[\Lambda]/\Lambda^2$ module by taking a suitable quotient of the
    normalized singular chain complex $C_{-*}(X)$ to form {\em unordered
    normalized singular chains} of $X$ (identifying simplices differing by
    permuting vertices and quotienting by those that are degenerate). 
\end{rem}

\begin{rem}
    There are multiple sign conventions for $\ainf$ modules over an $\ainf$
    algebra; the most common two conventions appear in e.g.,
    \cite{Seidel:2008cr}*{(2.6)} and \cite{Seidel:2008zr}*{(1j)} as well as
    many other places.
    These conventions are completely irrelevant for strictly
    unital $A = \K[\Lambda]/\Lambda^2$ modules, as
    the reduced degree of any element in $\bar{A} =
    \mathrm{span}_{\K}(\Lambda)$ is zero; hence the (Koszul) signs in various
    formulae are $+1$ in either convention.
\end{rem}
For $s=0$, \eqref{inftycomplexequations} says simply that the differential $d =
\delta_0$ squares to 0; for $s=1$, \eqref{inftycomplexequations} implies
$\delta:= \delta_1$ anti-commutes with $d$, and for $s=2$, $(\delta)^2 = -
(d\delta_2 + \delta_2 d)$, or that $\delta^2$ is chain-homotopic to zero, but
not strictly zero, as measured by the chain homotopy $\delta_2$.  

$S^1$-complexes, as strictly unital $\ainf$ modules over the
augmented algebra $A = \K[\Lambda]/\Lambda^2$, are the objects of a dg category which we will call 
\begin{equation}
    S^1\mod:= uA\mod 
\end{equation}
(compare \cite{Seidel:2008cr}*{p. 90, 94} for the
definition of this category where it is called $mod(A) = mod(A,\K)$) whose morphisms and compositions we now
recall.  Denote by $\e: A \ra \K$ the augmentation map, and $\bar{A} = \ker \e =
\mathrm{span}_{\K}(\Lambda)$ the
augmentation ideal. 
Let $M$ and $N$ be two strictly unital $\ainf$ $A$-modules. A
{\it unital pre-morphism of degree $k$} from $M$ to $N$ of degree $k$ is a collection of maps
$    F^{d|1}: \overline{A}^{\otimes d} \otimes M \ra N$,  $d \geq 0$,
of degree $k-d$, or equivalently since $\dim_{\K}(\overline{A}) = 1$
in degree -1, a collection of operators
\begin{equation}\label{simplifiedpremorphism}
    \begin{split}
        F &= \{F^d\}_{d \geq 0}\\
        F^d &: = F^{d|1}(\underbrace{\Lambda, \ldots, \Lambda}_{d\textrm{ copies}}, -): M \to N[k-2d].
\end{split}
\end{equation}
If $T (\bar{A}[1]) = \oplus_{d \geq 0} \bar{A}[1]^{\otimes d}$ denotes the tensor
algebra of $\bar{A}[1]$, then $F$ can be alternatively packaged into the data
of a single degree $k$ map $F := \oplus_{d \geq 0} F^d: T \bar{A}[1] \otimes M \to N$.
The space of pre-morphisms of each degree form the graded space of morphisms in
$S^1 \mod$, which we will denote by $\r{Rhom}_{S^1}(-,-)$:
\begin{equation}\label{extcomplex}
    \begin{split}
        \r{Rhom}_{S^1}(M,N) &:= \bigoplus_{k \in \Z} \r{Rhom}_{S^1}^k(M,N):= \bigoplus_{k \in \Z} \hom_{grVect}(T(\overline{A}[1]) \otimes M, N[k]) \\
    &= (\bigoplus_{k \in \Z} \hom_{grVect}( \oplus_{d \geq 0} M[2d], N[k])). 
\end{split}
\end{equation}
There is a differential $\partial$ on \eqref{extcomplex} described in
\cite{Seidel:2008cr}*{p. 90}; in terms of the simplified form of pre-morphisms
\eqref{simplifiedpremorphism}, one has
\begin{equation}
    (\partial F)^s = \sum_{i = 0}^s F^i \circ \delta^M_{s-i} - (-1)^{\deg(F)} \sum_{j=0}^s \delta^N_{s-j}\circ F^j.
\end{equation}
An {\em $\ainf$ $\K[\Lambda]/\Lambda^2$ module homomorphism}, or {\em
$S^1$-complex homomorphism} is a pre-morphism $F = \{F^d\}$ which is closed,
e.g., $\partial F = 0$. In particular, $F$ is an $\ainf$ module homomorphism if
the following equations are satisfied, for each $s$: 
\begin{equation}\label{s1homomorphismeqns}
    \sum_{i = 0}^s F^i \circ \delta^M_{s-i} = (-1)^{\deg(F)} \sum_{j=0}^s \delta^N_{s-j}\circ F^j.
\end{equation}
Note that the $s=0$ equation reads $F^0 \circ \delta_0^M = (-1)^{\deg(F)}
\delta_0^N \circ F^0$, so (if $\partial F = 0$) $F^0$
induces a cohomology level map $[F^0]: H^*(M) \ra H^{*+\deg(F)}(N)$. A module
homomorphism (or closed morphism) $F$ is said to be a {\it quasi-isomorphism}
if $[F^0]$ is an isomorphism on cohomology. A {\em strict} module homomorphism
$F$ is one for which $F^k = 0$ for $k > 0$.

\begin{rem}\label{nonunitalmorphism}
    There is an enlarged notion of a {\em non-unital} pre-morphism (used for
    modules which aren't necessarily strictly unital), which is a collection of
    maps $\{\hat{F}^d: A^{\otimes d} \otimes M \to N\}_d$ instead of
    $\{F^d: \overline{A}^{\otimes d} \otimes M \to N\}_d$. Any pre-morphism $F = \{F^d\}_d$ as we've
    defined extends to a non-unital pre-morphism $\hat{F} = \{\hat{F}^d\}$ by
    declaring $\hat{F}^d(\ldots, 1, \ldots, \mathbf{m}) = 0$. For strictly unital
    modules, the resulting inclusion from the complex of pre-morphisms to the
    complex of non-unital pre-morphisms is a quasi-isomorphism.
\end{rem}

\begin{rem}\label{derivedhom}
    When $M$ and $N$ are $dg$ modules, or strict $S^1$-complexes, $\mathrm{Rhom}_{S^1}(M,N)$ is a 
    {\it reduced bar model} of the chain complex of derived
    $\K[\Lambda]/\Lambda^2$ module homomorphisms, which is one of the reasons
    we've adopted the terminology
    ``$\mathrm{Rhom}$''. In the $\ainf$ setting, we recall that there is no
    sensible ``non-derived'' notion of a $\K[\Lambda]/\Lambda^2$ module map (compare \cite{Seidel:2008cr}).
\end{rem}

The composition in the category $S^1 \mod$
\begin{equation}
    \r{Rhom}_{S^1}(N,P) \otimes \r{Rhom}_{S^1}(M,N) \ra \r{Rhom}_{S^1}(M,P)
\end{equation}
is defined by 
\begin{equation}
    (G \circ F)^s = \sum_{j=0}^s G^{s-j} \circ F^j.
\end{equation}
\begin{rem}\label{extkk}
    If $M$ is any $S^1$-complex, then its endomorphisms
    $\r{Rhom}_{S^1}(M,M)$ equipped with composition, form a dg algebra.  As an
    example, consider $M = \K$, with trivial module structure (determined by
    the augmentation $\e: \K[\Lambda]/\Lambda^2 \ra \K$).  
    It is straightforward to compute that, as a dga
    \begin{equation}\label{kuidentification}
            \r{Rhom}_{S^1}(\K, \K) \cong \K[u], \ |u| = 2.
    \end{equation}
    (in terms of the definition of morphism spaces \eqref{extcomplex}, $u$
    corresponds to the unique morphism $G = \{G^d\}_{d \geq 0}$ of degree $+2$
    with $G^1 = \mathrm{id}$ and $G^s = 0$ for $s \neq 1$).
\end{rem}

In addition to taking the morphism spaces, 
one can define the (derived) {\em
tensor product} of $S^1$-complexes $N$ and $M$: using the isomorphism $A \cong
A^{op}$ coming from commutativity of $A = \K[\Lambda]/\Lambda^2$, first view
$N$ as a {\em right $\ainf$ $A$ module} (see \cite{Seidel:2008cr}*{p. 90, 94}
where the category of right $A$ modules are called $mod(\K, A)$, \cite{Seidel:2008zr}*{(1j)}, \cite{ganatra1_arxiv}*{\S 2})
and then take the usual (necessarily derived) tensor product of $N$ and $M$ over $A$ (see
\cite{Seidel:2008cr}*{p.  91} or \cite{ganatra1_arxiv}*{\S 2.5}).  The
resulting chain complex (which we will by abuse of notation indicate as the
derived tensor product over $S^1$) has underlying graded vector space 
\begin{equation}\label{tensorproduct}
    \begin{split}
        N \otimes_{S^1}^{\mathbb L} M &:= N\otimes_{A}^{\mathbb L} M := \bigoplus_{d \geq 0} N \otimes \overline{A}[1]^{\otimes d} \otimes M\\
        &= \bigoplus_{d\geq 0} (N \otimes_{\K} M)[2d]
    \end{split}
\end{equation}
(the degree $s$ part is $\bigoplus_{d \geq 0} \bigoplus_{s} N_t \otimes M_{s +
2d-t}$). Let us refer to an element $n \otimes m$ of the $d$th summand of this
complex by suggestive notation $n \otimes \underbrace{\Lambda \otimes \cdots
\otimes \Lambda}_{d\textrm{ times}} \otimes m$ as in the first line of
\eqref{tensorproduct}. With this notation, the differential on
\eqref{tensorproduct} acts as
\begin{equation}
    \partial(n \otimes\underbrace{\Lambda \otimes \cdots \otimes
    \Lambda}_{d} \otimes m)  = \sum_{i=0}^d \left( (-1)^{|m|} \delta_i^N  n\otimes \underbrace{\Lambda \otimes \cdots \otimes
    \Lambda}_{d-i} \otimes m +  n\otimes \underbrace{\Lambda \otimes \cdots \otimes
    \Lambda}_{d-i} \otimes\delta_i^M m\right).
\end{equation}
(here our sign convention follows \cite{ganatra1_arxiv}*{\S 2.5} rather than
\cite{Seidel:2008cr}, though the sign difference is minimal).

\begin{rem}
    Analogously to Remark \ref{derivedhom}, if $M$ and $N$ are unital dg modules over 
    $A = \K[\Lambda]/\Lambda^2$, the chain complex described above computes their
    derived tensor product, whose homology is
    $\mathrm{Tor}_{A}(M,N)$. While we have therefore opted
    for the notation $N \otimes_{A}^{\mathbb L} M$ (or rather the abbreviation $N \otimes_{S^1}^{\mathbb L} M$) 
    we note that the (derived) tensor product of $\ainf$ modules is often written in the $\ainf$
    literature without the superscript $\mbox{ }^{\mathbb L}$ as simply $N\otimes_{A} M$ (compare \cite{Seidel:2008cr}*{eq. (2.6)}).
\end{rem}
The pairing \eqref{tensorproduct} is suitably functorial with respect to
morphisms of the $S^1$-complexes involved, meaning that $-\otimes_{S^1} N$ and
$M\otimes_{S^1} -$ both induce dg functors from $S^1\mod$ to chain complexes
(compare \cite{Seidel:2008cr}*{p. 92}).  For instance, if $F = \{F^j\}: M_0 \to
M_1$ is a pre-morphism of $S^1$-complexes, then there are induced maps
\begin{align}
        F_\sharp: N \otimes_{S^1}^{\mathbb L} M_0 &\ra N \otimes_{S^1}^{\mathbb L} M_1\\
\nonumber        n \otimes\underbrace{\Lambda \otimes \cdots \otimes
        \Lambda}_{d} \otimes m  &\mapsto \sum_{j=0}^d n \otimes\underbrace{\Lambda \otimes \cdots \otimes
        \Lambda}_{d-j} \otimes F^j(m);\\
        F_\sharp: M_0 \otimes_{S^1}^{\mathbb L} N &\ra M_1 \otimes_{S^1}^{\mathbb L} N\\
\nonumber        m \otimes\underbrace{\Lambda \otimes \cdots \otimes
\Lambda}_{d} \otimes n  &\mapsto \sum_{j=0}^d (-1)^{\deg(F) \cdot |n|} F^j(m) \otimes\underbrace{\Lambda \otimes \cdots \otimes
        \Lambda}_{d-j} \otimes n;
\end{align}
which are chain maps if $\partial(F) = 0$.

Hom and tensor complexes of $S^1$-complexes, as in any category of $\ainf$
modules, satisfy the following strong homotopy invariance properties:
\begin{prop}[Homotopy invariance]\label{homotopyinvariance}
    If $F: M \ra M'$ is any quasi-isomorphism of $S^1$-complexes (meaning $\partial(F) = 0$ and $[F^0]: H^*(M) \stackrel{\cong}{\to} H^*(M')$ is an isomorphism),
    then composition with $F$ induces quasi-isomorphisms of hom and tensor
    complexes: 
\begin{equation}
    \begin{split}
        F \circ \cdot: \r{Rhom}_{S^1}(M', P) &\stackrel{\sim}{\ra} \r{Rhom}_{S^1}(M,P)\\
        \cdot \circ F: \r{Rhom}_{S^1}(P, M) &\stackrel{\sim}{\ra} \r{Rhom}_{S^1}(P,M')\\
        F_{\sharp}: N \otimes_{S^1}^{\mathbb L} M &\stackrel{\sim}{\ra} N \otimes_{S^1}^{\mathbb L} M'.\\
        F_{\sharp}: M \otimes_{S^1}^{\mathbb L} N &\stackrel{\sim}{\ra} M' \otimes_{S^1}^{\mathbb L} N.
\end{split}
\end{equation}
\end{prop}
The proof is a standard argument (though we do not know a specific reference): one exhibits acyclicity of the cone of each of the above maps by studying the spectral sequence with respect to the length filtration
(with respect to the number of $\overline{A}^{\otimes d}$ factors in the bar model of the complexes); the first page of the associated spectral sequence is the cone of the map associated to the derived homs/tensor products of the associated homology-level modules by the homology level map $[F^0]$, which is acyclic by hypothesis; hence the second page vanishes and the cone is acyclic (compare analogous arguments in \cite{Seidel:2008zr}*{Lemma 2.12} or \cite{ganatra1_arxiv}*{Prop.  2.2}).

Let $(P, \{\delta_i^P\})$ and $(Q, \{\delta_j^Q\}_j)$ be $S^1$-complexes, and
$f: P \to Q$ a chain map of some degree $\deg(f)$ (with respect to the
$\delta_0^P$ and $\delta_0^Q$ differentials). An {\em $S^1$-equivariant
enhancement of $f$} is a degree $\deg(f)$ homomorphism $\mathbf{F} =
\{\mathbf{F}^i\}_{i \geq 0}$ of $S^1$-complexes (e.g., a closed morphism, so
$\mathbf{F}$ satisfies \eqref{s1homomorphismeqns})
with $[\mathbf{F}^0] = [f]$. 
\begin{rem}
    Note that there are a series of obstructions to the existence of an
    $S^1$-equivariant enhancement of a given chain map $f$; for instance a first
    necessary condition is the vanishing of the cohomology class $[f] \circ [\delta_1^P] - [\delta_1^Q] \circ
    [f]$.
\end{rem}

Finally, we note that, just as the product of $S^1$ spaces $X \times Y$
possesses a diagonal action, the (linear) tensor product of $S^1$-complexes is
again an $S^1$-complex.  
\begin{lem}\label{diagonalaction}
    If 
    $(M,\delta_{eq}^M = \sum_{i=0}^{\infty} \delta_j^M u^j)$ 
    and $(N, \delta_{eq}^N = \sum_{i=0}^{\infty}
    \delta_i^N u^i)$ are
    $S^1$-complexes, then the graded vector space $M \otimes N$ is naturally an
    $S^1$-complex with $\delta_{eq}^{M \otimes N} = \sum_{i=0}^{\infty} \delta_k^{M
    \otimes N} u^k$, where
    \begin{equation}
        \delta_k^{M \otimes N}(\mathbf{m} \otimes \mathbf{n}):= 
            (-1)^{|\mathbf{n}|}\delta_k^M\mathbf{m} \otimes \mathbf{n} + \mathbf{m} \otimes \delta_k^N \mathbf{n} 
    \end{equation}
    We call the resulting $S^1$ action on $M \otimes N$ the {\em diagonal $S^1$-action}.
\end{lem}
\begin{proof}
    We compute 
    \begin{equation}
        \delta_j^{M\otimes N} \delta_k^{M \otimes N} (\mathbf{m} \otimes \mathbf{n}) = 
            \delta_j^M \delta_k^M\mathbf{m} \otimes \mathbf{n} + (-1)^{|\mathbf{n}| + 1}\delta_j^M \mathbf{m} \otimes \delta_k^N \mathbf{n}  + (-1)^{|\mathbf{n}|}\delta_k^M\mathbf{m} \otimes \delta_j^N \mathbf{n} + \mathbf{m} \otimes \delta_j^N \delta_k^N \mathbf{n}  \\
    \end{equation}
    Summing 
    over all $j+k = s$,  the middle two terms cancel in pairs
    and the sums of the leftmost terms (respectively rightmost) terms
    respectively vanish because $M$ (respectively $N$) is an $S^1$-complex.
\end{proof}

\begin{defn}\label{trivialaction}
Let $M := (M, d)$ be a chain complex over $\K$. The pullback of $M$ along the
(augmentation) map $\K[\Lambda]/\Lambda^2 \to \K$ is called the
{\em trivial $S^1$-complex}, or {\em chain complex with trivial $S^1$-action} associated to $M$, and denoted $\underline{M}^{triv}$.
Concretely, $\underline{M}^{triv} := (M, \delta_0 = d, \delta_k = 0\textrm{ for }k>0)$. 
\end{defn}

\subsection{Equivariant homology groups}\label{sec:equivariantgroups}
Let $M$ be an $S^1$-complex. Let $\K = \underline{\K}^{triv}$ denote the strict trivial rank-1 $S^1$-complex concentrated in degree 0.
\begin{defn}\label{defhomotopyorbit}
    The {\em homotopy orbit complex}
    of $M$ is the (derived) tensor product of $M$ with $\K$ over
    $C_{-*}(S^1)$:
    \begin{equation} \label{cyclichomologycomplex}
        M_{hS^1}:= \K \otimes^{\mathbb L}_{S^1} M.
    \end{equation}
\end{defn}
The (strict) morphism of $S^1$-complexes $\epsilon: \K[\Lambda]/\Lambda^2 \to \K$ (here $\K[\Lambda]/\Lambda^2$ comes equipped with structure maps $\delta_k = 0$ for $k \neq 1$ and $\delta_1 = \Lambda \cdot -$) induces by functoriality a
 chain map from $M$ to $M_{hS^1}$ called the {\em projection to homotopy
orbits}: 
\begin{equation}\label{projectionhomotopyorbits}
    pr: M \cong \K[\Lambda]/\Lambda^2 \otimes^{\mathbb L}_{S^1} M \to \K \otimes^{\mathbb L}_{S^1} M = M_{hS^1}
\end{equation}
\begin{rem}\label{chainlevelequivarianthomology}
    When $M = C_{-*}(X)$, with $S^1$-complex induced by a topological $S^1$
    action on $X$ as in Remark \ref{spaceS1actionS1complex}, 
    the complex \eqref{cyclichomologycomplex}
    computes the Borel equivariant homology of $X$, by the following reasoning: first, the $\ainf$ equivalence between $\K[\Lambda]/\Lambda^2$ and $C_{-*}^{sing}(S^1)$ induces an equivalence
    \[
        M_{hS^1} \simeq C_{-*}(pt) \otimes^{\mathbb L}_{C_{-*}^{sing}(S^1)} C_{-*}(X).
    \] 
    Next, one observes that $C_{-*}(ES^1) \to C_{-*}(pt)$ is a {\em free
    resolution} of dg $C_{-*}^{sing}(S^1)$ modules (where the
    $C_{-*}^{sing}(S^1)$ actions are induced by the $S^1$ actions
    on $ES^1$ and $pt$ respectively), as $S^1$ acts freely on $ES^1$. Hence, the derived tensor product becomes an ordinary tensor product 
    \[
        C_{-*}(pt) \otimes^{\mathbb L}_{C_{-*}^{sing}(S^1)} C_{-*}(X)   = C_{-*}(ES^1) \otimes_{C_{-*}^{sing} (S^1)} C_{-*}(X)  .
    \]
    Finally, it is a standard fact in algebraic topology (used in the
    construction of Eilenberg-Moore type spectral sequences, e.g.,
    \cite{McCleary:2001fk}*{Thm.  7.27}) that 
    \[
        C_{-*}(ES^1) \otimes_{C_{-*}^{sing} (S^1)} C_{-*}(X)   \simeq C_{-*}(ES^1 \times_{S^1} X) = C_{-*}(X_{hS^1}),
    \]
    which is the usual chain complex computing (Borel) equivariant homology.
    This gives some justification for the usage of the ${
    }_{hS^1}$ notation in Definition \ref{defhomotopyorbit}.
\end{rem}

\begin{defn}
    The {\em homotopy fixed point complex} of $M$ is
    the chain complex of morphisms from $\K$ to $M$ in the category of $S^1$-complexes:
    \begin{equation}\label{negativecyclic}
        M^{hS^1}:= \r{Rhom}_{S^1}(\K,M).
    \end{equation}
\end{defn}
The morphism of modules $\epsilon: \K[\Lambda]/\Lambda^2 \to \K$ induces a chain map $M^{hS^1} \to M$ called the {\em inclusion of homotopy fixed points}
\begin{equation}\label{inclusionhomotopyfixedpoints}
    \iota: M^{hS^1} = \r{Rhom}_{S^1}( \K, M) \to \r{Rhom}_{S^1}( \K[\Lambda]/\Lambda^2, M) \cong M.
\end{equation}

\begin{rem}
    To motivate the usage ``homotopy fixed points,'' note that in the
    topological category, the usual fixed points of a $G$ action can be
    described as $Maps_G(pt, X)$.  When $M = C_{-*}(X)$ for $X$ an $S^1$-space,
    there is a canonical map $C_{-*}(X^{hS^1}) \to (C_{-*}(X))^{hS^1}$. However
    (in contrast to the case of homotopy orbits discussed in Remark
    \ref{chainlevelequivarianthomology}), this map need not be an equivalence.
\end{rem}
Composition in the category $S^1 \mod$ induces a natural action of 
\begin{equation}
    \r{Rhom}_{S^1}(\K,\K) = \K[u]\ \ (|u| = 2) = H^*(BS^1)
\end{equation}
on the homotopy fixed point 
complex. There is a third important equivariant homology complex, called the
{\em periodic cyclic}, or {\em Tate} complex of $M$, defined as the
localization of $M^{hS^1}$ away from $u = 0$; 
\begin{equation}
    M^{Tate} := M^{hS^1} \otimes_{\K[u]} \K[u,u^{-1}].
\end{equation} 
The Tate construction sits in an exact sequence between the homotopy orbits and
fixed points.

\begin{rem}[Gysin sequences]\label{remark:gysin}
    It is straightforward from the viewpoint of $\ainf$ $C_{-*}(S^1)$ modules
    to explain the appearance of various Gysin and periodicity sequences. Take
    for instance the {\em Gysin exact triangle} 
    \[
        M \stackrel{pr}{\to} M_{hS^1} \to M_{hS^1}[2] \stackrel{[1]}{\to} 
    \]
    This is a manifestation of a canonical exact triangle of objects in
    $S^1\mod$: 
    \[
        \K[\Lambda]/\Lambda^2 \stackrel{\epsilon}{\to} \K \stackrel{u}{\ra} \K[2] \stackrel{[1]}\ra 
    \]
    (recall in Remark \ref{extkk} it was shown $\r{Rhom}_{S^1}(\K, \K) \cong \K[u]$),
    pushed forward by the functor $\r{Rhom}_{S^1}(\cdot, M)$.
    The other exact sequences arise similarly.
\end{rem}

As a special case of the general homotopy-invariance properties of $\ainf$
modules stated in Proposition \ref{homotopyinvariance}, we have:
\begin{cor}\label{S1homotopyinvariance}
If $F: M \ra N$ is a homomorphism of $S^1$-complexes (meaning a closed morphism), it induces
chain maps between equivariant theories
\begin{align}
    \label{inducedcyclic}F^{hS^1}: M^{hS^1} &\to N^{hS^1}\\ 
    F_{hS^1}: M_{hS^1} &\to N_{hS^1}\\
    \label{inducedtate}F^{Tate}: M^{Tate} &\to N^{Tate}
\end{align}
If $F$ is a quasi-isomorphism of $S^1$-complexes (meaning simply $[F^0]$ is a
homology isomorphism), then \eqref{inducedcyclic}-\eqref{inducedtate} are
quasi-isomorphisms of chain
complexes.  \qed
\end{cor}

Functoriality further tautologically implies that
\begin{prop}\label{functorialitysequences}
    If $F: M \ra N$ is a homomorphism of $S^1$-complexes, then the various
    induced maps \eqref{inducedcyclic} - \eqref{inducedtate} intertwine all of
    the long exact sequences for (equivariant homology groups of) $M$ with
    those for $N$.  \qed
\end{prop}

\subsection{$u$-linear models for $S^1$-complexes}\label{sec:ulinear}

It is convenient to package the data described in the previous two sections
into ``$u$-linear generating functions'', in the following way: Let $u$ be a
formal variable of degree $+2$. Let us use the abuse of notation
\[
    M[ [u ] ] := M \widehat{\otimes}_{\K} \K[u]
\]
for the $u$-adically completed tensor product in the category of graded vector
spaces; in other words $M [ [ u ] ] := \oplus_k M[ [ u] ]_k$, where $M[ [ u]
]_k = \{\sum_{i=0}^{\infty} m_i u^i\ |\ m_i \in M_{k-2i} \}$.
Then, we frequently write an $S^1$-complex $(M, \{\delta_k\}_{k \geq 0})$ as a
$\K$-module $M$ equipped with a map
\begin{equation}
    \delta_{eq}^{(M)} = \sum_{i=0}^{\infty} \delta_i^M u^i:  M \to M[ [ u] ]
\end{equation}
of total degree 1, satisfying $\delta_{eq}^2 = 0$ (where we are implicitly
conflating $\delta_{eq}$ with its $u$-linear extension to a map $M[ [ u ] ] \to
M [ [ u ] ]$ in order to $u$-linearly compose and obtain a map $M \to M[ [ u] ]$).

Pre-morphisms from $M$ to $N$ of degree $k$ can similarly be recast as maps
$F_{eq} = \sum_{i=0}^{\infty} F_i u^i: M \to N[ [ u] ] $ of pure degree $k$ (so
each $F_i$ has degree $k-2i$). The differential on pre-morphisms can be
described $u$-linearly as 
\begin{equation}
    \partial(F_{eq}) = F_{eq} \circ \delta_{eq}^M - \delta_{eq}^N \circ F_{eq},
\end{equation}
and composition is simply the $u$-linear composition $G_{eq} \circ F_{eq}$
(again, one implicitly $u$-linearly extends $G_{eq}$ and then $u$-linearly
composes); explicitly $(\sum_{i \geq 0} G_i u^i) \circ (\sum_{j \geq 0} F_j
u^j) = \sum_{k \geq 0} (\sum_{i+j=k} G^i \circ F^j) u^k$.

With respect to this packaging, the formulae for various equivariant homology
chain complexes can be given the following more readable form:
\begin{align}
    \label{htopyorbitsU} M_{hS^1} = (M ( ( u )) / u M [ [ u ] ],\  \delta_{eq})\\
M^{hS^1} = (M [ [ u ] ],\   \delta_{eq})\\
M^{Tate} = (M ( (u ) ), \delta_{eq})
\end{align}
where again, we use the abuse of notation $M( ( u) ) = M[ [ u] ] \otimes_{\K
[u]} \K [u, u^{-1}]$ (on the other hand, note that \eqref{htopyorbitsU} is {\em
not} completed). As before, any homomorphism (that is, closed morphism) of
$S^1$-complexes $F_{eq} = \sum_{i=0}^{\infty} F^i u^i $ induces a
$\K[u]$-linear chain map between homotopy-fixed point complexes by $u$-linearly
extended composition, and hence, by tensoring over $\K[u]$ with $\K ( ( u )) /
u \K [ [ u ] ]$ or $\K( ( u ))$, chain maps between homotopy orbit and Tate
complex constructions. With respect to these explicit complexes, the projection to homotopy orbits
\eqref{projectionhomotopyorbits} and inclusion of fixed points
\eqref{inclusionhomotopyfixedpoints} chain maps have simple explicit descriptions as follows:
\begin{equation} \label{projectionhomotopyorbits_explicit}
    \begin{split}
        pr: M &\to M_{hS^1} \\
        \alpha &\mapsto \alpha \cdot u^0;
    \end{split}
\end{equation}

\begin{equation} \label{inclusionhomotopyfixedpoints_explicit}
    \begin{split}
        \iota: M^{hS^1} &\to M \\
        \sum_{i=0}^\infty \alpha_i u^i &\mapsto \alpha_0.
    \end{split}
\end{equation}

\begin{rem}
    This $u$-linear lossless packaging of the data describing an $S^1$-complex
    is a manifestation of {\em Koszul duality}; in the case of $A =
    \K[\Lambda]/\Lambda^2$, it posits that there is a fully faithful embedding,
    $\r{Rhom}(\K, -) = (-)^{hS^1}$ from $A$-modules into $B:=
    \r{Rhom}_{A}(\K,\K) = \K[u]$ modules.
\end{rem}
From the $u$-linear point of view, it is easier to observe that the exact
triangle of $\K[u]$ modules $\K[[u]] \to \K((u)) \to \K ( ( u ) ) / \K [ [ u]
] = u^{-1} (\K ( ( u ) ) / u \K [ [ u] ])$ induces a (functorial in $M$) exact triangle between equivariant homology
chain complexes $M^{hS^1} \to M^{Tate} \to M_{hS^1}[2] \stackrel{[1]}{\to}$.

\section{Circle action on the open sector}\label{section:open}

\subsection{The usual and non-unital Hochschild chain complex}
\label{subsec:hochschild}
Recall that an $\ainf$ category over $\K$, $\cc$ is specified by the following
data: 
\begin{itemize}
\item a collection of objects $\ob \mc{C}$

\item for each pair of objects $X,X'$, a graded vector space over $\K$, $\hom_\cc (X,X')$

\item for any set of $d+1$ objects $X_0, \ldots, X_d$, higher multilinear (over $\K$) composition maps
    \begin{equation}\label{compositionmaps}
        \mu^d: \hom_\cc (X_{d-1},X_d) \times \cdots \times \hom_\cc (X_{0},X_1) \ra \hom_\cc (X_0,X_d)
    \end{equation}
    (sometimes equivalently viewed as a map from the tensor product)
    of degree $2-d$, satisfying the (quadratic) $\ainf$ relations, for each $k>0$:
    \begin{equation} 
    \label{ainfeq} \sum_{i,l} (-1)^{\maltese_i} \mu^{k-l+1}_\cc(x_k, \ldots, x_{i+l+1}, \mu^{l}_\cc(x_{i+l}, \ldots, x_{i+1}), x_i, \ldots, x_1) = 0.
\end{equation}
with sign
\begin{equation}\label{ainfsign}
    \maltese_i := ||x_1|| + \cdots + ||x_i||.
\end{equation}
where $|x|$ denotes degree and $||x||:= |x| - 1$ denotes reduced degree.
\end{itemize}
The first two equations of \eqref{ainfeq} imply that
$\mu^1$ is a differential, and the cohomology level maps $[\mu^2]$ are a
genuine composition for the (non-unital) category $H^*(\cc)$ with the same
objects and morphisms
\begin{equation}\label{cohomologymorphisms}
    \mathrm{Hom}_{H^*(\cc)}(X,Y) := H^*(\hom_{\cc}(X,Y), \mu^1).
\end{equation}
We say that $\cc$ is {\em cohomologically unital} if there exist
cohomology-level identity morphisms $[e_X] \in \r{Hom}_{H^*(\cc)}(X,X)$ for
each object $X$, making $H^*(\cc)$ into a genuine category. We say that $\cc$
is {\em strictly unital} 
if there exist elements ${e}_X^+ \in \hom_\cc(X,X)$, for every object $X$, satisfying
\begin{equation}\label{eq:strictunits}
    \begin{split}
       \mu^1(e_X^+) &= 0 \\
       (-1)^{|y|}\mu^2(e_{X_1}^+,y) &= y = \mu^2(y,e_{X_0}^+)\textrm{ for any }y \in \hom_{\cc}(X_0, X_1) \\
       \mu^d(\cdots, e_X^+, \cdots) &= 0 \textrm{ for } d > 2.
    \end{split}
\end{equation}
(we call such elements the chain-level, or strict, identity elements).

The {\it Hochschild chain complex}, or {\it cyclic bar complex} of an $\ainf$ category $\cc$
is the direct sum of all cyclically composable sequences of morphism spaces in
$\cc$:
\begin{equation}\label{cyclichochschild}
    \r{CH}_*(\cc):= \bigoplus_{k\geq 0, X_{i_0}, \ldots, X_{i_k} \in \ob
\cc} \hom_\cc(X_{i_k},X_{i_0}) \otimes \hom_\cc(X_{i_k-1},X_{i_k}) \otimes \cdots
\otimes \hom_\cc(X_{i_0},X_{i_1}),
\end{equation} 
The (cyclic bar) differential $b$ acts on Hochschild chains by summing over ways to
cyclically collapse elements by any of the $\ainf$ structure maps:
\begin{equation}\label{cyclicbardifferential}
    \begin{split}
        b (\mathbf{x}_d \otimes &x_{d-1} \otimes \cdots \otimes x_1) =\\
        &\sum (-1)^{\#_{k}^d}\mu^{d-i}(x_{k}, \ldots, x_1, \mathbf{x}_d, x_{d-1}, \ldots, x_{k+i+1}) \otimes x_{k+i} \otimes \cdots \otimes  x_{k+1} \\
        &+ \sum (-1)^{\maltese_{1}^{s}}\mathbf{x}_d \otimes \cdots \otimes \mu^j(x_{s+j+1}, \ldots, x_{s+1}) \otimes x_{s} \otimes \cdots \otimes x_1.
\end{split}
\end{equation}
with signs
\begin{align}
    \maltese_{i}^{k} &:= \sum_{j=i}^k ||x_i||\\
    \#_k^d &:= \maltese_1^k \cdot ( 1 + \maltese_{k+1}^d) + \maltese_{k+1}^{d-1} + 1
\end{align}
In this complex, Hochschild chains are (cohomologically) graded as follows:
\begin{equation} \label{hochschildchaingrading}
    \deg(\mathbf{x}_d \otimes x_{d-1} \otimes \cdots \otimes x_1) := \deg(\mathbf{x}_d) + \sum_{i=1}^{d-1} \deg(x_i) - d + 1 = |\mathbf{x}_d| + \sum_{i=1}^{d-1} ||x_i||.
\end{equation}
\begin{rem}
    Frequently the notation $\r{CH}_*(\cc,\cc)$ is used for
    \eqref{cyclichochschild} to emphasize that Hochschild homology is taken
    here with {\em diagonal coefficients}, rather than coefficients in another
    bimodule.  
\end{rem}

If $\cc$ is a strictly unital $\ainf$ category, then the chain complex
\eqref{cyclichochschild} carries a strict $S^1$ action $B: \r{CH}_*(\cc) \to
\r{CH}_{*-1}(\cc)$, involving summing over ways to cyclically permute chains
and insert identity morphisms; see Remark \ref{strictlyunitalB} below. 
However, there is a quasi-isomorphic non-unital Hochschild complex of
$\cc$, which always carries a strict $S^1$-action (even if $\cc$ is not strictly
unital), which we will now describe. 

As a graded vector space, the {\em non-unital Hochschild complex} consists of
two copies of the cyclic bar complex, the second copy shifted down in
grading by 1: 
\begin{equation}\label{nonunitalcomplex}
    \r{CH}_*^{nu}(\cc):= \r{CH}_*(\cc) \oplus \r{CH}_*(\cc)[1]
\end{equation}
With respect to the decomposition \eqref{nonunitalcomplex}, we sometimes refer
to elements as $\sigma := (\check{\alpha},\hat{\beta})$
with the notation $\check{\alpha}$ or $\hat{\beta}$ indicating that
a given element $\alpha$ and $\beta$ belong to the left or right factor
respectively. Similarly, we refer to the left and right factors as the {\it
check factor} and the {\it hat factor} respectively.

Let $b'$ denote a version of the differential \eqref{cyclicbardifferential}
omitting the ``wrap-around terms" (often simply called the {\it bar
differential}):
\begin{equation}
    \begin{split}
        b' (x_d \otimes &x_{d-1} \otimes \cdots \otimes x_1) =\\
        & \sum (-1)^{\maltese_{1}^{s}}x_d \otimes \cdots  \cdots \otimes x_{s+j+1} \otimes
        \mu^j(x_{s+j}, \ldots,  x_{s+1}) \otimes x_{s} \otimes \cdots \otimes x_1\\
      & + \sum (-1)^{\maltese_1^{d-j}} \mu^j(x_d, x_{d-1}, \ldots, x_{d-j+1})\otimes x_{d-j} \otimes \cdots \otimes x_1.
    \end{split}
\end{equation}
For an element $\hat{\beta} = x_d \otimes \cdots \otimes x_1$ in the hat
(right) factor of \eqref{nonunitalcomplex}, define an element $d_{\wedge \vee}
(\hat{\beta})$ in the check (left) factor of \eqref{nonunitalcomplex} as
follows:
\begin{equation}
    \begin{split}
    d_{\wedge \vee} (\hat{\beta}) &:= (-1)^{\maltese_2^d + ||x_1|| \cdot \maltese_2^d + 1} x_1 \otimes x_d \otimes \cdots \otimes x_2 + (-1)^{\maltese_1^{d-1}} x_d \otimes \cdots \otimes x_1.
\end{split}
\end{equation}
In this language, the differential on the non-unital Hochschild complex can be written as:
\begin{equation}
    b^{nu}: (\check{\alpha},\hat{\beta}) \mapsto (b(\check{\alpha}) +  d_{\wedge \vee} (\hat{\beta}), b'(\hat{\beta}))
\end{equation}
or equivalently can be expressed via the matrix
\begin{equation}
    b^{nu} = \left( \begin{array}{cc} b & d_{\wedge \vee} \\ 0 & b' \end{array} \right) .
\end{equation}
The left factor $\r{CH}_*(\cc) \hookrightarrow \r{CH}_*^{nu}(\cc)$ is by
definition a subcomplex. Moreoever, since the quotient complex is the standard
$\ainf$ bar complex with differential $b'$, which is acyclic for cohomologically
unital $\cc$ (by a standard length filtration spectral sequence argument,
compare \cite{Seidel:2008zr}*{Lemma 2.12} or \cite{ganatra1_arxiv}*{Prop.
2.2}), it follows that: 
\begin{lem}\label{inclusionquasi} The inclusion map
    $\iota: \r{CH}_*(\cc) \hookrightarrow \r{CH}_*^{nu}(\cc)$ 
    is a quasi-isomorphism (when $\cc$ is cohomologically unital).\qed
\end{lem}

\begin{rem} \label{rem:unitinsertion}
    The non-unital Hochschild complex of $\cc$ can be conceptually explained in
    terms of cyclic bar complexes as follows (compare \cite{Loday:1992fk}*{\S
    1.4.1} and \cite{Sheridan:2015ab}*{\S 3.5}): first, {\em augment} the category $\cc$
    by adjoining strict units; meaning consider
    the $\ainf$ category $\cc^+$ with $\ob\ \cc^+ = \ob
    \cc$,
\begin{equation}
    \hom_{\cc^+}(X,Y) = \begin{cases} \hom_{\cc}(X,Y) & X \neq Y \\
                        \hom_{\cc}(X,X) \oplus \K \langle e^+_X \rangle & X = Y
                \end{cases}
    \end{equation}
whose $\ainf$ structure maps completely determined by the fact that $\cc$
is an $\ainf$ subcategory, 
and the elements
$e^+_X$ elements act as strict units in the sense of \eqref{eq:strictunits}.
Next, consider the {\em normalized (or reduced) Hochschild complex} of the
strictly unital category $\cc^+$, $\r{CH}_*^{red}(\cc^+)$, by definition the
quotient of $\r{CH}_*(\cc^+)$ by the acyclic subcomplex consisting of $e^+$
terms in any position but the first. 
Now, take the further quotient of $\r{CH}_*^{red}(\cc^+)$ by the subcomplex of
length 1 Hochschild chains of the form $e^{+}_X$ for some $X$.
The resulting complex, denoted $\widetilde{\r{CH}_*}(\cc^+)$, 
can be identified as a chain complex with
$\r{CH}_*^{nu}(\cc)$ via the map
\begin{equation}\label{eq:fbijection}
    \begin{split}
        f: \widetilde{\r{CH}_*}(\cc^+) &\stackrel{\cong}{\to} \r{CH}_*^{nu}(\cc)\\
    f(y_k \otimes \cdots \otimes y_1) &= \begin{cases}
        (0, y_{k-1} \otimes \cdots \otimes y_1) & y_k = e^+_X \textrm{ for some }X \\
        (y_k \otimes \cdots \otimes y_1, 0) & \textrm{otherwise}
    \end{cases}
\end{split}
\end{equation}
In particular, the differential in $\r{CH}^{nu}(\cc)$ on a Hochschild chain
$\hat{\beta}$ in the right factor (of the decomposition \eqref{nonunitalcomplex}) agrees with the (usual cyclic bar) Hochschild differential applied to $e^+_X
\otimes \beta$ under the correspondence $f$.
\end{rem}

\subsection{Circle action on the Hochschild complex}\label{subsec:cyclic}
The $S^1$ action on the non-unital (or usual) Hochschild complex is built out
of several intermediate operations.
First, let $t: \r{CH}_*(\cc) \to \r{CH}_*(\cc)$ denote the (signed) cyclic
permutation operator on the cyclic bar complex generating the $\Z/k\Z$ cyclic
action on the length $k$ expressions
\begin{equation}\label{toperator}
    t: x_k \otimes \cdots \otimes x_1 \mapsto (-1)^{||x_1|| \cdot \maltese_2^k + ||x_1|| + ||x_k||} x_1 \otimes x_k \otimes \cdots \otimes x_{2},
\end{equation}
(this is not a chain map). 
Let $N$ denote the {\it norm} of this operation; that is the sum of all powers
of $t$ (this depends on $k$, the length of a given Hochschild chain):
\begin{equation}\label{Noperator}
    \begin{split}
        N: x_k \otimes \cdots \otimes x_1 = \sigma \mapsto  (1 + t + t^2 + \cdots + t^{k-1}) \sigma;
\end{split}
\end{equation}
Let $s^{nu}: \r{CH}^{nu}_*(\cc) \to \r{CH}^{nu}_{*-1}(\cc)$ be the linear map
which sends check chains to the corresponding hat chains, and hat chains to
zero:
\begin{equation}\label{snu}
    s^{nu}(x_d \otimes  \cdots \otimes x_1, y_t \otimes \cdots \otimes y_1) := (-1)^{\maltese_1^d + ||x_d|| + 1}(0,  x_d \otimes \cdots \otimes x_1 ).
\end{equation}
(again, not a chain map), and
finally define
\begin{equation}
    \label{eq:Bnudef}
    \begin{split}
        B^{nu}:  \r{CH}^{nu}_*(\cc) &\to \r{CH}^{nu}_{*-1}(\cc)\\
        B^{nu}(x_k \otimes \cdots \otimes x_1,y_l \otimes \cdots \otimes y_1) &:= \sum_{i=1}^k (-1)^{\maltese_1^i \maltese_{i+1}^k + ||x_k|| + \maltese_1^k + 1}(0,  x_i \otimes \cdots \otimes x_1 \otimes x_k \otimes \cdots \otimes x_{i+1}))\\
        &= s^{nu}  (N(x_k \otimes \cdots \otimes x_1), y_l \otimes \cdots \otimes y_1)\\
        &= \sum_{i=0}^{k-1} s^{nu}  (t^{i}(x_k \otimes \cdots \otimes x_1), y_l \otimes \cdots \otimes y_1)
\end{split}
\end{equation}

\begin{lem}
    $(B^{nu})^2 = 0$ and $b^{nu} B^{nu} + B^{nu}b^{nu} = 0$. That is,
    $\r{CH}_*^{nu}(\cc)$ is a strict $S^1$-complex, with the action of
    $\Lambda = [S^1]$ given by $B^{nu}$.\qed
\end{lem}

Let $b_{eq} = b^{nu} + u B^{nu}$ be the strict $S^1$-complex structure on the
non-unital Hochschild complex $\r{CH}_*^{nu}(\cc)$, $u$-linearly packaged
as in \S \ref{sec:ulinear}. Using this, we can define cyclic homology groups:
\begin{defn}
    The {\em (positive) cyclic} chain complex, the {\em negative cyclic} chain
    complex, and the {\em periodic cyclic} chain complexes of $\cc$ are the
    {\em homotopy orbit complex},
    {\em homotopy fixed point complex}, and {\em Tate constructions} of the
    $S^1$-complex
    $(\r{CH}_*^{nu}(\cc), b_{eq})$ respectively. That is:
    \begin{align}
        \r{CC}^+_*(\cc) &:= (\r{CH}_*^{nu}(\cc))_{hS^1} = (\r{CH}_*^{nu}(\cc) \otimes_\K \K ( ( u )) / u \K [ [ u ] ], b_{eq}) \\
        \r{CC}^-_*(\cc) &:= (\r{CH}_*^{nu}(\cc))^{hS^1} = (\r{CH}_*^{nu}(\cc) \widehat{\otimes}_\K \K[ [ u ] ], b_{eq}) \\
        \r{CC}^\infty_*(\cc) &:= (\r{CH}_*^{nu}(\cc))^{Tate} = (\r{CH}_*^{nu}(\cc) \widehat{\otimes}_\K \K(( u ) ), b_{eq})
    \end{align}
    with grading induced by setting $|u| = +2$, and where (as in \S
    \ref{sec:ulinear}), $\widehat{\otimes}$ refers to the $u$-adically
    completed tensor product in the category of graded vector spaces. The
    cohomologies of these complexes, denoted $\r{HC}^{+/-/\infty}_*(\cc)$, are
    called the {\em (positive), negative,} and {\em periodic cyclic homologies}
    of $\cc$ respectively.
\end{defn}

The $C_{-*}(S^1)$ module structure on $\r{CH}_*^{nu}(\cc)$ is suitably functorial in the following
sense: Let $\mathbf{F}: \cc \ra \cc'$ be an $\ainf$ functor, 
There is an induced chain map on non-unital Hochschild
complexes
\begin{equation}\label{inducedmapnonunital}
    \begin{split}
        \mathbf{F}_{\sharp}^{nu}: \r{CH}_*^{nu}(\cc) &\ra \r{CH}_*^{nu}(\cc',\cc')\\
    (x,y) &\mapsto (\mathbf{F}_{\sharp}(x), \mathbf{F}_{\sharp}'(y))
\end{split}
\end{equation}
where 
\begin{align}
    \mathbf{F}_{\sharp}' (x_k \otimes \cdots \otimes x_0) &:= \sum_{i_1, \ldots, i_s} \mathbf{F}^{i_1}(x_k\cdots) \otimes \cdots \otimes \mathbf{F}^{i_s}(\cdots x_0)\\
    \mathbf{F}_{\sharp} (x_k \otimes \cdots \otimes x_0) &:= \sum_{i_1, \ldots, i_s, j} \mathbf{F}^{j + 1 + i_1}(x_j, \ldots, x_0, x_k, \ldots, x_{k-i_1+1}) \otimes \mathbf{F}^{i_2}(\ldots)\otimes \cdots \otimes \mathbf{F}^{i_s}(x_{j+i_s}\ldots, x_{j+1}).
\end{align}
which is an isomorphism on homology if $\mathbf{F}$ is a quasi-isomorphism
(indeed, even a Morita equivalence). This functoriality preserves $S^1$ structures:
\begin{prop}\label{functorialityB}
    $\mathbf{F}_{\sharp}^{nu}$ gives a strict morphism of strict $S^1$-complexes, 
    meaning $\mathbf{F}_{\sharp}^{nu} \circ b^{nu} = b^{nu} \circ \mathbf{F}_{\sharp}^{nu}$ and $\mathbf{F}_{\sharp}^{nu} \circ B^{nu} = B^{nu} \circ \mathbf{F}_{\sharp}^{nu}$.  In other words, the pre-morphism of $\ainf$ $\K[\Lambda]/\Lambda^2$ modules defined as 
    \begin{equation}
        \mathbf{F}_{*}^{d}(\underbrace{\Lambda, \ldots, \Lambda}_d, \sigma) := \begin{cases} \mathbf{F}_{\sharp}^{nu} (\sigma) & d =0 \\
            0 & d \geq 1
        \end{cases}
    \end{equation}
    is closed, i.e., an $\ainf$ module homomorphism.
\end{prop}
\begin{proof}[Sketch]
    It is well known that $\mathbf{F}_{\sharp}^{nu}$ is a chain map, so it
    suffices to verify that $\mathbf{F}_{\sharp}^{nu} \circ B^{nu}  = B^{nu}
    \circ \mathbf{F}_{\sharp}^{nu}$, or in terms of \eqref{inducedmapnonunital}
    \begin{equation}
        \mathbf{F}_{\sharp}' \circ s^{nu} N = s^{nu} N \circ \mathbf{F}_{\sharp}.
    \end{equation}
   We leave this an exercise, noting that applying either side to  
   a Hochschild chain $x_k \otimes \cdots \otimes x_1$, the sums match
   identically.  
\end{proof}

\begin{rem}\label{strictlyunitalB}
    If $\cc$ is strictly unital, one can also define an operator $B: \r{CH}_*(\cc)
    \to \r{CH}_{*-1}(\cc)$ on the usual cyclic bar complex by 
    \[
        B = (1-t)s N
    \]
    where up to a sign $s$ denotes the operation of inserting, at the beginning
    of a chain, the unique strict unit $e_X^+$ preserving cyclic composability:
    \begin{equation}\label{degeneracyoperator}
        \begin{split}
            s: x_k \otimes \cdots \otimes x_1 \mapsto (-1)^{||x_k|| + \maltese_1^k + 1} e_{X_{i_k}}^+ \otimes x_k \otimes \cdots \otimes x_1,\\
            \textrm{where }x_k \in \hom_\cc(X_{i_k},X_{i_0}).
    \end{split}
    \end{equation}
    It can be shown that $B^2 = 0$ and $Bb + bB = 0$, $\r{CH}_*(\cc)$ is a
    strict $S^1$-complex; moreoever that the quasi-isomorphism $\r{CH}_*(\cc)
    \cong \r{CH}_*^{nu}(\cc)$ is one of $S^1$-complexes.  In fact, $B$ descends
    to the {\em reduced Hochschild complex} $\r{CH}_*^{red}(\cc)$ described in
    Remark \ref{rem:unitinsertion} where it takes the simpler form
    \[
        B^{red} = sN
    \]
    (as applying $t s N$ results in a Hochschild chain with a strict unit not
    in the first position, which becomes zero in $\r{CH}_*^{red}(\cc)$). 
    If $\cc$ was not necessarily strictly unital, following Remark
    \ref{rem:unitinsertion} one can still consider the quotient of the reduced
    Hochschild complex of the augmented category $\cc^+$, which we called
    $\widetilde{\r{CH}}_*(\cc^+)$. The discussion here equips this complex with
    an $S^1$ action $\widetilde{B}^{red}$. Under the bijection $f$
    \eqref{eq:fbijection}, $\widetilde{B}^{red}$ is sent to $B^{nu}$ and $s$ is
    sent to $s^{nu}$.
\end{rem}

\begin{rem}
    Continuing Remark \ref{rem:unitinsertion}, suppose we have constructed $\r{CH}^{nu}_*(\cc)$
    as 
    $\widetilde{\r{CH}}_*(\cc):= \r{CH}^{red}_*(\cc^+)/ \oplus_X \K \langle e_X^+ \rangle$, 
    the quotient of the reduced Hochschild complex of
     the augmented category $\cc^+$. Given any $\mb{F}$ as above, extend
     $\mb{F}$ to an augmented
    functor $\mb{F}^+$ by mandating that
    \begin{equation}
        \begin{split}
        (\mb{F}^+)^1(e^+_X) &= e^+_{\mb{F}X}\\
        ({\mb{F}^+})^d(\ldots, e^+_X, \ldots )  &= 0.
    \end{split}
    \end{equation}
    It is easy to see that $(\mb{F}^+)_*: \r{CH}_*(\cc^+) \to \r{CH}_*((\cc')^+)$ descends to a map $\widetilde{\mb{F}}: \widetilde{\r{CH}}_*(\cc^+) \to \widetilde{\r{CH}}_*((\cc')^+)$. Under the bijection \eqref{eq:fbijection}, this precisely corresponds to 
     $\mb{F}_{\sharp}^{nu}$ described above. 
    In particular, the fact that
    strictly unital functors induce strict
    $S^1$-morphisms between (usual) Hochschild complexes immediately implies Proposition \ref{functorialityB}.
\end{rem}

\begin{rem}\label{rem:otheroptionsNUcomplex}
    There are options besides the non-unital Hochschild complex for seeing the
    $C_{-*}(S^1)$ action on a Hochschild complex of the
    Fukaya category.  For instance one could:
\begin{enumerate}
    \item perform a strictly unital replacement (via homological algebra as in
        \cite{Seidel:2008zr}*{\S 2} \cite{Lefevre02}*{Thm. 3.2.1.1}), and work
        with the Hochschild complex of the replacement. However, this doesn't
        retain a relationship between the $\ainf$ operations and geometric
        structure, and hence is difficult to use with open-closed maps. 

    \item Geometrically construct a strictly unital structure on the Fukaya
        category via constructing {\em homotopy units} \cite{Fukaya:2009qf},
        which roughly involves building a series of geometric higher homotopies between the operation of $\ainf$ multiplying by a specified geometrically defined cohomological unit, and the operation of $\ainf$ multiplying by a strict unit (which is algebraically defined, but may also be geometrically characterized in terms of forgetful maps).
        From this one defines a strictly unital $\ainf$ category
        $\mc{F}^{hu}$ with $\hom_{\mc{F}^{hu}}(X,X) = \hom_{\mc{F}}(X,X) \oplus
        \K \langle  e^+_X, f_X\rangle$ and $\hom_{\mc{F}^{hu}}(X,Y) = \hom_{F}(X,Y)$ for $X \neq Y$, extending the $\ainf$ structure on $\mc{F}$,
        such that each $e^+_X$ a strict unit and $\mu^1(f_X) = e^+_X - e_X$ for $e_X$ a
        chosen a cohomological unit. The geometric higher homotopies alluded to
        above give operations used to define e.g., $\mu^k$  of a sequence of
        elements containing one or more $f_X$ terms. 

        Remark \ref{strictlyunitalB} then equips the usual Hochschild complex
        $\r{CH}_*(\mc{F}^{hu},\mc{F}^{hu})$ with a strict $S^1$ action. Using this
        one can construct a cyclic open-closed map
        with source $\r{CH}_*(\mc{F}^{hu}, \mc{F}^{hu})$, in a manner completely
        analogous to the construction of $\mc{F}^{hu}$ and the cyclic open-closed map here. This option is
        equivalent to the one we have chosen (and has some benefits), but
        requires additional technicalities/moduli spaces beyond the route taken here (both in constructing and
        defining the category $\mc{F}^{hu}$, and then in defining further ``higher
        homotopies'' between inserting a cohomological unit asymptotic and imposing a strict unit ---
        i.e. forgettable --- constraint into the cyclic open-closed map in various places, which
        give operations that correspond to applying the cyclic open-closed
        map to a Hochschild chain with one or more $f_X$ terms). 
        
        A construction of homotopy units was introduced in the
        work of Fukaya-Oh-Ohta-Ono \cite{Fukaya:2009qf}*{Ch. 7, \S 31}.  See
        \cite{ganatra1_arxiv} for an implementation in the (possibly wrapped) exact
        (or otherwise tautologically unobstructed), multiple Lagrangians
        setting. 
\end{enumerate}
\end{rem}

\subsection{The Fukaya category}\label{subsec:fukaya}
The goal of this subsection is to review (under simplifying technical
hypotheses) the definition of the Fukaya category of a symplectic manifold. The
outcome, a (homologically unital but not necessarily strictly unital) $\ainf$
category, will in particular carry a circle action on its non-unital
Hochschild complex.

In \S \ref{subsec:assumptions} below, we detail a set of simplifying
assumptions imposed on all of the moduli spaces of Floer curves considered in
this paper (mostly pertaining to transversality and compactness), and recall
examples of the variety of geometric situations in which they are satisfied. 
Such assumptions are in particular satisfied in the
technically simplest cases in which (compact or wrapped) Fukaya categories can
be defined, namely exact (Liouville) and monotone or aspherical symplectic
manifolds. In \S \ref{subsec:admissiblefukaya} we'll quickly review the construction of
the Fukaya category under such hypotheses. The initial thread of discussion will focus on compact Lagrangians, but 
immediately extends to wrapped Fukaya categories of Liouville manifolds as described in a series of Remarks (here we are using the 
framework of quadratic Hamiltonians as defined in \cite{Abouzaid:2010kx} for wrapped Fukaya categories, whose construction is nearly as simple as that of compact Fukaya categories and requires only a few minor modifications).

\subsubsection{Geometric setup and assumptions about moduli spaces of Floer trajectories}\label{subsec:assumptions}

To simplify technicalities, the main assumption we make
about moduli spaces in this paper is:
\begin{assumption}[Main assumption about moduli spaces]\label{mainassumption}
    All (semi-stably compactified) moduli spaces of Floer trajectories considered in this paper of virtual dimension $\leq 1$ are --- for generic choices of complex structure and Hamiltonian (``perturbation data'') --- compact transversally cut out manifolds with boundary of dimension equal to virtual dimension. Moreoever, the union of any such moduli space with fixed ``input'' asymptotic conditions over all possible ``output'' asymptotic conditions remains compact, and in particular is empty for all but finitely many possible output conditions (vacuous when there are only finitely many possible outputs).
\end{assumption}
Let $M = (M^{2n}, \omega)$ denote our target symplectic manifold and fix a collection of (always properly embedded) Lagrangian submanifolds $\{L_i\}$ in $M$ which we wish to be the objects of our Fukaya category. We will call any $M$, $\{L_i\}$, and choices of Floer perturbation data used to define moduli spaces for which Assumption \ref{mainassumption} holds {\em admissible}. We will say $M$ and/or $M, \{L_i\}$ are admissible if they possess an ample supply of Floer data for which Assumption \ref{mainassumption} holds for the moduli spaces considered below involving these targets. Examples of admissible $M$ include:
\begin{itemize}
    \item any  Liouville manifold  (in particular noncompact), which is to say that $\omega$ is exact with a fixed choice of primitive $\lambda$, such that flowing out by the Liouville vector field $Z$ (defined by $\iota_Z \omega = \lambda$) induces a diffeomorphism 
        \begin{equation}\label{cylend}
            M \backslash \mathring{\bar{M}} \cong \partial \bar{M} \times [0, \infty) 
            \end{equation}
            for some codimension-0 manifold-with-boundary $\bar{M}$ (called a Liouville domain whose completion is $M$);  or

        \item any compact symplectic manifold which is either monotone ($[\omega] = 2 \tau c_1(M)$ for some constant $\tau > 0$) or symplectically aspherical ($\omega(\pi_2(M)) = 0$).
\end{itemize}
If $M$ is Liouville, we henceforth fix a cylindrical end \eqref{cylend}, and use $r$ to refer to the corresponding $[0,\infty)$ coordinate.
Examples of (properly embedded) admissible Lagrangian submanifolds $L \subset M$ in admissible $M$ include
\begin{itemize}
    \item in Liouville $M$, one can take any exact $L$ (i.e., with $\lambda_{L} = df$), equipped with fixed choice of primitive which vanishes outside a compact set (which implies, as in \eqref{cylend} that $L$ is modeled on the cone of a Legendrian near infinity);

    \item in (compact) monotone $M$, one can take monotone $L$ (in the sense that $\omega(-)= \rho \mu_L(-): H_2(M,L) \to \R$ for some constant $\rho>0$, where $\omega$ is symplectic area and $\mu_L$ is the Maslov class). 

    \item  in (compact) symplectically aspherical $M$ one can take $L$ to be tautologically unobstructed (i.e., $L$ bounds no $J$-holomorphic discs for some $J$, which holds for all $J$ if $\omega(\pi_2(M,L)) = 0$).
\end{itemize}
The conditions above on $M$ and $L$ serve to rule out ``bad'' (unstable) breakings (such as $J$-holomorphic sphere bubbles in $M$ or $J$-holomorphic disc bubbles in $M$ with boundary on $L$) from arising in the limit of a sequence of curves in the moduli spaces considered, which could obstruct compactness and/or simultaneously complicate transversality arguments.
\begin{rem}[More general examples of admissible $M$ and $L$]
More generally, one could impose that the possible non-compactness of $M$ and (if $M$ is non-compact) $L$ must be of the
geometrically tame variety and that $M$/$L$  have no/bound no $J$-holomorphic
spheres/discs, or if they do that such spheres/discs can either be shown (using classical methods) either  
not to arise in the compactifications of virtual 1-dimensional moduli spaces,
or to arise but only contribute cancelling contributions to the resulting
algebraic formulae.  
\end{rem}
For non-compact $M$ and $\{L_i\}$, on any given moduli space of trajectories considered, further (non-generic) assumptions on the profile of Floer perturbation data near $\infty$ are required to ensure Assumption \ref{mainassumption} holds (to preclude sequences of curves escaping to $\infty$ so that usual Gromov compactness techniques apply, and also to obtain the second finiteness statement of Assumption \ref{mainassumption}, which is trivial in the compact case due to there being a finite list of outputs). We will say a few words about this in Remarks \ref{admissibleH}, \ref{liouvilleFloerdatum}, \ref{wrappedFloerdatum} and \ref{wrappedainf}; 
the verification of Assumption \ref{mainassumption} for the $\ainf$ structure maps (by citing established works) appears in Lemma \ref{smoothcompactcor}. 
The verification of Assumption \ref{mainassumption} for other moduli spaces considered in the paper is identical and hence omitted. However, we will in various places point out the restrictions are needed on Floer data in non-compact cases to preclude curves escaping to infinity and obtain finiteness along the lines of Lemma \ref{smoothcompactcor}.

\subsubsection{Admissible Fukaya categories}\label{subsec:admissiblefukaya}

For an admissible $M$ we will review the definition of the Fukaya category associated to an
an admissible collection of Lagrangians in $M$, which we will term an {\em admissible Fukaya category}. Examples of admissible Fukaya categories (in light of the examples given above) include
\begin{enumerate}
    \item in a compact aspherical $M$, the Fukaya category of tautologically unobstructed Lagrangians,

    \item in a monotone $M$, the Fukaya category of monotone Lagrangians,

    \item in Liouville $M$, the Fukaya category of compact exact Lagrangians, and

    \item in Liouville $M$, the wrapped Fukaya category of exact (cylindrical at infinity) Lagrangians.
\end{enumerate}

Fix first an underlying ground field $\K$ and grading structure ($\Z$ or $\Z/2$
here but see Remark \ref{othergradings}) we wish to use when defining the
category.  If $2c_1(M) = 0$ and we wish to define a $\Z$-graded category, we
begin by
equipping $M$ with a {\em grading structure}, which is a trivialization of the
square of the canonical bundle $(\Lambda_{\C}^n T^*M)^{\otimes 2}$. 
Next one equips the Lagrangian
submanifolds under consideration with some extra structure depending on the ground
field $\K$ and the grading structure. Concretely,
we say an {\em admissible Lagrangian brane} consists of a properly embedded admissible
Lagrangian submanifold $L \subset M$
which is equipped with the following 
extra two optional pieces of data 
(which are only required if one wants to
work with $\mathrm{char\ }\K \neq 2$ or with $\Z$ gradings respectively, the
latter in particular is always excluded in the monotone case):
    \begin{align}
        &\textrm{an orientation and Spin structure; and}\\
        &\textrm{a grading in the sense of \cite{Seidel:2000uq} (with respect to the fixed grading structure on $M$).}
    \end{align}
(these choices of extra data require $L$ to be Spin and satisfy $2c_1(M,L)
= 0$, where $c_1(M,L) \in H^2(M,L)$ is the relative first Chern class respectively).
\begin{rem}\label{othergradings}
    There are other possible grading structures on $M$ and $L$ one can use to equip the Fukaya category with suitable gradings (under geometric hypotheses), for instance $\Z/2k$ gradings, homology class gradings or hybrids thereof (c.f., \cite{Seidel:2000uq, Sheridan:2015aa}). We suppress discussion of these, but --- seeing as such matters are largely orthogonal to our arguments --- note that our results apply in such contexts as well.
\end{rem}
Henceforth, by abuse of notation all Lagrangians are implicitly admissible Lagrangian branes.  Denote by $\ob
\f$ a finite collection of such (admissible) Lagrangian( brane)s.  Choose a (potentially time-dependent)
Hamiltonian $H = H_t: M \to \R$ satisfying
the following genericity condition:
\begin{assumption}\label{nondegenerateH}
      All time-1 chords of $X_{H_t}$ between any pair of Lagrangians in $\ob \f$
    are non-degenerate.  
\end{assumption}
\begin{rem}
    It is straightforward to adapt all of our constructions to larger
    collections of Lagrangians, by for instance, choosing a different
    Hamiltonian $H_{L_0, L_1}$ for each pair of Lagrangians $L_0, L_1$ (as well as a different $H$ for closed orbits), and
    by choosing Floer perturbation data depending on corresponding sequences of
    objects,
    see e.g., \cite{Seidel:2008zr}*{\S 9j}.
    We have opted for using a single $H_t$ simply to keep the notation
    simpler.
\end{rem}

\begin{rem}[Admissible Hamiltonians in the Liouville case]\label{admissibleH}
    When $M$ is Liouville, we need to impose further restrictions on the profile of $H$ near $\infty$ in order to satisfy Assumption \ref{mainassumption}. If $\ob \f$ consists solely of compact exact Lagrangians, it suffices to impose that $H$ is compactly supported or more generally
of the form $f(r)$ near infinity for some function
with non-negative first and second derivatives. If $\ob \f$ contains any non-compact Lagrangians, we will impose following \cite{Abouzaid:2010kx} that $H$ satisfies the
    following {\em quadratic at $\infty$} condition:
        $H = r^2$ on the cylindrical end \eqref{cylend} (outside a compact subset).
\end{rem}

For any pair of Lagrangians $L_0, L_1 \in \ob \f$ the set of time 1
Hamiltonian flow lines of $H$ from $L_0$ to $L_1$, $\chi(L_0, L_1)$, can be thought of as the
critical points of an action functional 
on the {\it path space} from $L_0$ to $L_1$, $\mathcal{P}_{L_0, L_1}$ (this
functional is a priori multi-vaued, but it is certainly $\R$-valued in the presence of primitives $\lambda$ for
$\omega$ and $f_i$ for $\lambda|_{L_i}$).
Given a choice of grading structure on $M$ and grading for each $L_i$ above,
elements of $\chi(L_0, L_1)$ can be graded by the {\em Maslov index}
\begin{equation}
\deg: \chi(L_0, L_1) \ra \Z 
\end{equation}
(in the absence of grading structures this is always well defined mod 2,
provided our Lagrangians are oriented, which is automatic if they are Spin). As
a graded $\K$-module, the  morphism space in the Fukaya category between $L_0$ and $L_1$, also
known as the ({\em wrapped} if $M$ is Liouville) {\em Floer homology cochain
complex of $L_0$ and $L_1$ with respect to $H$}, has one (free) generator for
each element of $\chi(L_0,L_1)$; concretely
\begin{equation}
    \hom_{\f}^i(L_0, L_1) = CF^*(L_0, L_1, H_t, J_t):= \bigoplus_{x\in \chi(L_0,L_1), deg(x)=i} |o_x|_{\K},
\end{equation}
where the {\it orientation line} $o_x$ is the real vector space associated to
$x$ by index theory
(see \cite{Seidel:2008zr}*{\S 11h}; a priori $o_x$ depends on a choice of
trivialization of $x^*TM$ compatible with the grading structure. However, there
is a unique such choice in the presence of a $\Z$-grading, and in the
$\Z/2$-graded case any two choices made induce canonically isomorphic
orientation lines)
and for any one dimensional real vector space $V$ and any ring $\K$, the {\it $\K$-normalization} 
\begin{equation}
    \label{normalization}|V|_{\K}
\end{equation} 
is the $\K$-module generated by the two possible orientations on $V$, with the
relationship that their sum vanishes (if one does not want to worry about
signs, note that $|V|_{\Z/2} \cong \Z/2$ canonically).  

The $\ainf$ structure maps arise as counts of parameterized families of
(suitably coherently perturbed) solutions to Floer's equation with source a disc with $d$ inputs and one
output.  We will quickly summarize the definition and relevant choices required, referring the
reader to standard references for more details (the basic reference we follow is
\cite{Seidel:2008zr} for Fukaya categories of compact exact Lagrangians in
Liouville manifolds; see also \cite{Sheridan:2016} for the mostly
straightforward generalization to the monotone case. In the main body of exposition, we
focus on the (slightly simpler) case of compact (admissible) Lagrangians in compact (admissible) symplectic manifolds; we
detail the additional data and variations required for Fukaya
categories of exact Lagrangians in Liouville manifolds (which are simpler if
one is working only with compact exact Lagrangians) in Remarks
\ref{admissibleH}, \ref{liouvilleFloerdatum},
\ref{wrappedFloerdatum} and \ref{wrappedainf}. 

For $d \geq 2$, we use the notation
$\overline{\mc{R}}^d$
for the (Deligne-Mumford compactified) moduli space of discs with $d+1$ marked
points modulo reparametrization, with one point $z_0^-$ marked as negative and
the remainder $z_1^+, \ldots, z_d^+$ (labeled counterclockwise from $z_0^-$)
marked as positive.  
Orient the open (interior) locus $\mc{R}^d$ as in
\cite{Seidel:2008zr}*{\S 12g} and \cite{Abouzaid:2010kx}.
$\overline{\mc{R}}^d$ can be given the structure of a manifold-with-corners,
and its higher strata are trees of stable discs with a total of $d$ exterior
positive marked points and 1 exterior negative marked point. 
Denote the positive and negative semi-infinite strips by
\begin{align}
\label{strip1} Z_+ &:= [0,\infty) \times [0,1]\\
\label{strip2} Z_- &:= (-\infty,0] \times [0,1]
\end{align}
One first equips the spaces $\overline{\mc{R}}^d$ for each $d$ with a {\it
consistent collection of strip-like ends} $\mathfrak{S}$: that is, for each
component $S$ of $\overline{\mc{R}}^d$, a collection of maps $\e^{\pm}_k: Z_\pm \ra
S$ all with disjoint image in $S$, chosen so that positive/negative strips map
to neighborhoods of positively/negatively-labeled boundary marked points
respectively, smoothly varying with respect to the manifold-with-corners
structures and compatible with choices made on boundary and corner strata, which
are products of lower dimensional copies of $\overline{\mc{R}}^k$'s.

In order to associate transversely cut out moduli spaces of such maps, 
one studies a parametric family of solutions to Floer's equation depending on a choice of ``Floer (or perturbation) data'' over the parameter space. Concretely, 
a {\em Floer datum} for a family of domains (in this case $\overline{\mc{R}}^d$) is a choice of, for each domain $S$
in the parametric family, of 
\begin{itemize}
\item an $S$-(or domain-)dependent almost complex
structure $J_S$ and Hamiltonian $H_S$,  
\item a one-form $\alpha$ on $S$;
    \end{itemize}
which depends (smoothly) on the particular domain
in $\overline{\mc{R}}^d$ (and the position in that domain), and are {\em
compatible with strip-like ends}, meaning $\alpha$ pulls back to $dt$ and
$(H_S, J_S)$ pull back to a fixed choice $(H_t, J_t)$ in coordinates
\eqref{strip1}-\eqref{strip2}.
One inductively chooses a {\em Floer datum for the $\ainf$ structure}, which is
a choice of Floer data for the collection of domains $\{\overline{\mc{R}}^d\}_{d \geq 2}$
which is {\em consistent}, meaning that the Floer data chosen on a given
family of domains $\overline{\mc{R}}^d$ agree smoothly along the boundary and corners (which are products of lower dimensional $\overline{\mc{R}}^k$'s)
with previous choices of Floer data made. Such consistent choices exist essentially
because spaces of Floer data are contractible.

\begin{rem}[Floer data for compact exact Lagrangians in Liouville manifolds]\label{liouvilleFloerdatum}
    If $M$ is Liouville and we are studying the Fukaya category of compact
    exact Lagrangains, there is an additional requirement imposed on any Floer
    datum one uses; namely one requires that $J_S$ be of {\em contact type} in a
    neighborhood of infinity in the sense of \cite{Seidel:2008zr}*{(7.3)}, and
    $H_S$ be either $0$ or of the form $f(r)$ near infinity for some function
    with non-negative first and second derivatives. The more restrictive types
    of Floer data chosen for wrapped Fukaya categories in Remark
    \ref{wrappedFloerdatum} of course suffice as well. 
\end{rem}

\begin{rem}[Floer data for wrapped Fukaya categories]\label{wrappedFloerdatum}
    Following \cite{Abouzaid:2010kx}, we recall the additional information and constraints appearing in Floer data for
    wrapped Floer theory (with quadratic Hamiltonians). If $M$ is a Liouville manifold let $\psi^{\rho}: M \to M$ denote the time $\log(\rho)$ (outward) Liouville flow.
    One fixes for each $S$ in addition to $(H_S, J_S, \alpha_S)$, a 
    collection of constants $w_k \in \R_{>0}$ for each end called {\em weights} (so $w_k$ is the weight associated to the $k$th end), and a
    map $\rho_S: \partial S \to \R_{>0}$, called the {\em time-shifting map}. These should all satisfy:
    \begin{enumerate}
        \item \label{wrappedtimeshiftingweight}
            $\rho_S$ should be constant and equal to the weight $w_k$ on the $k$th strip-like end;
        \item \label{wrappedoneform} The 1-form $\alpha_S$ should be {\em sub-closed} meaning $d\alpha_S \leq 0$, equal to $w_k dt$ in the local coordinates on each strip-like end, and 0 when restricted to $\partial S$. By Stokes' theorem, this condition implies the sum of weights over all negative ends is greater than or equal to the sum of weights over all positive ends, and there should therefore be at least one negative end always (in this case there is one); 
        \item \label{wrappedhamiltonian} The Hamiltonian should be quadratic at infinity and pull back to $\frac{H \circ \psi^{w_k}}{w_k^2}$ in coordinates on each end (note that such a Hamiltonian is quadratic if $H$ is by an elementary computation \cite{Abouzaid:2010kx}*{Lem. 3.1});
        \item \label{wrappedalmostcomplexstructure} The almost complex structure should be of contact-type at infinity and equal to $(\psi^{w_k})^*J_t$ in coordinates on each end.
    \end{enumerate}
    There is a {\em rescaling action} by $(\R_{>0}, \cdot)$ on the space of such surface dependent data, which 
    sends $(\rho_S, \{w_k\}, \alpha_S, H_S, J_S) \mapsto (\lambda \rho_S, \{\lambda w_k\}, \lambda \alpha_S, \frac{H_S \circ \psi^{\lambda}}{\lambda^2},(\psi^{\lambda})^*J_S)$ for $\lambda \in \R_{>0}$. 
    Using this action, one also relaxes the consistency
    requirement imposed: The Floer datum on $\overline{\mathcal{R}}^d$ must agree smoothly, on a boundary or corner stratum, with {\em some rescaling} of the previously made choice (compare \cite{Abouzaid:2010kx}*{Def 4.1}).
\end{rem}

Given our choices of Floer data, we can define the moduli spaces appearing in the $\ainf$ operations. First
for any pair of objects $L_0, L_1$, and any pair of chords $x_0, x_1 \in
\chi(L_0, L_1)$, define 
$ \widetilde{\mc{R}}^{1}(x_0; x_1)$
to be the moduli space of maps $u: \R_s \times [0,1]_t \to M$ with boundary condition and asymptotics $u(s,0) \in L_0$, $u(s,1) \in L_1$, $\lim_{s \to +\infty} u(s,t) = x_1$ and $\lim_{s \to -\infty} u(s,t) = x_0$
satisfying
Floer's equation for $(H_t, J_t)$:
\begin{equation}\label{coordfreefloer}
(du - X \otimes dt)^{0,1} = 0
\end{equation}
where $X$ is the Hamiltonian vector field associated to $H_t$ and $(0,1)$ is taken with respect to $J_t$.
The translation action on $\R_s$ descends to a map on this moduli space (as the equation satisfied is $s$-independent), and we define the moduli space of (unparametrized) Floer strips to be
\begin{equation}
\label{semistablemodulispace}
\mc{R}^1(x_0; x_1) := \widetilde{\mc{R}}^1(x_0; x_1) / \R
\end{equation}
(with the added convention that whenever we are in a component of
$\widetilde{\mc{R}}^1(x_0; x_1)$ with expected dimension 0, this quotient
is replaced by the empty set).
Now for $d\geq 2$ let $L_0, \ldots, L_d$ be objects of $\f$ and fix any sequence of chords
$\vec{x} = \{x_k \in \chi(L_{k-1},L_k)\}$ as well as another chord $x_0 \in
\chi(L_0,L_d)$. 
We write $\mc{R}^d(x_0; \vec{x})$ for the space of maps 
\[u: S \ra M \] 
with source an arbitrary element $S \in \mc{R}^d$, satisfying boundary
conditions and asymptotics
\begin{equation}\label{conditionsasymptotic}
    \begin{cases}
        u(z) \in  L_k & \mathrm{if\ }z\in \partial S\textrm{ lies between }z^k\textrm{ and }z^{k+1} \\
        \lim_{s\ra \pm \infty} u \circ \e^k(s,\cdot) = x_k
    \end{cases}
\end{equation}
(where the limit above is taken as $s \to +\infty$ if the $k$th end is positive and $-\infty$ if it is negative) and differential equation 
\begin{equation}
    (du - X_S \otimes \alpha_S)^{0,1} = 0
\end{equation}
where $X_S$ is the Hamiltonian vector field associated to $H_S$ and where $0,1$
is taken with respect to the complex structure $J_S$ (for the choice of
consistent Floer datum we have fixed). 

The consistency condition imposed on Floer data 
over the abstract moduli spaces $\overline{\mc{R}}^d$, 
along with the compatibility with strip-like ends, implies that the (Gromov-type)
compactification of the space of maps $\overline{\mc{R}}^d(x_0; \vec{x})$ can be formed by
adding the images of the natural inclusions of products of lower-dimensional such moduli spaces:
\begin{equation}\label{boundaryainfinitymoduli}
  \overline{\mc{R}}^{d_2}(y; \vec{x}_2) \times  \overline{\mc{R}}^{d_1}(x_0; \vec{x}_1)  \ra \overline{\mc{R}}^d(x_0; \vec{x})
\end{equation}
where $y$ agrees with one of the elements of $\vec{x}_1$ and $\vec{x}$ is
obtained by removing $y$ from $\vec{x}_1$ and replacing it with the sequence
$\vec{x}_2$. Here, we let $d_1$ range from 1 to $d$, with $d_2 = d-d_1 + 1$,
with the stipulation that $d_1=1$ or $d_2 = 1$ is the semistable case
\eqref{semistablemodulispace}.

\begin{rem}[Operations for wrapped Fukaya categories]\label{wrappedainf}
In the setting of the wrapped Fukaya category (continuing Remark
\ref{wrappedFloerdatum}), one needs to incorporate the map $\rho_S$ into
the Lagrangian boundary conditions and asymptotics specified in Floer's equation; namely
instead of \eqref{conditionsasymptotic} we require the moving boundary condition $u(z) \in
(\psi^{\rho_S(z)})^* L_k$ if $z\in \partial S$ lies between $z^k$  and
$z^{k+1}$
where $(\psi^\rho)^* L_i$ denotes the pullback by $\psi^{\rho}$ (or application of $(\psi^{\rho})^{-1} = \psi^{\frac{1}{\rho}}$).  
We similarly impose that on the $k$th end, $\lim_{s\ra \pm \infty} u \circ
\e^k(s,\cdot) = (\psi^{\rho_S(z) := w_k})^* x_k$. The point is that Liouville
flow for time $\log(\rho)$ defines a canonical identification between Floer
complexes
\begin{equation}\label{rescaling}
    CF^*(L_0,L_1;H,J_t) \simeq 
    CF^*\left((\psi^\rho)^* L_0, (\psi^\rho)^* L_1; 
    \frac{H}{\rho} \circ \psi^\rho, (\psi^\rho)^* J_t\right).
\end{equation}
The right hand object is equivalently the (wrapped) Floer complex for $((\psi^\rho)^* L_0,
(\psi^\rho)^* L_1)$ for a strip with one-form $\rho dt$ using Hamiltonian $\frac{H}{\rho^2}
\circ \psi^\rho$ and $(\psi^\rho)^* J_t$. Up to Liouville flow, the Floer equation
and boundary conditions satisfied on the $k$th strip-like end therefore
coincides with the usual Floer equation for $(H_t, J_t)$ between $L_{k-1}$ and
$L_k$. In light of this condition and the weakened consistency requirement for
Floer data described in Remark \ref{wrappedFloerdatum}, one can again deduce
\eqref{boundaryainfinitymoduli}, that lower-dimensional strata of the Gromov
bordification of the space of maps can be identified (now possibly using a non-trivial
Liouville rescaling) with products of previously defined moduli spaces.
\end{rem}

In the graded setting, every connected component of the moduli space $\overline{\mc{R}}^d(x_0;\vec{x})$ has expected (or virtual) dimension $\deg (x_0) + d - 2 - \sum_{1\leq k \leq d} \deg(x_k)$; more generally, this moduli space consists of components of varying expected dimension (a number which can be computed using index theory in terms of the underlying homotopy class of $u$) all of whose mod 2 reductions are $\deg (x_0) + d - 2 - \sum_{1\leq k \leq d} \deg(x_k)$.
The following Lemma is the prototypical method of verifying Assumption \ref{mainassumption} for the various moduli spaces considered throughout the paper:
\begin{lem}\label{smoothcompactcor}
    Assumption \ref{mainassumption} holds for 
    the moduli spaces
    $\overline{\mc{R}}^d(x_0;\vec{x})$ for admissible $M, \{L_i\}$ and generic choices of a Floer datum for the $\ainf$ structure (satisfying the constraints detailed in Remarks \ref{admissibleH}-\ref{wrappedFloerdatum} in the Liouville case).
    Namely: components of these moduli spaces of virtual dimension $\leq 1$ are (for generic choices) compact
    manifolds-with-boundary of the given expected dimension.
    Moreoever, given a fixed $\vec{x}$ these moduli spaces are empty for all but
    finitely many $x_0$ (automatic if there are only finitely many possible $x_0$ to begin with e.g., if all of the Lagrangians being considered are
    compact).  
\end{lem}
\begin{proof}
    If $M$ is compact (and admissible), these assertions (the last of which is automatic) follow from
    standard Gromov compactness and transversality methods 
    as in \cite{Seidel:2008zr}*{(9k), (11h), Prop. 11.13}. 
    In the case $M$ and possibly also its Lagrangians are non-compact, there is an additional concern that solutions could escape to infinity in the target. To address this, one can e.g., appeal to
    the {\it integrated maximum principle} (compare \cite{Abouzaid:2010ly}*{Lemma 7.2} or \cite{Abouzaid:2010kx}*{\S B}), which implies
that elements of $\mc{R}(x_0; \vec{x})$ have image contained in a compact
subset of $M$ dependent on $x_0$ and $\vec{x}$, from where one can again appeal to standard Gromov compactness techniques. (this is strongly dependent on
the form of $H$, $J$, and $\alpha$ chosen for our Floer data as in Remarks \ref{admissibleH}, \ref{liouvilleFloerdatum}, \ref{wrappedFloerdatum}). The same result
can be used to show that solutions do not exist for $x_0$ of sufficiently negative {\it action}
compared to $\vec{x}$ (with our conventions, action is bounded
above and there are finitely many $x_0$ with action above any fixed level), 
verifying the last assertion.
\end{proof}
Choose a generic Floer datum for the $\ainf$ structure satisfying Lemma \ref{smoothcompactcor} and let
let $u \in \overline{\mc{R}}^d(x_0; \vec{x})$ be a rigid curve, meaning for us an element of the virtual dimension 0 component (which has dimension 0 in this case).
By \cite{Seidel:2008zr}*{(11h), (12b),(12d)}, given the fixed orientation of $\mc{R}^d$ (in the case $d \geq 2$ that is; for $d=1$ one instead needs to ``orient the operation of quotienting by $\R$'' as in \cite{Seidel:2008zr}*{(12f)}), any such element $u
\in \overline{\mc{R}}^d(x_0; \vec{x})$ determines an
isomorphism of orientation lines
\begin{equation}\label{ainforientation}
    \mc{R}^d_{u}: o_{x_d} \otimes \cdots \otimes o_{x_1} \lra o_{x_0}.
\end{equation}
Now for any 1-dimensional vector spaces $V_1, \ldots, V_k, W$, an isomorphism $f: V_k \otimes \cdots \otimes V_1 \to W$ induces a canonical map between $\K$-normalizations $|V_1|_{\K} \otimes \cdots \otimes |V_k|_{\K} \cong |V_1 \otimes \cdots \otimes V_k|_{\K} \to |W|_{\K}$, which by abuse of notation, to simplify notation, we also call $f$ (rather than $|f|_{\K}$). Using this,
define the $d$th $\ainf$ operation, for $d \geq 1$
\begin{equation}
    \mu^d: \hom^*_{\mc{F}}(L_{d-1}, L_d) \otimes \cdots \otimes \hom^*_{\mc{F}}(L_0,L_1) \lra \hom^*_{\mc{F}}(L_0,L_d)
\end{equation}
as a sum
\begin{equation}
    \mu^d([x_d], \ldots, [x_1]) := 
    \sum_{u \in \overline{\mc{R}}^d(x_0; \vec{x}) \textrm{ rigid}} (-1)^{\bigstar_d}\mc{R}^d_{u} ([x_d], \ldots, [x_1])
\end{equation}
where $[x_i] \in |o_{x_i}|_{\K}$ is an arbitrary element, $\mc{R}^d_u$ is the map (on $\K$-normalizations induced by) \eqref{ainforientation}, and the sign is given by
\begin{equation}
    \bigstar_d = \sum_{i=1}^d i \cdot \deg (x_i)
\end{equation}
(note that this sum is finite by Corollary \ref{smoothcompactcor}).
An analysis of the codimension 1 boundary of 1-dimensional moduli spaces along
with their induced orientations establishes that the maps $\mu^d$ satisfy the
$\ainf$ relations (see \cite{Seidel:2008zr}*{Prop. 12.3}).

We record here two abuses of notation which will systematically appear in definitions and usage of operations such as $\mu^d$. First, as
above, we will frequently use the same symbol for a multilinear map $F: V_1 \times \cdots \times V_k \to W$ and its corresponding linear map $F: V_1 \otimes \cdots \otimes V_k \to W$. Second, we will frequently use $x_i$ to refer to the arbitrary element $[x_i] \in |o_{x_i}|_{\K}$ to simplify expressions, e.g., above we might write $\mu^d(x_d, \ldots, x_1)$ in place of $\mu^d([x_d], \ldots, [x_1])$.

\section{Circle action on the closed sector}\label{section:closed}
\subsection{Floer cohomology and symplectic cohomology}\label{shsection}
Let $M$ be admissible as in \S 
\ref{subsec:assumptions}.
Given a (potentially time-dependent)
Hamiltonian $H: M \ra \R$, \emph{Hamiltonian Floer cohomology} when it is
defined is formally the Morse cohomology of the $H$-perturbed action functional $\mathcal{A}_H: \mc{L} M \ra \R$
on the free loop space $\mc{L} M$ of $M$. If $\omega$ is
exact and comes with a fixed primitive $\lambda$, this functional can be
written as:
\[
    x \mapsto -\int_x \lambda + \int_{0}^1 H_t(x(t))dt
\]
In general,
$\mathcal{A}_H$ may be multi-valued, but $d \mathcal{A}_H$ is always
well-defined, leading at least to a Morse-Novikov type theory. 
Recall that the set of {\it critical points} of $\mathcal{A}_{H_t}$
(when $H_t$ is implicit) is precisely the set $\mathcal{O}$ of time-1 orbits of
the associated (time-dependent) Hamiltonian vector field $X_H$, and we assume
$H_t$ is chosen sufficiently generically so that
\begin{assumption}\label{assumption:nondegeneratecriticalpoints}
    The elements of $\mathcal{O}$ are non-degenerate.
\end{assumption}
Optionally, given the data of a grading structure on $M$ in the sense of \S \ref{subsec:admissiblefukaya}
one can define an absolute $\Z$ grading on orbits by
$\deg(y) := n - CZ(y)$,
where $CZ$ is the Conley-Zehnder index of $y$ (and such a grading is always well-defined mod 2).

Fix a (potentially $S^1$-dependent) almost complex structure $J_t$. In the formal
picture, this induces a metric on $\mathcal{L} M$.  A {\it Floer trajectory} is
formally a gradient flowline of $\mathcal{A}_{H_t}$ using the metric
induced by $J_t$; concretely it is a map $u: (-\infty, \infty) \times S^1 \ra
M$ satisfying Floer's equation \eqref{coordfreefloer} (which is formally the gradient flow equation
for $\mathcal{A}_{H_t}$),
and converging exponentially near $\pm \infty$ to a pair of specified orbits $y^{\pm} \in
\mathcal{O}$.
In standard coordinates $s,t$ on the cylinder (i.e., $s \in \R$, $t \in \R/\Z = S^1$) this reads as
\begin{equation}
    \label{floersequation2} \partial_s u = -J_t (\partial_t u
    - X).  
\end{equation}

The space of non-constant Floer trajectories between a fixed $y^+$ and $y^-$
modulo the free $\R$ action given by translation in the $s$ direction is
denoted $\mathcal{M}(y^-; y^+)$. As in Morse theory, one should compactify this
space by allowing \emph{broken trajectories}: 
\begin{equation}\label{brokenhamfloer}
    \overline{\mc{M}}(y^-;y^+) = \coprod \mc{M}(y^k;y^+) \times \mc{M}(y^{k-1}; y^k) \times \cdots \times \mc{M}(y^1;y^2) \times \mc{M}(y^-;y^1) .
\end{equation}
In the graded situation, every component of $\overline{\mc{M}}(y^-; y^+)$ has expected/virtual dimension $\deg(y^-) - \deg(y^+)-1$; in general $\overline{\mc{M}}(y^-; y^+)$ has components of varying virtual dimension (of fixed parity $\deg(y^-) - \deg(y^+) - 1$) depending on the underlying homotopy class of the cylinder. By Assumption \ref{mainassumption} for $\overline{\mc{M}}(y^-; y^+)$, for generic choices of (time-dependent) $J_t$, the virtual dimension $\leq 1$ components of the moduli spaces $\overline{\mathcal{M}}(y^-; y^+)$ are compact manifolds (with boundary) of the given expected dimension; fix such a $J_t$.

Putting this all together, the {\it Floer co-chain complex} for $(H_t, J_t)$
over $\K$ has generators corresponding to orbits of $H_t$:
\begin{equation} 
    CF^i(M):= CF^i(M;H_t,J_t) := \bigoplus_{y\in \mc{O}, deg(y)=i} |o_y|_{\K},
\end{equation}
where the {\it orientation line} $o_y$ is a real vector space associated to
every orbit in $\mc{O}$ via index theory (see e.g., \cite{Abouzaid:2010kx}*{\S
C.6}) --- as before this index-theoretic definition a priori depends on a choice
of trivialization of $y^*TM$ compatible with the grading structure, but any two
choices induce isomorphic lines --- and $|V|_{\K}$  is the $\K$-normalization of
$V$ as in \eqref{normalization}.

The differential $d:
CF^*(M; H_t,J_t) \lra CF^*(M; H_t,J_t)$
counts rigid elements of the compactified moduli spaces \eqref{brokenhamfloer}. To fix sign issues, we recall that 
for a rigid element $u \in \mc{M}(y_0; y_1)$ (meaning $u$ belongs to a component of virtual, hence actual, dimension 0)
there is a natural isomorphism between orientation lines induced by index theory (see
e.g., \cite{Seidel:2008zr}*{(11h), (12b),(12d)}, \cite{Abouzaid:2010kx}*{Lemma
C.4})
\begin{equation}\label{muu}
    \mu_u: o_{y_1} \lra o_{y_0}.
\end{equation}
Then, one defines the differential as
\begin{equation}\label{differentialdefinition}
    d([y_1]) = 
    \sum_{u \in \overline{\mc{M}}(y_0;y_1)\textrm{ rigid}} (-1)^{\deg(y_1)} \mu_u ([y_1]),
\end{equation} 
where $[y_1] \in |o_{y_1}|_{\K}$ is an arbitrary element and $\mu_u$ is the map (on $\K$-normalizations induced by) \eqref{muu}.
One can show $d^2 = 0$ (under the assumptions made), and we call the resulting cohomology group $HF^*(H_t, J_t)$.

If $M$ is compact (and admissible), Assumption \ref{mainassumption} holds for all (suitably generic) $J_t$ (and all $H_t$ whose time-1 orbits are non-degenerate as in Assumption \ref{assumption:nondegeneratecriticalpoints}). If $M$ is non-compact and admissible then further hypotheses are needed on the profile of $(H_t, J_t)$ at $\infty$ to obtain admissibility (in particular to prevent curves from escaping to $\infty$ in $M$ and ensure compactness of $\bigcup_{y^-} \overline{\mc{M}}(y^-;y^+)$; we recall the two most relevant possible hypotheses for our purposes in \S \ref{subsubsec:sh}-\ref{subsubsec:relh}, which can lead to distinct Floer cohomology groups. For simplicity, the discussion in \S \ref{subsubsec:relh} subsumes the case of compact $M$ as well. 

\begin{rem}
    Our (cohomological) grading convention for Floer cohomology follows
    \cite{Seidel:2010fk, Ritter:2013aa, Abouzaid:2010kx, ganatra1_arxiv}.
\end{rem}

\subsubsection{Symplectic cohomology} \label{subsubsec:sh}

{\it Symplectic cohomology} \cite{Cieliebak:1995fk, Cieliebak:1996aa,
Floer:1994uq, Viterbo:1999fk}, the target of the open-closed map for wrapped Fukaya categories, is Hamiltonian Floer cohomology for a
particular class of Hamiltonians on non-compact convex symplectic manifolds.
There are several methods for defining this group.
We define it here by making the following specific choices of target,
Hamiltonian, and almost complex structure:
\begin{itemize}
    \item $M$ is a Liouville manifold {\em equipped with a conical end},
        meaning that it comes equipped with a choice of \eqref{cylend}.
        (this serves primarily as a
        technical device; the resulting invariants are independent of the specific choice).

    \item  The Hamiltonian term $H_t$ is a sum $H+F_t$ of an {\it autonomous
        Hamiltonian} $H: M \ra \R$ which is {\it quadratic at $\infty$}, namely
        \begin{equation} \label{eq:quadratic}
            H|_{M \backslash \bar{M}}(r,y) = r^2,
        \end{equation}
        and a time-dependent perturbation $F_t$ such that on the collar \eqref{cylend} of $M$,
        \begin{equation}\label{s1perturbation}
            \textrm{for any $r_0 \gg 0$, there exists an $R > r_0$ such that $F(t,r,y)$ vanishes in a neighborhood of $R$}.
        \end{equation}
        (for instance, $F_t$ could be supported near non-trivial orbits of $H$,
        where it is modeled on a Morse function on the circle). We denote by
        $\mathcal{H}(M)$ the class of Hamiltonians satisfying
        \eqref{eq:quadratic}.

    \item The almost complex structure should belong to the class $\mc{J}(M)$ 
        of complex structures which are {\it (rescaled) contact type} on the
            cylindrical end \eqref{cylend}, meaning that for some $c > 0$,
            \begin{equation}
                \lambda \circ J = f(r) dr
            \end{equation}
            where $f$ is any function with $f(r) > 0$ and $f'(r) \geq 0$.
\end{itemize}

A well known result (see e.g., \cite{Ritter:2013aa, Abouzaid:2010kx}) asserts
that Assumption \ref{mainassumption} holds for the resulting spaces of broken
Floer trajectories \eqref{brokenhamfloer}. Hence 
if $M$, $H_t$, and $J_t$ are as above, one has a well-defined Floer chain complex
$CF^*(M, H_t, J_t)$, which we refer to as the {\em symplectic co-chain complex} $SC^*(M)$; this will be the Floer chain complex we use when working with wrapped Fukaya categories.
We call the resulting cohomology group {\em symplectic cohomology}
$SH^*(M)$. 

\subsubsection{Relative cohomology} \label{subsubsec:relh} 

We review here the Floer cohomology group that is the target of the open-closed map for
an admissible symplectic manifold $M$ when working with a Fukaya category of
{\em compact} admissible Lagrangian submanifolds in the sense of \S
\ref{subsec:assumptions}.
Fix a (non-degenerate, generic) pair $(H_t, J_t)$
which is arbitrary for compact $M$ and which satisfies the following additional
properties if $M$ is Liouville:
\begin{itemize}
\item $H$ is linear of very small negative slope near infinity:
\begin{equation}
    H_t|_{M \backslash \bar{M}}(r, y) = - \lambda r
\end{equation}
where $r$ is the cylindrical coordinate and $\lambda \ll 1$ is a sufficiently
small number (smaller than the length of any Reeb orbit on $\partial \bar{M}$);
and

\item   $J_t$ is (rescaled) contact type near infinity as before.
\end{itemize}

It is well known that 
Assumption \ref{mainassumption} holds for the moduli spaces \eqref{brokenhamfloer} for generic $(H_t, J_t)$ as above (see e.g., \cite{Ritter:2013aa}), and also that:
\begin{prop}
    For generic $(H_t, J_t)$ as above there is an isomorphism $HF^*(H_t, J_t) \cong H^*(\bar{M}, \partial \bar{M})$ (which equals $H^*(M)$ in case $M$ is compact, using the convention then that $\bar{M} = M$ and $\partial \bar{M} = \emptyset$).\qed 
\end{prop}
The isomorphism can be realized in one of two ways:
\begin{itemize}
    \item  choose $H_t$ as above to be a $C^2$ small (time-independent) Morse function, in which case a
        well-known argument of Floer \cite{Floer:1989aa} equates $HF^*(H_t,
        J_t)$ with the Morse complex of $H$ by showing that all Floer
        trajectories must in fact be Morse trajectories of $H|_{\bar{M}}$
        (which in turn, as $H$ is inward pointing near $\bar{M}$, compute the
        relative cohomology) 

    \item Construct a geometric {\em PSS morphism} \cite{Piunikhin:1996aa} $PSS:
        H^*(\bar{M}, \partial \bar{M}) \cong H_{2n-*}(M) \to HF^*(H_t, J_t)$.
\end{itemize}

\subsection{The cohomological BV operator}\label{naivesquare}
The first order BV operator is a Floer analogue of a natural operator that
exists on the Morse cohomology of any manifold with a smooth $S^1$ action. Like
the case of ordinary Morse theory, this operator exists even when the
Hamiltonian and complex structure (c.f. Morse function and metric) are not
$S^1$-equivariant.

For $p \in S^1$, consider the following collection of cylindrical ends on $\R
\times S^1$: 
\begin{equation}
\begin{split}
        \e^{+}_p: (s,t) &\mapsto (s+1,t+p),\ \ \ s \geq 0\\
        \e^{-}_p: (s,t) &\mapsto (s-1,t),\ \ \ s \leq 0
\end{split}
\end{equation}

Pick $K: S^1 \times (\R \times S^1) \times M \ra \R$ dependent on $p$, satisfying
\begin{equation}
    (\e^{\pm}_p)^* K(p, s, \cdot,\cdot) = H(t,m)
\end{equation}
meaning that
\begin{equation}
    K_p(s,t,m) = \begin{cases} 
        H(t+p,m) & s \geq 1\\
        H(t,m) & s \leq -1,
    \end{cases}
\end{equation}
so in the range $-1 \leq s \leq 1$, $K_p(s,t,m)$ interpolates between
$H_{t+p}(m)$ and $H_t(m)$ (and outside of this interval is independent of $s$).

Similarly, pick a family of almost complex structures $J: S^1 \times (\R \times
S^1) \times M \ra \R$ satisfying
\begin{align}
    (\e^{\pm}_p)^* J(p, s, t,m) &= J(t,m)
\end{align}

Now, for $x^+, x^- \in \mc{O}$, define
\begin{equation}
    \mc{M}_{1}(x^-;x^+)
\end{equation}
to be the following {\it parametrized moduli space} of Floer cylinders
\begin{equation}
    \left\{ (p,u)\ |\ p \in S^1, u: S \ra M,\ \begin{cases}   
            \lim_{s \ra \pm \infty} (\e^{\pm}_p)^* u(s,\cdot) = x^{\pm} \\
            (du - X_K \otimes dt)^{0,1} = 0.
        \end{cases}\right\}
\end{equation}

There is a natural bordification by adding broken Floer cylinders to either end
\begin{equation} \label{compactificationBV}
    \overline{\mc{M}}_{1}(x^-;x^+) =  \coprod \mc{M}(a_0; x^+) \times \cdots \times \mc{M}(a_{k};a_{k-1}) \times \mc{M}_{1}(b_1; a_k) \times \mc{M}(b_2; b_1) \times \cdots \times \mc{M}(x^-;b_{l})
\end{equation}

\begin{rem}[Choices of $K$ and $J$ when $M$ is non-compact]\label{noncompactBVdata}
    When $M$ is non-compact and Liouville, further constraints on the profile of $K$ and $J$ are required near $\infty$ (beyond genericity) in order to satisfy Assumption \ref{mainassumption}.
    In the case of symplectic cohomology described in \S
    \ref{subsubsec:sh}, it suffices to choose
    $K$ carefully as follows. Given that $H_t(M) = H + F_t$ is a sum
    of an autonomous term and a time-dependent term that is zero at infinitely
    many levels tending towards infinity, we can ensure that 
    \begin{equation}\label{quadratichamiltonianlevels}
        \textrm{at infinity many levels tending towards infinity, }K_p(s,t,m)\textrm{ is equal to } r^2,
    \end{equation}
    and in particular is autonomous. In the setting of \S \ref{subsubsec:relh} (when $M$ is non-compact), we can similarly ensure a version of \eqref{quadratichamiltonianlevels} with $r^2$ replaced by $-\lambda r$ (in this case we could also more simply ensure that $K_p(s,t,m)= -\lambda r$ outside a compact set). In either case, one can take $J$ to be (rescaled) contact type on the cylindrical end.
    As usual, the verification of Assumption \ref{mainassumption} for the moduli spaces \eqref{compactificationBV} on Liouville $M$ follows by combining the results
    \cite{Abouzaid:2010kx}*{\S B} or  \cite{Abouzaid:2010ly}*{Lemma 7.2} (which prevent curves escaping to $\infty$ and ensure $\mc{M}_1(x^+, x^-)$ is empty for all but finitely many $x^-$ given the constraints near $\infty$ fixed in this remark) with classical transversality and compactness arguments.
\end{rem}

As before, $\overline{\mc{M}}_{1}(x^-; x^+)$ contains components of varying expected dimension depending on the underlying homotopy class $\beta$ of a map. Due to the fact that we are studying 1-parametric families of domains and not quotienting by $\R$, the relevant expected dimension is 2 more than the expected dimension of the components of $\overline{\mc{M}}(x^-; x^+)$ underlying the same homotopy class $\beta$. In particular in the graded case, this expected dimension is $\deg(x^+) - \deg(x^-) + 1$ for every component. By Assumption \ref{mainassumption}
for admissible choices of the above data (i.e., generic choices satsifying Remark \ref{noncompactBVdata} in the non-compact case), every component of $\overline{\mc{M}}_{1}(x^-; x^+)$ of virtual dimension $\leq 1$ is a compact manifold-with-boundary of dimension equal to its virtual dimension. (In particular, the boundary of the 1-dimensional components consists of the once-broken trajectories in \eqref{compactificationBV}).
In the usual fashion, counting rigid elements of this compactified moduli space of maps with
the right sign (explained more carefully in the next section) give an
operation $\delta_1: CF^*(M) \to CF^{*-1}(M)$, satisfying 
\[
d \delta_1 + \delta_1 d = 0
\]
(coming from the fact that the codimension 1 boundary of $\overline{\mc{M}}_{1}(x^-; x^+)$ is $\coprod_y
\overline{\mc{M}}(y; x^+) \times \overline{\mc{M}}_{1}(x^-; y) \cup \overline{\mc{M}}_{1}(y; x^+) \times \overline{\mc{M}}(x^-; y)$).
It would be desirable for $\delta_1$ square to zero on the chain level, which would give
$(CF^*(M), \delta_0 = d, \delta_1)$ the structure of a {\em strict $S^1$-complex}, or {\em
mixed complex}. However, the $S^1$ dependence of our Hamiltonian and almost
complex structure prevent this, in a manner we now explain.

Typically one attempts to prove a geometric/Floer-theoretic operation (such as
$\delta_1^2$) is zero by exhibiting that the relevant moduli problem has no
zero-dimensional solutions (due to, say, extra symmetries in the equation), or
otherwise arises as the boundary of a 1-dimensional moduli space.
To that end, we first describe a moduli space parametrized by $S^1 \times S^1$ which looks
like two of the previous parametrized spaces naively superimposed, leading us
to call the associated operation $\delta_2^{naive}$. The extra symmetry
involved in this definition will allow us to easily conclude
\begin{lem} \label{naivelemma}
    $\delta_2^{naive}$ is the zero operation.
\end{lem}
For $(p_1,p_2) \in S^1 \times S^1$, consider the following collection of
cylindrical ends:
\begin{equation}
\begin{split}
    \e^{+}_{(p_1,p_2)}: (s,t) &\mapsto (s+1,t+p_1+p_2),\ \ \ s \geq 0\\
    \e^{-}_{(p_1,p_2)}: (s,t) &\mapsto (s-1,t),\ \ \ s \leq 0
\end{split}
\end{equation}

Pick $K: (S^1 \times S^1) \times (\R \times S^1) \times M \ra \R$ dependent on $(p_1,p_2)$, satisfying
\begin{equation}
    \e^{\pm}_{(p_1,p_2)} K(p_1,p_2, s, \cdot,\cdot) = H(t,m)
\end{equation}
meaning that
\begin{equation}
    K_{(p_1,p_2)}(s,t,m) = \begin{cases} 
        H(t+p_1+p_2,m) & s \geq 1\\
        H(t,m) & s \leq -1,
    \end{cases}
\end{equation}
so in the range $-1 \leq s \leq 1$, $K_{p_1+p_2}(s,t,m)$ interpolates between
$H_{t+p_1+p_2}(m)$ and $H_t(m)$.

Similarly, pick a family of almost complex structures $J: S^1 \times S^1 \times
(\R \times S^1) \times M \ra \R$ 
\begin{align}
    \e^{\pm}_{(p_1,p_2)} J(p_1,p_2, s, t,m) &= J(t,m),
\end{align}
such that
\begin{equation}
    J\textrm{ only depends on the sum }p_1+p_2.
\end{equation}

Now, for $x^+, x^- \in \mc{O}$, define
\begin{equation}
    \mc{M}_{2}^{naive}(x^-; x^+)
\end{equation}
to be the {\it parametrized moduli space} of Floer cylinders
\begin{equation}
    \left\{ (p_1, p_2,u)\ |\ (p_1,p_2) \in S^1 \times S^1, u: S \ra M,\ \begin{cases}   
            \lim_{s \ra \pm \infty} (\e^{\pm}_{(p_1,p_2)})^* u(s,\cdot) = x^{\pm} \\
            (du - X_K \otimes dt)^{0,1} = 0.
        \end{cases}\right\}
\end{equation}
For generic choices of $K$ and $J$ (again bearing in mind the extra impositions
of Remark \ref{noncompactBVdata} in the non-compact case), this moduli space,
suitably compactified by adding broken trajectories, will be (for components of
virtual dimension $\leq 1$) a manifold with boundary of the correct (expected)
dimension; the dimension agrees mod 2 in the $\Z/2$-graded case and exactly in the graded case with $\deg(x^+) - \deg(x^-)
+ 2$ (the details are similar to the previous section, and will be omitted).
Counts of rigid elements in this moduli space will thus, in the usual fashion
give a map of degree $-2$, which we call $\delta_2^{naive}$.

\begin{proof}[Proof of Lemma \ref{naivelemma}]
    Let $(p_1, p_2, u)$ be an element of $\mc{M}_{2}^{naive}(x^-;x^+)$. Then,
    $(p_1-r, p_2 + r, u)$ is an element too, for any $r \in S^1$, as the
    equation satisfied by the map $u$ only depends on the sum $p_1+p_2$. We
    conclude that elements of $\mc{M}_{2}^{naive}(x^-;x^+)$ are never rigid, and
    thus that the resulting operation $\delta_2^{naive}$ is zero.  
\end{proof}

We would like $\delta_2^{naive}$ to be genuinely equal to $\delta_1^2$, which
would imply that $\delta_1^2 = 0$.  However, this is only true
on the homology level; the lack of $S^1$ invariance of our Hamiltonian and
almost complex structure, and the corresponding family of choices of homotopy
between $\theta^*H_t$ and $H_t$, over varying $\theta \in S^1$, breaks symmetry and
ensures that $\delta_1^2 \neq \delta_2^{naive}$ as geometric chain maps.
However, there is a geometric chain homotopy, $\delta_2$ between $\delta_1^2$
and $\delta_2^{naive}$ 
along with a hierarchy of higher homotopies $\delta_k$ forming the
$S^1$-complex structure on $CF^*(M)$, which we define in the next section (see
in particular Lemma \ref{weaks1actionSHlemma} for the proof of the
$S^1$-complex equations, one of which recovers the chain homotopy between
$\delta_1^2$ and $\delta_2^{naive}=0$).

\subsection{The \texorpdfstring{$\ainf$}{A-infinity} circle action}\label{sec:angledeccylinder}
We turn to a ``coordinate-free'' definition of the relevant parametrized
moduli spaces, which will help us incorporate the construction into open-closed
maps. 
\begin{defn}
    An {\em $r$-point angle-decorated cylinder} consists of a semi-infinite or
    infinite cylinder $C \subseteq (-\infty,\infty) \times S^1$, along with a
    collection of auxiliary points $p_1, \ldots, p_r \in C$, satisfying
    \begin{equation}\label{heightordering}
           (p_1)_s \leq \cdots \leq (p_r)_s,
    \end{equation}
    where $(a)_s$ denotes the $s \in (-\infty, \infty)$ coordinate. The {\em
    heights} associated to this data are the $s$ coordinates
    \begin{equation}
        h_i = (p_i)_s,\ i = 1, \ldots, r
    \end{equation}
    and the {\em angles} associated to $C$ are the $S^1$ coordinates
    \begin{equation}
        \theta_i := (p_i)_t, \ i \in 1, \ldots, r.
    \end{equation}
\end{defn}
The {\em cumulative rotation} of an $r$-point angle-decorated cylinder is the first angle:
\begin{equation}
    \eta := \eta(C, p_1, \ldots, p_r) = \theta_1.
\end{equation}
The {\em $i$th incremental rotation} of an $r$-point angle-decorated cylinder
is the difference between the $i$th and $i-1$st angles:
\begin{equation}
    \kappa^{inc}_i := \theta_i - \theta_{i+1} \textrm{  (where $\theta_{r+1} = 0$)}.
\end{equation}
Inductively each $\theta_i$ can be expressed as a sum of all incremental rotations from $i$ to $r$:
\begin{equation}\label{incrementalcumulative}
    \theta_i = \sum_{j=i}^r \kappa^{inc}_j
\end{equation}
\begin{defn}\label{rpointedmodulispace}
    The {\em moduli space of $r$-point angle-decorated cylinders} 
    \begin{equation}
        \mc{M}_r
    \end{equation}
    is the space of $r$-point angle-decorated infinite cylinders, modulo
    translation.  
\end{defn}
\begin{rem}[Orientation for $\mc{M}_r$]\label{orientationmr}
Note that $C_r$, the space of all $r$-point angle-decorated infinite cylinders (not
modulo translation) has a canonical complex orientation.  Thus, to orient the
quotient space $\mc{M}_r:=C_r/\R$ it is sufficient to give a choice of
trivialization of the action of $\R$ on $C_r$. We choose $\partial_s$ to be the
vector field inducing said trivialization.
\end{rem}

For an element of this moduli space, the angles and relative heights of the
auxiliary points continue to be well-defined, so there is a non-canonical
isomorphism
\begin{equation}
    \mc{M}_r \simeq (S^1)^r \times [0,\infty)^{r-1}
\end{equation}
The moduli space $\mc{M}_r$ thus possesses the structure of an open
manifold-with-corners, with boundary and corner strata given by the various
loci where heights of the auxiliary points $p_i$ are coincident (we allow the
points $p_i$
themselves to coincide; one alternative is to first Deligne-Mumford compactify,
and then collapse all sphere bubbles containing multiple $p_i$'s. That the
result still forms a smooth manifold with corners is a standard local
calculation near any such stratum).  Given an arbitrary representative $C$ of
$\mc{M}_r$ with associated heights $h_1, \ldots, h_r$, we can always find a
translation $\tilde{C}$ satisfying $\tilde{h}_r = -\tilde{h}_1$; we call this
the {\it standard representative} associated to $C$. 

Given a representative $C$ of this moduli space, and a fixed constant $\delta$,
we fix a positive cylindrical end around $+\infty$
\begin{equation}\label{posendangles}
    \begin{split}
     \epsilon^+: [0,\infty) \times S^1 &\ra C\\
     (s,t) &\mapsto (s + h_r + \delta, t)\\
\end{split}
\end{equation}
and a negative cylindrical end around $-\infty$ (note the angular rotation in $t$!):
\begin{equation}\label{negendangles}
    \begin{split}
        \epsilon^-: (-\infty,  0] \times S^1 & \ra C \\
        (s,t) &\mapsto ( s - (h_1 - \delta)), t +  \theta_1).
    \end{split}
\end{equation}
These ends are disjoint from the $p_i$ and vary smoothly with $C$; via thinking
of $C$ as a sphere with two points with asymptotic markers removed, these
cylindrical ends correspond to the positive asymptotic marker having angle 0
and the negative asymptotic marker having angle $\theta_1 = \kappa^{inc}_1 +
\kappa^{inc}_2 + \cdots + \kappa^{inc}_r$.

There is a compactification of $\mc{M}_r$ consisting of {\em broken $r$-point
angle-decorated cylinders}
\begin{equation}
    \overline{\mc{M}}_r = \coprod_{s} \coprod_{j_1, \ldots, j_s; j_i > 0, \sum j_i = r} \mc{M}_{j_1} \times \cdots \times \mc{M}_{j_s}.
\end{equation}
The stratum consisting of $s$-fold broken configurations lies in the
codimension $s$ boundary, with the manifolds-with-corners structure explicitly
defined by local gluing maps using the ends \eqref{posendangles} and
\eqref{negendangles}. The gluing maps, which rotate the bottom
cylinder of the gluing in order to match its top end \eqref{posendangles} with the bottom end
\eqref{negendangles} of the upper cylinder, induce cylindrical ends on the glued cylinders which agree 
with the choices of ends made in \eqref{posendangles}-\eqref{negendangles}. Concretely, for a 1-fold broken configuration of the form $\mc{M}_{r-k} \times \mc{M}_{k}$, the gluing map (for any choice of sufficiently small gluing parameter) has the following effect on angles (denoting coordinates in the second, bottom factor by $\bar{\theta}_j$ ($1 \leq j \leq k$) and in the first, top factor by $\theta_i$ ($1 \leq i \leq r-k$):
\begin{equation}\label{gluingangles}
    \left( (\theta_1, \ldots, \theta_{r-k}) , (\bar{\theta}_1, \ldots, \bar{\theta}_{k}) \right) \mapsto \left( \bar{\theta}_1 + \theta_1, \bar{\theta}_2 + \theta_1, \ldots, \bar{\theta}_{k} + \theta_1, \theta_1, \ldots, \theta_{r-k} \right);
\end{equation}
see Figure \ref{fig:angledecoratedgluing}.
\begin{figure}[h]
    \caption{\label{fig:angledecoratedgluing} The gluing map for an angle-decorated cylinder rotates all of the angles of the bottom cylinder by the first angle of the top cylinder as in \eqref{gluingangles}.}
    \centering
    \includegraphics[scale=0.8]{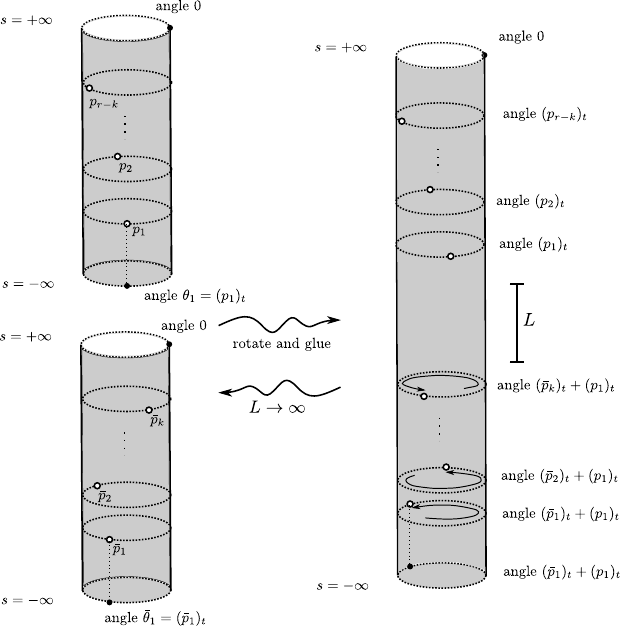}
\end{figure}
More simply, in the glued surface, the list of incremental angles $(\kappa_1^{inc,glued}, \ldots, \kappa_r^{inc,glued})$ is equal to the concatenation of the lists of incremental angles of the original bottom and top surfaces, $(\bar{\kappa}_1^{inc}, \bar{\kappa}_2^{inc}, \ldots, \bar{\kappa}_k^{inc}, \kappa_1^{inc}, \kappa_2^{inc},\ldots, \kappa_{r-k}^{inc})$.

The compactification $\overline{\mc{M}}_r$ thus has codimension-1 boundary
covered by the images of the natural inclusion maps
\begin{align} 
    \label{anglestratum1}\overline{\mc{M}}_{r-k} \times \overline{\mc{M}}_k &\lra \partial \overline{\mc{M}}_{r},\ 0 < k < r\\
    \label{anglestratum2}\overline{\mc{M}}_r^{i,i+1} &\lra \partial \overline{\mc{M}}_{r},\  1 \leq i < r,
\end{align}
where $\overline{\mc{M}}_r^{i,i+1}$
denotes the compactification of the locus where $i$th and $i+1$st heights are
coincident  
\begin{equation}
    \mc{M}_r^{i,i+1} := \{C \in \mc{M}_r\ |\ h_i = h_{i+1} \}.
\end{equation}
With regards to the above stratum, for $r > 1$ there is a projection map which
will be relevant, a version of the forgetful map which remembers only the
first of the angles with coincident heights: 
\begin{equation}
    \begin{split}
    \pi_i: \mc{M}_r^{i,i+1} &\lra \mc{M}_{r-1}\\
    (h_1, \ldots, h_i, h_{i+1}=h_i, h_{i+2}, \ldots, h_r) &\longmapsto (h_1, \ldots, h_i, h_{i+2}, \ldots, h_r)\\
    (\theta_1, \ldots, \theta_i, \theta_{i+1}, \ldots, \theta_r) &\longmapsto (\theta_1, \ldots, \theta_{i-1}, \theta_{i}, \widehat{\theta_{i+1}}, \theta_{i+2}, \ldots, \theta_r).
\end{split}
\end{equation}
$\pi_i$ is compatible with the choice of positive and negative ends
\eqref{posendangles}-\eqref{negendangles} and hence $\pi_i$ extends to
compactifications
\begin{equation} \label{addmapcompactified}
    \pi_i: \overline{\mc{M}}_{r}^{i,i+1} \ra \overline{\mc{M}}_{r-1}.
\end{equation}
We will equip each $r$-point angle-rotated cylinder $\tilde{C} := (C,p_1,
\ldots, p_r)$ with perturbation data for Floer's equation or a {\em Floer
datum} in the sense of the last section, which consists of
\begin{itemize}
    \item     The positive and negative cylindrical ends on $\epsilon^{\pm}: C^{\pm} \ra C$ chosen in
        \eqref{posendangles}-\eqref{negendangles}. 

    \item The one-form on $C$ given by $\alpha = dt$.

    \item A surface-dependent Hamiltonian $H_{\tilde{C}}: C \ra \mc{H}(M)$
        compatible with the positive and negative cylindrical ends, meaning that
         \begin{equation}
             (\epsilon^{\pm})^* H_{\tilde{C}} = H_t,
        \end{equation}
        where $H_t$ was the previously chosen Hamiltonian.

    \item A surface dependent almost complex structure $J_{\tilde{C}}: C \ra \mc{J}_1(M)$ also
        compatible with $\epsilon^{\pm}$, meaning that
        \begin{equation}
            (\epsilon^{\pm})^* J_{\tilde{C}} = J_t
        \end{equation}
        for our previously fixed choice $J_t$.
\end{itemize}
A choice of {\em Floer data for the $S^1$-action}
is an inductive (smoothly varying) choice of Floer data, for each $r$ and each representative
$S = (C,p_1, \ldots, p_r)$ of $\overline{\mc{M}}_r$, satisfying the
following consistency conditions at boundary strata: 
    \begin{align}
    \label{s1consistency} &\textrm{At a boundary stratum \eqref{anglestratum1}, the datum chosen coincides with the product of}\\
        \nonumber &\textrm{Floer data already chosen on lower-dimensional spaces. }\\
        \label{s1forgetful}  &\textrm{At a  boundary stratum \eqref{anglestratum2}, the Floer data coincides with the pull back, via the } \\
           \nonumber &\textrm{forgetful map $\pi_i$ defined in \eqref{addmapcompactified}
               of the Floer data chosen on $ \overline{\mc{M}}_{r-1}$.}
        \end{align}
Inductively, since the space of choices at each level is non-empty and
contractible (and since the consistency conditions are compatible along
overlapping strata), universal and consistent choices of Floer data exist.
From the gluing map, a representative $S$ sufficiently near the boundary strata
\eqref{anglestratum1} inherits cylindrical regions, also known as {\it thin
parts}, which are the surviving images of the cylindrical ends of
lower-dimensional strata. Together with the cylindrical ends of $S$, this
determines a collection of cylindrical regions:
\begin{defn}\label{cylregions} Given a fixed positive constant $\delta$,
    the {\em ($\delta$-spaced) rotated cylindrical regions} for an $r$-point angle-decorated
    cylinder $(C,p_1, \ldots, p_r)$ consist of the following cylindrical ends
    and finite cylinders (where $h_i = (p_i)_s$ and $\theta_i = (p_i)_t$): 
    \begin{itemize}
        \item The {\em top cylinder}
            \begin{equation}
                \begin{split}
                 \epsilon^+: [0,\max(top(C) - (h_r + \delta),0)] \times S^1 &\ra C\\
                 (s,t) &\mapsto (\min(s + h_r + \delta,top(C)), t)\\
            \end{split}
        \end{equation}

    \item The {\em bottom cylinder} 
            \begin{equation}
                \begin{split}
                    \epsilon^-: [\min(bottom(C) - (h_1 - \delta), 0), 0] \times S^1 & \ra C \\
                    \nonumber (s,t) &\mapsto ( \max(s - (h_1 - \delta),bottom(C)), t + \theta_1) \\
                    &= ( \max(s - (h_1 - \delta),bottom(C)), t + \sum_{j=1}^r \kappa_i)
                \end{split}
            \end{equation}
        
    \item For any $1 \leq i \leq r-1$ satisfying $h_{i+1} - h_i > 2 \delta$, the {\em $i$th thin part}
    \begin{equation}
        \begin{split}
            \epsilon_i: [h_{i} + \delta, h_{i+1} - \delta] \times S^1 &\ra C\\
            (s,t) &\mapsto (s, t + \eta_i)
        \end{split}
    \end{equation}
    \end{itemize}
\end{defn}
Note that a given $r$-point angle-decorated cylinder may not have an the $i$th
thin part, for a given $i \in [1,r-1]$, and indeed may not have any thin
parts. 
The consistency conditions at boundary strata can be ensured in particular by requiring that
for any ($\delta$-spaced)
rotated cylindrical region $\epsilon: C' \ra C$ of sufficiently large length (greater than some fixed $L$, say $L=3\delta$) associated to $(C,p_1, \ldots, p_r)$ and $\delta$, we have that
\begin{equation}
\epsilon^*(K_C, J_C) = (K_t, J_t).
\end{equation}
Given the cylindrical regions of Definition \ref{cylregions}, this would imply the following condition on
$(K_C, J_C)$ (assuming $L \gg 3 \delta$): for $z = (s,t) \in C$
\begin{equation}
    \begin{split}\label{rotatecondition}
        (K_{z}, J_{z}) &= (K_t, J_t) \ \textrm{ for }s > h_r + \delta;\\
        (K_z, J_z) &= (\epsilon^-)^*(K_t, J_t) = (K_{t+\theta_1},J_{t+ \theta_1 }) \ \textrm{ for }s < h_1 - \delta;\\
        (K_z,J_z) &= \epsilon_{i}^* (K_z,J_z) = (K_{t+\theta_i}, J_{t + \theta_i}) \textrm{ if } h_{i+1} - h_{i} > 3\delta 
        \textrm{ and } s \in [h_{i}+\delta, h_{i+1} - \delta]. 
    \end{split}
\end{equation}

Given a choice of Floer data for the $S^1$-action and a pair of asymptotics
$(x^+,x^-) \in \mc{O}$ for each $k \geq 1$ there is an associated parametrized
moduli space of Floer cylinders with source an arbitrary element of $S \in
\mathcal{M}_r$ (where the Floer equation is with respect to the Hamiltonian
    $H_S$ and $J_S$, with asymptotics $(x^+,x^-)$:
\begin{equation}
    \mc{M}_{r}(x^-; x^+):=
    \{ S = (C,p_1, \ldots, p_r) \in \mc{M}_r,\ u: C \ra M\ |\ 
        \begin{cases}
        \lim_{s \ra \pm \infty} (\e^{\pm})^* u(s,\cdot) &= x^{\pm} \\
(du - X_{H_S} \otimes dt)^{(0,1)_S} &= 0.
    \end{cases} \}
\end{equation}
The consistency condition imposes that the 
boundary of the Gromov bordification $\overline{\mc{M}}_r(x^-; x^+)$ 
is covered by the images of the natural inclusions
\begin{align}
    \label{goodboundary} \overline{\mc{M}}_{r-k}(y; x^+) \times \overline{\mc{M}}_k( x^-; y) \ra \partial \overline{\mc{M}}_r(x^-; x^+) \\
    \label{badboundary}\overline{\mc{M}}_r^{i,i+1} (x^-; x^+) \ra \partial \overline{\mc{M}}_r(x^-; x^+),
\end{align}
along with the usual semi-stable strip breaking boundaries
\begin{equation}
    \begin{split}
    \label{stripbreaking}
     \overline{\mc{M}}(y; x^+) \times \overline{\mc{M}}_r(x^-; y) &\ra \partial \overline{\mc{M}}_r(x^-; x^+) \\
    \overline{\mc{M}}_r(y; x^+) \times \overline{\mc{M}}(x^-; y)  &\ra \partial \overline{\mc{M}}_r(x^-; x^+) 
\end{split}
\end{equation}

\begin{rem}[Floer data in the Liouville case]\label{LiouvilleS1complex}
    Continuing Remark \ref{noncompactBVdata}, when $M$ is Liouville we impose the following further constraint on Floer data:
        \begin{equation}
            \begin{split}
                &\textrm{$H_{\tilde{C}}$ is equal to $r^2$ or $-\lambda r$ (depending on whether we are in the setting of \S \ref{subsubsec:sh} or \S \ref{subsubsec:relh})} \\
                &\textrm{at infinitely many levels of $r$ tending to $\infty$, and $J_{\tilde{C}}$ is (rescaled) contact type near $\infty$.}
        \end{split}
        \end{equation}
        (in fact, in the setting of \S \ref{subsubsec:relh} we can take $H_{\tilde{C}}$ to be
simply equal to $-\lambda r$ for all $r$ outside of a compact set). 
By e.g., \cite{Abouzaid:2010ly}*{Lemma 7.2} or \cite{Abouzaid:2010kx}*{\S B}, this hypothesis implies that sequences of curves with fixed asymptotics cannot escape to $\infty$ in $M$, and that $\mc{M}_{r}(x^-;x^+)$ given a fixed $x^+$ is non-empty for only finitely many $x^-$, both necessary inputs to verifying Assumption \ref{mainassumption}.
\end{rem}

In the $\Z$-graded case, the virtual dimension of (every component of) $\overline{\mc{M}}_r(x^-; x^+)$ is
\begin{equation}
        \deg (x^+) - \deg(x^-) + (2r -1);
\end{equation}
in the $\Z/2$-graded case every component has virtual dimension of the above parity.
By Assumption \ref{mainassumption}
for a generic fixed choice of Floer data for the $S^1$-action (satisfying Remark
\ref{LiouvilleS1complex} in the Liouville case), the components of virtual
dimension $\leq 1$ of the moduli spaces $\overline{\mc{M}}_r(x^-; x^+)$ are
compact manifolds-with-boundary of the correct (expected) dimension.
As usual, signed counts of rigid elements of this moduli space for varying
$x^+$ and $x^-$ (using induced maps on orientation lines, twisted as in the
differential by $(-1)^{\deg(x_+)}$---see \eqref{differentialdefinition}) give
the matrix coefficients for the overall map
\begin{equation}
    \delta_{r}: CF^*(M) \ra CF^{*-2r + 1}(M).
\end{equation}
In the degenerate case $r=0$ we define $\delta_0$ to be the (already defined) differential:
\begin{equation}
    \delta_0: = d: CF^*(M) \to CF^{*+1}(M).
\end{equation}

\begin{lem} \label{weaks1actionSHlemma}
    For each $r$, 
    \begin{equation}
        \label{circleactioneqn}\sum_{i=0}^r \delta_i \delta_{r-i} = 0.
    \end{equation}
\end{lem}
\begin{proof}
    The counts of rigid elements associated to the boundary of 1-dimensional
    components of $\partial \overline{\mc{M}}_r(x^+; x^-)$, along with a
    description of this codimension 1 boundary
    \eqref{goodboundary}-\eqref{stripbreaking} immediately implies that 
    \begin{equation}
        (\sum_{i=1}^r \delta_i \delta_{r-i}) + (\sum_i \delta_r^{i,i+1}) + (d \delta_r + \delta_r d) = 0,
    \end{equation}
    where $\delta_r^{i,i+1}$ for each $i$ is the operation associated to the
    moduli space of maps \eqref{badboundary}. (Observe that $\delta_2^{1,2}$ is
    precisely the operation $\delta_2^{naive}$ from \S \ref{naivesquare}).
    Note that the consistency condition \eqref{s1forgetful} implies that the Floer
    datum chosen for any element $S \in \mc{M}_r^{i,i+1}$ only depends on
    $\pi_i(S)$, where the forgetful map $\pi_i: \mc{M}_r^{i,i+1} \ra
    \mc{M}_{r-1}$ has 1-dimensional fibers. 
    Hence given an element $(S,u) \in \overline{\mc{M}}_r^{i,i+1} (x^-; x^+)$,
    it follows that $(S',u) \in \overline{\mc{M}}_r^{i,i+1}(x^-; x^+)$ for all $S'
    \in \pi_i^{-1}\pi_i(S)$.
    In other words, elements of $\overline{\mc{M}}_r^{i,i+1} (x^-; x^+)$ are
    never rigid, so the associated operation $\delta_r^{i,i+1}$ is zero.
\end{proof}
By definition we conclude:
\begin{cor}
    The pair $(CF^*(M; H_t, J_t), \{\delta_r\}_{r \geq 0})$ as defined above
forms an $S^1$-complex, in the sense of Definition
\ref{homotopycircleaction}.\qed 
\end{cor}
By using continuation maps parametrized by various $(S^1)^r \times
(0,1]^{r}$ (or equivalently, by spaces of angle-decorated
cylinders that are not quotiented by overall $\R$-translation), one can 
prove that
\begin{prop}\label{canonicalcontinuation}
    Any continuation map $f: CF^*(M, H_1) \to CF^*(M, H_2)$ enhances to a homomorphism $\mathbf{F}$ of $S^1$-complexes (which is in particular a quasi-isomorphism if $f$ is). Moreoever, this homomorphism is canonical up to homotopy, in the sense that any two homomorphisms $\mathbf{F}$ and $\mathbf{F}'$ enhancing $f$ constructed geometrically from parametrized continuation maps differ by an exact pre-morphism of $S^1$-complexes (also constructed geometrically). \qed
\end{prop} 
We omit the proof, which is standard (see e.g., \cite{Zhao:2019aa} 
though note some notational differences).  In particular, the $S^1$-complex defined
on the symplectic co-chain complex
$SC^*(M)$ or the Hamiltonian Floer complex (with small negative slope if $M$ is
non-compact) is an invariant of $M$, up to quasi-isomorphism.

\begin{rem}[Relation to earlier definitions in the literature]\label{earlierdefinitions}
In \cite{Bourgeois:2012fk} three different definitions of $S^1$-equivariant
symplectic cohomology are considered and shown to be equivalent; one of the definitions involves taking the $S^1$-equivariant homology associated to a certain
$S^1$-complex defined on $CF^*(M) = SC^*(M)$ (\cite{Bourgeois:2012fk}*{Prop. 2.19}). After normalizing for differing conventions (e.g., homological versus cohomological conventions for Floer theory, and the fact that their $u^{-1}$ is our $u$), it is direct to see that the $S^1$-complex constructed therein coincides up to equivalence with the one here (and even agrees on the chain level, seeing as the choices of Floer data chosen in that paper constitute a choice of Floer data for the $S^1$-action in our sense; compare e.g., \cite{Bourgeois:2012fk}*{Fig. 1} with \eqref{rotatecondition}).
\end{rem}

\subsection{The circle action on the interior}\label{sec:interior}
From the formal point of view of Floer homology of $M$ as the Morse homology of
an action functional on the free loop space $\mc{L} M$, one would expect
the contributions coming from constant loops to be acted on trivially by the
$C_{-*}(S^1)$ action (which comes from rotation of free loops).  This is indeed
the case, as we now review.

Suppose that the Hamiltonian $H_t$ defining $CF^*(M)$ is
chosen to be $C^2$-small, time-independent, and Morse in the compact region of $\bar{M}$ (which recall
equals $M$ if $M$ is compact). Then, Floer proved that all orbits of $H_t$
inside $\bar{M}$ are (constant orbits at) Morse critical points of $H$, and
all Floer cylinders between such orbits which remain in $\bar{M}$ are in fact
Morse trajectories of $H$ \cite{Floer:1989aa}. 

Let $C_{Morse}(H)$ denote the Morse complex of $H$. In the setting where $H$ is
as in \S \ref{subsubsec:relh} ($M$ can be Liouville or compact), all contributions to $CF^*(M)$ (both orbits and
cylinders) come from $\bar{M}$, so Floer's argument gives an isomorphisms
\begin{equation}
    C_{Morse}(H) \cong CF^*(M)
\end{equation}
In the setting where  $H$ is quadratic at infinity (and $M$ is Liouville) as in
\S \ref{subsubsec:sh}, one can ensure the collection of orbits coming from
$\bar{M}$ is an action-filtered subcomplex (and e.g., the integrated maximum principle will
ensure that all cylinders between such orbits lie in $\bar{M}$). Hence, there
is an inclusion of subcomplexes 
\begin{equation}\label{shsubcomplex}
    C_{Morse}(H) \to SC^*(M),
\end{equation}
which, under smallness constraints on the Floer data for the $S^1$-action gives an {\em $S^1$-subcomplex} \cite{Zhao:2019aa}*{Lemma 5.4} (meaning the operators $\delta_k$ preserve the subcomplex and in fact the action filtration), hence a morphism of $S^1$-complexes.
We will discuss both of the above cases at once: in either case by considering a Hamiltonian which is $C^2$-small on $\bar{M}$ we obtain an inclusion of $S^1$-subcomplexes
\begin{equation}\label{subcomplex}
    C_{Morse}(H) \to CF^*(M)
\end{equation}
with the understanding that in the former case this inclusion is the whole complex. 
\begin{lem}\label{lem:constantloops}
    There exists a choice of Floer data for the $S^1$-action so that
    $C_{Morse}(H)$ becomes a trivial $S^1$-subcomplex; meaning that
    the various operators $\delta_r$, $r \geq 1$, associated to the $C_{-*}(S^1)$
    action strictly vanish on the subcomplex. 
\end{lem}
\begin{proof}
    By the integrated maximum principle, any Floer trajectory with asymptotics along two generators
    in $C_{Morse}(H)$ remains in the interior of $\bar{M}$. We can choose the
    Hamiltonian term of our Floer data on $\mathcal{M}_r$ in this region of $M$ to be autonomous
    (i.e., $t$ and $s$-independent on the cylinder), $C^2$-small and Morse --- in fact equal to $H$;  
    then Floer's theorem \cite{Floer:1989aa} again guarantees that any Floer
    trajectory in $\mathcal{M}_r(x^-; x^+)$ between Morse critical points
    $x^{\pm}$ is in fact a Morse trajectory of $H$.
    It follows that for $x^+, x^-$ critical points of $H$, any
    element $u = (C, \vec{p})$ in the parametrized moduli space of maps $\overline{\mc{M}}_{r}(x^-;x^+)$
    solves an equation that is independent of the choice of parameter $\vec{p}
    \in (S^1)^r \times (0,1]^{r-1}$.  Namely, $u$ lives in a family of
    solutions of dimension at least $2r-1$ (given by varying $\vec{p}$), and
    hence $u$ cannot be rigid. The associated operation $\delta_r$, which counts
    rigid solutions, is therefore zero.  
\end{proof}
By invariance of the $S^1$-complex structure on $CF^*(M)$ (up to homotopically canonical quasi-isomorphism as in Proposition \ref{canonicalcontinuation}), we conclude
\begin{cor}\label{trivialconstantloops}
    For $M$ compact and admissible, or Liouville with $(H,J)$ as in \S \ref{subsubsec:relh}, $CF^*(M)$ is quasi-isomorphic to a trivial $S^1$-complex canonically up to homotopy. 
\end{cor}
\begin{cor}
    For $M$ Liouville with $(H,J)$ as in \S \ref{subsubsec:sh}, the inclusion chain map
    \begin{equation}
        C^*_{Morse}(M) \ra SC^*(M)
    \end{equation}
    lifts (cohomologically) canonically to a chain map
    \begin{equation}
        (C_{Morse}(H) [ [ u ] ], d_{Morse}) = (C_{Morse}(H))^{hS^1} \to (SC^{*}(M))^{hS^1} = (SC^*(M)[ [ u ] ], \delta_{eq}).
    \end{equation}
    inducing a cohomological map
    \begin{equation}\label{constantloopsmap}
        H^*(M) [ [ u ] ] \ra H^*(SC^{*}(M)^{hS1}).
    \end{equation}
\end{cor}
\begin{rem}
    Another possibly more direct way of producing the map $H^*(M) [ [ u ] ] \ra H^*(SC^{*}(M)^{hS1})$ is via an $S^1$-equivariant enhancement $\widetilde{PSS}$ of the PSS morphism $PSS: C^*(M) \to SC^*(M)$. We omit a further description here, and simply note that the resulting map can be shown to coincide cohomologically with the map defined above.
\end{rem}
Since the $S^1$-complex structure on $C^*_{Morse}(H)$ is trivial, one can (canonically) split the inclusion of homotopy fixed points map \eqref{inclusionhomotopyfixedpoints_explicit} $H^*(C^*_{Morse}(H)^{hS^1}) \to H^*(C^*_{Morse}(H))$ by the map
\begin{equation}
    H^*(M) \stackrel{[x \mapsto x \cdot 1]}{\ra} H^*(M) [ [ u ] ];
\end{equation}
the associated composition
\begin{equation}
        H^*(M) \stackrel{[x \mapsto x \cdot 1]}{\ra} H^*(M) [ [ u ] ] \ra H^*(SC^{*}(M)^{hS1}) \stackrel{[\iota]}{\ra} SH^*(M)
\end{equation}
coincides with the usual map $H^*(M) \ra SH^*(M)$. In particular, we note that
the homotopy fixed point complex of $SC^*(X)$ possesses a canonical (geometrically defined) cohomological element,
\begin{equation}\label{tilde1}
    \tilde{1} \in H^*(SC^{*}(M)^{hS^1}),
\end{equation}
lifting the usual unit $1 \in SH^*(M)$ (under the map $[\iota]$), defined as
the image of $1$ under \eqref{constantloopsmap}.

\section{Cyclic open-closed maps}\label{section:openclosed1}
\subsection{Open-closed Floer data}\label{floeropenclosed}
Here we review  the sort of Floer
perturbation data that needs to be specified on the domains appearing in
the open-closed map and their cyclic analogues. The main body of our treatment, following \S \ref{subsec:fukaya} consists of a (slightly modified) simplification of the setup from
\cite{Abouzaid:2010kx} tailored to the case of Fukaya categories of compact admissible $M$; in Remarks \ref{liouvilleFloerdatumOC} and \ref{wrappedFloerdatumOC} below we will indicate the modifications we need to make (following \cite{Abouzaid:2010kx} and building on Remarks \ref{admissibleH}, \ref{liouvilleFloerdatum}, \ref{wrappedFloerdatum}, and \ref{wrappedainf} above)  the case of compact Fukaya categories of Liouville manifolds (minor modifications) or wrapped Fukaya categories (slightly more involved modifications). There is one notable deviation from \cite{Abouzaid:2010kx} in that we allow our interior marked point to have a varying asymptotic marker and choose Floer data depending on this choice (as is done in constructions of BV-type operations in Hamiltonian Floer theory involving such asymptotic markers, see e.g., \cite{Seidel:2010uq, Seidel:2014aa}).

Let $S$ be a disc with $d$ boundary punctures
$z_1, \ldots, z_d$ (labeled in counterclockwise order) marked as positive and an interior marked point $p$ removed,
marked as either positive or negative; for the main body of the construction $p$ is negative. 
We also equip the
interior marked point $p$ with an {\it asymptotic marker}, that is a half-line
$\tau_p \in T_p S$ (or equivalently an element of the unit tangent bundle,
defined with respect to some metric). Call any such $S = (S, z_1, \ldots, z_d,
p, \tau_p)$ an {\em open-closed framed disc.}

In addition to the notation for semi-infinite strips \eqref{strip1} -
\eqref{strip2}, we use the following notation to refer to the positive and
negative semi-infinite cylinder: 
\begin{align} 
    A_+ &:= [0,\infty) \times S^1\\
    A_- &:= (-\infty,0] \times S^1
\end{align}
A {\em Floer datum} on a stable open-closed framed disc $S$ consists of the following choices on each component:
\begin{enumerate} 
    \item A collection of {\em strip-like or cylindrical ends} $\mathfrak{S}$
        around each boundary or interior marked point respectively of sign matching the sign of the marked point; strip-like
        ends were defined in \S \ref{subsec:fukaya} and a (positive resp.
        negative) cylindrical end is map 
        \begin{equation*}
                \begin{split}
            \delta^{\pm}_j: A_{\pm} &\ra S
        \end{split}
            \end{equation*}
            (so for the main body of the construction we use a negative cylindrical end around $p$).
            All of the
            strip-like ends around each of the $z_i$ should be positive, and
            all (strip-like or cylindrical) ends should have disjoint image in
            $S$. 
            The cylindrical end around $p$ should further should be  {\it
            compatible with the asymptotic marker}, meaning the points with
            angle zero should asymptotically approach the marker:
            \begin{equation}\label{compatiblewithmarker}
                \lim_{s \ra \pm \infty} \delta^{\pm}(s,0) = \tau_p.
            \end{equation}

        \item A one-form $\alpha_S$ on $S$, an $S$-dependent Hamiltonian function
            $H_S$ on $M$, and an $S$-dependent almost-complex structure $J_S$ on
            $M$, such that on each strip-like end these data pull back to a given fixed $(dt, H_t, J_t)$
            (which we used to define Lagrangian Floer homology chain-complexes) and on the cylindrical end this data pulls back to a given fixed $(dt, H_t^{cyl}, J_t^{cyl})$ which we used to define our Hamiltonian Floer homology chain complex (note: in many cases we could further simplify and choose $(H_t^{cyl}, J_t^{cyl}) = (H_t, J_t)$ given a sufficiently generic choice of $(H_t, J_t)$)

\end{enumerate}

Given a stable open-closed framed disc $S$ equipped with a Floer datum
$F_S$, a collection of Lagrangians $\{L_0, \ldots, L_{d-1}\}$, and asymptotics
$\{x_1, \ldots, x_d; y\}$ with $x_i$ a chord between $L_{i-1}$ and $L_{i \textrm{ mod }d}$, a map $u: S \ra M$ satisfies {\em Floer's
  equation for $F_S$ with boundary and asymptotics $\{L_0, \ldots,
L_{d-1}\}, \{x_1, \ldots, x_d; y\}$} if
\begin{equation}\label{floerequationOC}
        (du - X_S\otimes \alpha_S)^{0,1} = 0\textrm{ using the Floer data given by $F_S$} 
    \end{equation}
(meaning $X_S$ is the Hamiltonian vector field associated to $H_S$, and $0,1$ parts are taken with respect to $J_S$),
and
\begin{equation}\label{floerequationocasymptotics}
    \begin{cases}
        u(z) \in L_i & \textrm{ if $z \in \partial S$ lies counterclockwise from $z_i$ and clockwise from $z_{i+1\ \mathrm{mod}\ d}$ } \\
        \lim_{s \ra +\infty} u \circ \e^k(s,\cdot) = x_k &\\
        \lim_{s \ra \mp\infty} u \circ \delta(s, \cdot) = y &
    \end{cases}
\end{equation}
(here $\epsilon^k$ denotes the $k$th strip-like end, $\delta$ denotes the cylindrical end, and the sign $\mp$ in the last line is $-$ if $\delta$ is a negative end --- which is the case for the main body of the construction --- and $+$ if $\delta$ is a positive end).

\begin{rem}[Floer data for compact Lagrangians in Liouville manifolds]\label{liouvilleFloerdatumOC}
    If $M$ is Liouville and we are studying the Fukaya category of compact exact Lagrangians, then we take $H_t^{cyl}, J_t^{cyl}$ (the data required to define Floer cohomology) as in \S \ref{subsubsec:relh} and
    we again impose the additional requirements on Floer data described in Remark \ref{liouvilleFloerdatum}, 
    As before the $H_t^{cyl}, J_t^{cyl}$ and more restrictive types of Floer data chosen for wrapped Fukaya categories in Remark
    \ref{wrappedFloerdatumOC} below would also work. The ability to choose $H_t^{cyl}$ and $J_t^{cyl}$ as in \S \ref{subsubsec:relh} is indicative of a more general freedom in the Floer data here, which also will allow us later to define operations in which the interior marked point (and all boundary marked points) are positive (see \S \ref{sec:compactopenclosed}).
\end{rem}

\begin{rem}[Floer data and Floer's equation for wrapped Fukaya categories] \label{wrappedFloerdatumOC}
    Almost exactly as in Remark \ref{wrappedFloerdatum}, and following \cite{Abouzaid:2010kx} in order to associate
    operations between the wrapped Fukaya category and symplectic cohomology,
    one needs to make the following modifications to the notion of Floer data.
    First, one takes 
    $H_t^{cyl}, J_t^{cyl}$ to be the data defining the symplectic co-chain complex as in \S \ref{subsubsec:sh}. 
    Then one equips $S$ with strip-like and cylindrical ends as above. Let $\psi^{\rho}$ as before denote the time $\log(\rho)$ Liouville flow on $M$. The modifications to the Floer data are:
    \begin{itemize}
        \item {\em Extra choices} of weights and time-shifting maps: Exactly as in Remark \ref{wrappedFloerdatum}, one associates  a {\em weight} $w_k \in \R_{>0}$ to each boundary or interior marked point and a time-shifting map $\rho_S: \partial S \to \R_{>0}$ agreeing with $w_k$ near the $k$th strip-like end. 

        \item {\em Modified requirements on one-form}: The 1-form $\alpha_S$ should be subclosed meaning $d \alpha_S \leq 0$, restrict to 0 along $\partial S$, and restrict to $w_k dt$ on each (strip-like or cylindrical) end (as in item \ref{wrappedoneform} of Remark \ref{wrappedFloerdatum}). It follows by Stokes' theorem that the weight at the (output) cylindrical end should be greater than the sum of weights over all (input) strip-like ends. In particular, it is not possible for $\alpha_S$ to be subclosed and restrict to 0 along $\partial S$, conditions necessary to appeal to the integrated maximum principle if the interior marked point were also positive. (This is a reflection of the fact that wrapped Fukaya categories do not admit geometric operations with no outputs.)

        \item {\em Modified requirements on Hamiltonians, as in item \ref{wrappedhamiltonian} of Remark \ref{wrappedFloerdatum}}: The Hamiltonian term should pull back to 
$\frac{H \circ \psi^{w_k}}{w_k^2}$ along any strip-like end and to 
$\frac{H^{cyl} \circ \psi^{w_k}}{w_k^2}$ along the cylindrical end. The Hamiltonian term should also be quadratic at infinitely many levels of \eqref{cylend} tending to infinity (this is a slight weakening of Remark \ref{wrappedFloerdatum} coming from the fact that the Hamiltonian used to define $SC^*(X)$ is not quadratic at every level near infinity due to \eqref{s1perturbation}).

\item {\em Modified requirements on almost complex structures, as in item
    \ref{wrappedalmostcomplexstructure} of Remark \ref{wrappedFloerdatum}} The
    almost-complex structure should be contact-type at infinity and pull back
    to $(\psi^{w_k})^*J_t$ along each strip-like end and
    $(\psi^{w_k})^*J_t^{cyl}$ along the cylindrical end.
\end{itemize}
Exactly as in Remark \ref{wrappedFloerdatum}, there is a rescaling action on the space of such Floer data, and we will relax any consistency requirement imposed on Floer data to allow for an arbitrary rescaling when equating different choices of Floer data. Finally, we note the slight modifications to the boundary and asymptotic conditions of Floer's equation \eqref{floerequationocasymptotics}, following Remark \ref{wrappedainf}: on the boundary component of $\partial S$ lying counterclockwise from $z_i$ and clockwise from $z_{i+1\textrm{ mod }d}$ we impose the moving boundary condition
$u(z) \in
(\psi^{\rho_S(z)})^* L_i$,
on the $k$th strip-like end we impose $\lim_{s\ra + \infty} u \circ \e^k(s,\cdot) = (\psi^{w_k})^* x_k$, and on the cylindrical end, we impose $\lim_{s\ra -\infty} u \circ \delta(s,\cdot) = (\psi^{w})^* y$ where $w$ is the weight associated to the interior puncture $p$.
\end{rem}
Exactly as in the proof of Lemma \ref{smoothcompactcor}, the constraints to Floer data
in the Liouville case made in the above two remarks help ensure Assumption
\ref{mainassumption} holds for associated moduli spaces.

\subsection{Non-unital open-closed maps}\label{sec:nonunital}
We begin by constructing a variant of the open-closed map of
\cite{Abouzaid:2010kx} with source the non-unital Hochschild complex of
\eqref{nonunitalcomplex}, which we call the {\it non-unital open-closed map} (and indicate by $\oc$ or $\oc^{nu}$)
\begin{equation}
    \oc:= \oc^{nu}: \r{CH}^{nu}_{*-n}(\f) \lra CF^*(M).
\end{equation}
This map actually has a straightforward explanation from the perspective of Remark
\ref{rem:unitinsertion}: we define the map $\oc$ from $\widetilde{\r{CH}}_*(\f)$ by counting discs with an arbitrary number of boundary punctured and one interior puncture asymptotic to an orbit (as in \cite{Abouzaid:2010kx} with the proviso that we treat the formal elements $e_L^+$ as ``fundamental class $[L]$ point constraints (i.e., empty constraints)'': we fill back in the relevant boundary puncture and impose no constraints on that marked point.
With respect to the decomposition \eqref{nonunitalcomplex}, we define a pair of maps
\begin{equation}
    \check{\oc} \oplus \hat{\oc}: \r{CH}_*(\f) \oplus \r{CH}_*(\f)[1] \ra CF^*(M)
\end{equation}
giving the left and right component of the non-unital open-closed map
\begin{equation}
    \begin{split}
        \oc: \r{CH}^{nu}_{*-n}(\f) &\lra SC(M),\\
        (x,y) &\longmapsto \check{\oc}(x) + \hat{\oc}(y).
    \end{split}
\end{equation}
Since the left (check) factor is equal to the usual cyclic bar complex for
Hochschild homology, $\check{\oc}$ will be defined exactly as the open-closed map is defined in
\cite{Abouzaid:2010kx} (briefly recalled below), and the new
content is the map $\hat{\oc}$. 
We will define $\hat{\oc}$ below (and recall the definition of $\check{\oc}$)
and prove, extending \cite{Abouzaid:2010kx} that
\begin{lem}\label{ocnuchain}
    $\oc$ is a chain map of degree $n$.
\end{lem}
We note a notational difference from \cite{Abouzaid:2010kx}, which uses
$\oc$ to refer to what we call here $\check{\oc}$; in contrast in this paper we use
$\oc$ exclusively to refer to the (non-unital) open-closed map $\oc = \oc^{nu}:=
\check{\oc} \oplus \hat{\oc}$ with domain the non-unital Hochschild complex.
Of course, the two maps $\oc$ and $\check{\oc}$ are homologically the same.
That is, assuming Lemma \ref{ocnuchain},
\begin{cor}\label{homologylevelnonunitaloc}
    As homology level maps, $[\oc] = [\check{\oc}]$.
\end{cor}
\begin{proof}
    By construction, the chain level map $\check{\oc}$ constructed in \cite{Abouzaid:2010kx} factors as 
   \begin{equation}
       \r{CH}_{*-n}(\f) \subset \r{CH}_{*-n}^{nu}(\f) \stackrel{\oc}{\lra} CF^*(M).
   \end{equation}
   The first inclusion is a quasi-isomorphism by Lemma \ref{inclusionquasi},
   since $\f$ is known to be cohomologically unital.
\end{proof}
The moduli space controlling the operation $\check{\oc}$, denoted
\begin{equation}\label{opencloseddm}
    \overline{\check{\mc{R}}}_d^{1}
\end{equation}
is the (Deligne-Mumford compactification of the)
abstract moduli space of discs with $d$ boundary positive punctures $z_1,
\ldots, z_d$ labeled in counterclockwise order and 1 interior negative
puncture $z_{out}$, with an asymptotic marker $\tau_{out}$ at $z_{out}$ pointing
towards $z_d$. The space \eqref{opencloseddm} has a manifold with corners
structure, with boundary strata described in \cite{Abouzaid:2010kx}*{\S C.3}
(there, the space is called $\overline{\mc{R}}_d^1$)--in short, codimension one
strata consist of disc bubbles containing any cyclic subsequence of $k$ inputs
attached to an element of $\check{\mc{R}}_{d-k+1}^1$ at the relative position
of this cyclic subsequence.
Orient the top stratum $\check{\mc{R}}_d^1$ by trivializing it, 
sending $[S]$ to the unit disc representative $S$ with $z_d$ and $z_{out}$
fixed at 1 and 0, and taking the orientation induced by the (angular) positions
of the remaining marked points:
\begin{equation}\label{openclosedcheckorientation}
    -dz_1 \wedge \cdots \wedge dz_{d-1}.
\end{equation}
\begin{figure}[h]
    \caption{\label{fig:oc_check}A representative of an element of the moduli space $\check{\mc{R}}^1_{4}$ with special points at 0 (output), $-i$. }
    \centering
    \includegraphics[scale=0.7]{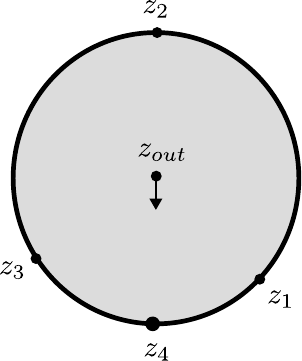}
\end{figure}

The moduli space controlling the new map $\hat{\oc}$ is nearly identical to
$\check{\mc{R}}_d^1$,
but there is additional freedom in the direction of the asymptotic marker at the 
interior puncture $z_{out}$. 
The top (open) stratum is easiest to define: let
\begin{equation}\label{freemarker}
    \mc{R}_d^{1,free}
\end{equation}
be the moduli space of discs with $d$ positive boundary punctures and one
interior negative puncture as in $\check{\mc{R}}_d^1$, but with the asymptotic
marker $\tau_{out}$
pointing anywhere between $z_1$ and $z_d$. 
\begin{rem}\label{manifoldcorners}
There is a delicate point in naively compactifying $\mc{R}_d^{1,free}$:
on any formerly codimension 1 stratum in which $z_1$ and $z_d$ bubble off, the
position of $\tau_{out}$ becomes fixed too, and so the relevant stratum actually
should have codimension 2 (and hence does not contribute to the codimension-1
boundary equation for $\hat{\oc}$. Moreoever, there is no nice corner chart
near this stratum).  For technical convenience, we pass to an alternate, larger
(blown-up) model for the compactification in which these strata have
codimension 1 but consist of degenerate contributions. 
\end{rem}
In light of Remark \ref{manifoldcorners}, we use \eqref{freemarker} as
motivation and instead define 
\begin{equation}\label{hatmodulispace}
    \hat{\mc{R}}_{d}^1
\end{equation}
to be the abstract moduli space of discs with $d+1$ boundary punctures $z_f$,
$z_1$, \ldots, $z_d$ and an interior puncture $z_{out}$ with asymptotic marker $\tau_{out}$
pointing towards the boundary point $z_f$, modulo automorphism. 
We mark $z_f$ as ``auxiliary,'' but otherwise the space is abstractly
isomorphic to $\check{\mc{R}}^1_{d+1}$.
Identifying $\hat{\mc{R}}_d^1$ with the space of unit discs with  $z_{out}$ and
$z_f$ fixed at 1 and 0, the remaining (angular) positions of $z_1, \ldots, z_d$
determine an orientation 
\begin{equation}
    \label{openclosedhatorientation}
    -dz_1 \wedge \cdots \wedge dz_d.
\end{equation}
\begin{figure}[h]
    \caption{A representative of an element of the moduli space $\mc{R}^1_{4,free}$ and the corresponding element of $\hat{\mc{R}}_4^1$.\label{fig:oc_hat} }
    \centering
    \includegraphics[scale=0.7]{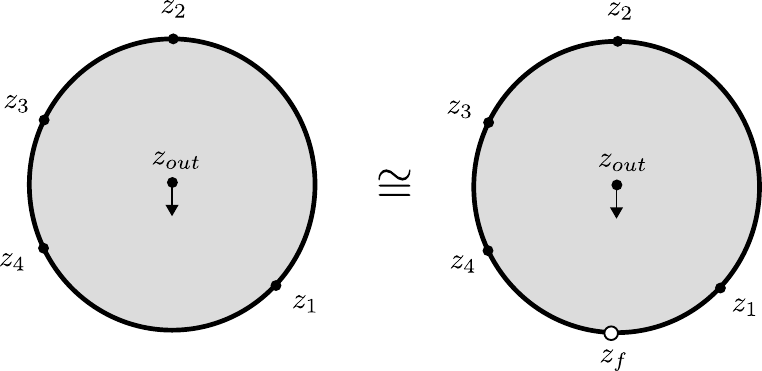}
\end{figure}

The {\it forgetful map}
\begin{equation}
    \pi_f: \hat{\mc{R}}_d^1 \ra \mc{R}_d^{1,free}
\end{equation}
puts back in the point $z_f$ and forgets it. Since the point $z_f$ is
recoverable from the direction of the asymptotic marker at $z_{out}$, 
\begin{lem}
$\pi_f$ is a diffeomorphism.\qed
\end{lem}
The perspective of the former space \eqref{hatmodulispace} gives us a model for
the compactification 
\begin{equation} 
    \label{freecompactification}
\overline{\mc{R}}_d^{1,free}
\end{equation}
as the ordinary Deligne-Mumford compactification 
\begin{equation}\label{dmcompactificationocforgotten}
    \overline{\hat{\mc{R}}}_{d}^1.
\end{equation}
We call a component $T$ of a representative $S$ of
\eqref{dmcompactificationocforgotten}
the {\it main component} if it contains the interior
marked point, and a {\it secondary component} if its output is attached to
the main component.  
As a manifold with corners, \eqref{dmcompactificationocforgotten} is equal to
the compactification $\overline{\check{\mc{R}}}_{d+1}^1$ except
from the point of view of assigning Floer datum, as we will be forgetting the point $z_f$
instead of fixing asymptotics for it. It is convenient therefore (for the
purpose of indicating choices of Floer data made) to name components of strata
containing $z_f$ differently. At any stratum: 
\begin{itemize}
        \item we treat the main component (containing $z_{out}$ and $k$
            boundary marked points) as belonging to
            $\overline{\hat{\mc{R}}}_{k-1}^1 $ if it contains $z_f$ and
            $\overline{\check{\mc{R}}}_{k}^1$ otherwise; and

        \item 
            If the $i$th boundary marked point of any non-main component was
            $z_f$, we view it as an element of 
            $\mc{R}^{k,f_i}$, the space of discs with 1 output and $k$ input
            marked points removed from the boundary, with the $i$th point
            marked as ``forgotten,'' constructed in Appendix
            \ref{discsforgotten}.

        \item We treat any other non-main component as belonging to $\mc{R}^k$ as
            usual.  
    \end{itemize}
Thus, the codimension-1 boundary of the Deligne-Mumford compactification is
covered by the natural inclusions of the following strata
\begin{align}
    \label{hatocstrata1} \overline{\mc{R}}^m &\times_i \overline{\hat{\mc{R}}}_{d-m+1}^1\ \ \ 1 \leq i < d-m+1\\
    \label{hatocstrata2} \overline{\mc{R}}^{m,f_k}&\times_{d-m+1} \overline{\check{\mc{R}}}^1_{d-m+1} \ \ \ 1 \leq j \leq m,\ 1 \leq k \leq m
\end{align}
where the notation $\times_j$ means that the output of the first component is
identified with the $j$th boundary input of the second. See Figure \ref{fig:oc_hat_degenerations}.
\begin{figure}[h]
    \caption{\label{fig:oc_hat_degenerations} A schematic of the two distinct types of codimension-1 boundary strata of \eqref{dmcompactificationocforgotten} in codimension 1. On the left side, corresponding to \eqref{hatocstrata1}, a disc bubble forms involving any collection of boundary marked points not including $z_f$. On the right side, corresponding to \eqref{hatocstrata2}, a disc bubble forms involving the point $z_f$.}
    \centering
    \includegraphics[scale=0.7]{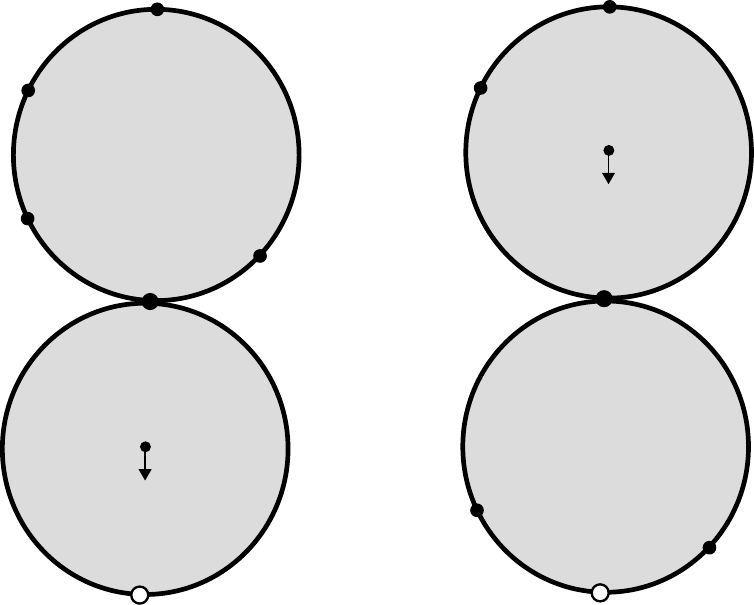}
\end{figure}

The forgetful map $\pi_f$ extends to a map 
$\overline{\pi}_{f}$ from the compactification $\overline{\hat{\mc{R}}}_{d}^1$ (to the space of stable framed open-closed discs with $d$ marked points)
as follows:  $\overline{\pi}_f$ puts the auxiliary point $z_f$
back in, eliminates any component which is not main or secondary which has only
one non-auxiliary marked point $p$, and labels the positive marked point below
this component by $p$. Given a representative $S$ of
$\overline{\hat{\mc{R}}}_d^1$, we call $\overline{\pi}_f(S)$ the {\it
associated reduced surface}. We will study maps from the associated reduced
surfaces $\overline{\pi}_f(S)$, parametrized by $S$. To this end, we define a {\em
Floer datum} on a stable disc $S$ in $\overline{\hat{\mc{R}}}_d^1$ to consist
of a Floer datum for the underlying reduced surface $\overline{\pi}_f(S)$ in
the sense of \S \ref{floeropenclosed}.

First, in Appendix \ref{discsforgotten} we describe an inductive construction
of Floer data for (the underlying reduced surfaces of) the compactified moduli
space of discs with a forgotten point $\overline{\mc{R}}^{d,f_i}$, for every
$d$ and $i$ satisfying: 
\begin{equation}\label{forgottencondition}
        \begin{split}
            &\textrm{For $d>2$, the choice of Floer datum on $\mc{R}^{d,f_i}$
                should be pulled back from the forgetful map $\mc{R}^{d,f_i} \ra \mc{R}^{d-1}$. }\\
                    &\textrm{For $d=2$, the Floer datum on the surface $S$ (with $z_i$ forgotten) should be translation invariant.}
\end{split}
\end{equation}

Next, we choose a \emph{Floer datum
    for the non-unital open-closed map} which is an inductive set of choices 
    $(\mathbf{D}_{\check{\oc}}, \mathbf{D}_{\hat{\oc}})$, for each $d \geq 1$ and every 
    representative $S \in \overline{\check{\mc{R}}}_d^1$, $T \in
    \overline{\hat{\mc{R}}}_d^1$, of a Floer datum for $S$ and (the associated
    reduced surface of) $T$  respectively. As usual, these choices should be
    smoothly varying, and restrict smoothly to previously chosen Floer data on
    boundary strata. Note that for a given $d$ the boundary strata have
    components that are either $\overline{\check{\mc{R}}}_{d'}^1$ or
    $\overline{\hat{\mc{R}}}_{d'}^1$ for $d' < d$, a stratum $\mc{R}^{d'}$
    (over which we have chosen a Floer datum for the $\ainf$ structure), or a
    stratum $\mc{R}^{d,f_i}$ where we have chosen a Floer datum in Appendix
    \ref{discsforgotten} as described above. (As usual for Liouville manifolds 
    we use the notion of Floer data and consistency described in
Remarks \ref{liouvilleFloerdatumOC} or \ref{wrappedFloerdatumOC} in the wrapped case.) Contractibility of the space of choices at
    every stage (and consistency of the compatibility conditions imposed at
    corners) ensures as usual that a Floer datum for the non-unital open-closed map
    exists.

Fixing such a choice, we obtain, for any
$d$-tuple of Lagrangians $L_0, \ldots, L_{d-1}$, and asympotic conditions
$\vec{x} = (x_{d}, \ldots, x_1),\ x_i \in \chi(L_{i-1},L_{i \textrm{ mod }d})$, $y_{out} \in \mc{O}$
a pair of moduli spaces
\begin{align}
    \label{openclosedmaps} &\check{\mc{R}}_d^1(y_{out}; \vec{x})\\
    \label{hatopenclosedmaps}&\hat{\mc{R}}_d^1(y_{out}; \vec{x}),
\end{align}
of parametrized families of solutions to Floer's equation
\begin{align}
    \{(S,u)| S \in \check{\mc{R}}_d^1: u: S \ra M, (du - X\otimes \alpha)^{0,1} = 0\textrm{ using the Floer datum given by $\mathbf{D}_{\check{\oc}}(S)$} \}\\
    \{ (S, u) | S \in \hat{\mc{R}}_d^1, u:\pi_f(S) \ra M |   (du - X \otimes \alpha)^{0,1} = 0 \textrm{ using the Floer datum given by $\mathbf{D}_{\hat{\oc}}(S)$} \}
\end{align}
satisfying asymptotic and boundary conditions (in either case) as in \eqref{floerequationocasymptotics} (with the modification for wrapped Fukaya categories involving Liouville rescalings described in Remark \ref{wrappedFloerdatumOC}).
The expected dimension of every component of \eqref{openclosedmaps} and \eqref{hatopenclosedmaps} respectively, is (in the $\Z$-graded case), or agrees mod 2 with (in the $\Z/2$-graded case):
\begin{align}
    \deg(y_{out}) - n + d -1 - \sum_{k=1}^{d} \deg(x_k);\\
    \deg(y_{out}) - n + d - \sum_{k=1}^{d} \deg(x_k).
\end{align}
As usual there are Gromov-type bordifications
\begin{align}
    \label{checkOCmaps}&\overline{\check{\mc{R}}}_d^1(y_{out}; \vec{x});\\
    \label{hatOCmaps}&\overline{\hat{\mc{R}}}_d^1(y_{out}; \vec{x}),
\end{align}
which allow semi-stable breakings, as well as maps from strata corresponding to
the boundary strata of $\overline{\check{\mc{R}}}_d^1$ and
$\overline{\hat{\mc{R}}}_d^1$.

By Assumption \ref{mainassumption},
for generic choices of Floer datum for the non-unital open-closed map, the components of \eqref{checkOCmaps} and \eqref{hatOCmaps} of virtual dimension $\leq 1$ are compact manifolds-with-boundary of dimension agreeing with virtual dimension.
Fix such a Floer datum. At a rigid element $u$ of each of the above moduli spaces, we obtain,
using the fixed orientations of moduli spaces of domains
\eqref{openclosedcheckorientation}-\eqref{openclosedhatorientation} and
\cite{Abouzaid:2010kx}*{Lemma C.4}, isomorphisms of orientation lines
\begin{align}
    (\check{\mc{R}}_d^1)_u: o_{x_{d}} \otimes \cdots \otimes o_{x_1} \ra o_{y_{out}}\\
    (\hat{\mc{R}}_d^1)_u: o_{x_{d}} \otimes \cdots \otimes o_{x_1} \ra o_{y_{out}}.
\end{align}
These isomorphisms in turn define the $|o_{y_{out}}|_{\K}$ component of the
check and hat components of the non-unital open-closed map with $d$ inputs in
the lines $|o_{x_d}|_{\K}, \ldots, |o_{x_1}|_{\K}$, up to a sign twist:
\begin{equation}\label{checkocdefinition}
    \begin{split}
        \check{\oc}_d &([x_d],  \ldots, [x_1]) := \\
    &
    \sum_{u \in \overline{\check{\mc{R}}}^d_1(y; x_d, \ldots, x_1)\textrm{ rigid}} (-1)^{\check{\star}_d} (\check{\mc{R}}_d^1)_u([x_d], \ldots, [x_1]),\\
    \check{\star}_d &:=\deg(x_d) + \sum_{k=1}^d k \deg(x_k)
\end{split}
\end{equation}

\begin{equation}\label{hatocdefinition}
    \begin{split}
    \hat{\oc}_d([x_{d}], \ldots, [x_1]) &:= 
    \sum_{u \in  \overline{\hat{\mc{R}}}_d^1(y_{out}; \vec{x}) \textrm{ rigid}}  (-1)^{\hat{\star}_d} (\hat{\mc{R}}_d^1)_u([x_d], \ldots, [x_1]),\\
\hat{\star}_d &:= \sum_{i=1}^d i \cdot \deg(x_i).
\end{split}
\end{equation}

By analyzing the boundary of one-dimensional components of the moduli spaces
$\overline{\check{\mc{R}}}_d^1(y_{out}; \vec{x})$, the consistency condition
imposed on Floer data, and a sign analysis, it was proven in
\cite{Abouzaid:2010kx} that 
\begin{lem}[\cite{Abouzaid:2010kx}, Lemma 5.4]
\label{occhain}
$\oc:=\check{\oc}$ is a chain map of degree $n$; that is $(-1)^n d_{CF} \circ \check{\oc} = \check{\oc} \circ b$.\qed
\end{lem}
Similarly, we prove the following, completing the proof of Lemma
\ref{ocnuchain}: 
\begin{lem}\label{ochatchainold}
    The following equation holds:
    \begin{equation}
        (-1)^n d_{CF} \circ \hat{\oc} = \check{\oc} \circ d_{\wedge \vee} +
        \hat{\oc} \circ b'
    \end{equation}
\end{lem}
\begin{proof}
The consistency condition imposed on Floer data implies that the boundary of the 1-dimensional
components of $\overline{\check{\mc{R}}}_d^1(y;
\vec{x})$ are covered by the images of the natural inclusions of the rigid
(zero-dimensional) components of the moduli spaces of maps coming from the
boundary strata \eqref{hatocstrata1}, \eqref{hatocstrata2} along with (the rigid components of) semi-stable breakings:
\begin{align}
    \overline{\hat{\mc{R}}}_d^1(y_1; \vec{x}) \times \overline{\mc{M}}(y_{out}; y_1) &\ra \partial 
    \overline{\check{\mc{R}}}_d^1(y_{out}; \vec{x}) \label{semistablehat1}\\ 
    \overline{\mc{R}}^1(x; x_i) \times \overline{\hat{\mc{R}}}_d^1(y_{out}; \tilde{\vec{x}}) &\ra \partial 
    \overline{\check{\mc{R}}}_d^1(y_{out}; \vec{x}),\label{semistablehat2}
\end{align}
where $\tilde{\vec{x}}$ denotes the collection of inputs $\vec{x}$ with $x_i$ replaced with $x$.
Let $\mu^{d,i}$ be the operation associated to the space of discs with $i$th
point marked as forgotten $\mc{R}^{d,f_i}$, which is described in detail in
Appendix \ref{discsforgotten}. $\mu^{d,i}$ takes a composable sequence of $d-1$
inputs, separated into an $i-1$ tuple and a $d-i$ tuple; in line with Remark
\ref{rem:unitinsertion} we will use the
suggestive notation
\begin{equation} \mu^d(x_d, \ldots, x_{i+1}, e^+, x_{i-1},
    \ldots, x_1):= \mu^{d,i}(x_d, \ldots, x_{i+1}; x_{i-1}, \ldots, x_1)\text{\footnotemark}
\end{equation}
\footnotetext{In fact, when the Fukaya category is equipped with homotopy units,
    one can ensure that there is a strict unit element $e^+$ in each self-hom
    space, for which $\mu^k$ with an $e^+$ element admits a geometric description as above.
See e.g., \cite{Fukaya:2009qf} or \cite{ganatra1_arxiv}.}(recall the abuse of notation $x_i:=[x_i]$).
Then, up to sign, by the standard codimension-1 boundary principle for
Floer-theoretic operations, we have shown that
\begin{equation} \label{hatocchainequation}
    \begin{split} 
        0 = &d_{CF} \hat{\oc}(x_d, \ldots, x_1) - \sum_{i,j} (-1)^{\maltese_1^i} \hat{\oc} (x_d,  \ldots,   x_{i+j+1}, \mu^j(x_{i+j}, \ldots, x_{i+1}), x_i, \cdots,  x_1)\\
        &- \sum_{i,j,k} (-1)^{\sharp_j^k} \check{\oc} ( \mu^{j + k + 1}(x_{j}, \ldots, x_1, e^+, x_d, \ldots, x_{d-k+1}),  x_{d-k},  \ldots,  x_{j+1})
\end{split}
\end{equation}
with desired signs
\begin{align}
    \label{desiredsign1}\maltese_m^n &= \sum_{j=m}^n ||x_i||\\
    \label{desiredsign2}\sharp_j^k &=  \maltese_1^j \maltese_{j+1}^{d}  + \maltese_{j+1}^{d} + 1.
\end{align}
However, as shown in Appendix \ref{discsforgotten}, 
\begin{equation}
    \mu^{j + k + 1}(x_{j}, \ldots, x_1, e^+, x_d, \ldots, x_{d-k+1}) = \begin{cases}
        x_1 & j = 1,\ k = 0\\
        (-1)^{|x_d|} x_d & j = 0, k = 1\\
        0 & \textrm{otherwise}
    \end{cases}
\end{equation}
(in this manner, $e^+$, though a formal element, behaves as a strict unit). So if \eqref{hatocchainequation} held, 
it would follow that 
\begin{equation}
        \begin{split}
        \label{lastsumhatoc}
        d_{CF} \circ \hat{\oc}(x_d \otimes \cdots \otimes x_1) &=(-1)^{||x_1|| \maltese_2^d + \maltese_2^d + 1} \check{\oc}(x_1 \otimes x_d \otimes  \cdots \otimes x_2) \\
            &+ (-1)^{|x_d| + \maltese_1^d + 1} \check{\oc}(x_d \otimes \cdots \otimes x_1) +\hat{\oc} \circ b' (x_d \otimes \cdots \otimes x_1)\\
            &=  \check{\oc} ( (-1)^{\maltese_1^d + ||x_d||}(1-t) (x_d \otimes \cdots \otimes x_1)) + \hat{\oc} \circ b'(x_d \otimes \cdots \otimes x_1).\\
            &= \left(\check{\oc} \circ d_{\wedge \vee} + \hat{\oc} \circ b'\right) (x_d \otimes \cdots \otimes x_1).
        \end{split}
    \end{equation} 
So we are done if we establish the signs are exactly
\eqref{desiredsign1}-\eqref{desiredsign2}.

Using the notation
\begin{equation}
    \oc(e^+ \otimes x_d \otimes \cdots \otimes x_1) := \hat{\oc}(x_d \otimes \cdots \otimes x_1),
\end{equation}
where again $e^+$ is simply a formal symbol referring to the position of the
auxiliary (forgotten) input point, we observe that the equation
\eqref{hatocchainequation} is exactly the equation for $\oc$ being a chain map
on inputs of the form $(e^+ \otimes x_d \otimes \cdots \otimes x_1)$ (where we
treat an ``$e^+$'' input as an auxiliary unconstrained point on our domain).
The sign verification therefore follows from that of $\check{\oc}$ being a
chain map (in \cite{Abouzaid:2010kx}*{Lemma 5.4}), for we have used identical
orientations on the abstract moduli space $\hat{\mc{R}}_d^1$ as on
$\check{\mc{R}}_{d+1}^1$ (identifying $z_f$ with $z_{d+1}$), and on $\mc{R}^{d,f_i}$ as on $\mc{R}^d$, and we can
even insert a formal degree zero orientation line $o_{e^+}$ into the procedure
for orienting moduli spaces of open-closed maps (see \cite{Abouzaid:2010kx}*{\S
C.6}), corresponding to the marked point (obtained by filling in) $z_f$.  Note that $o_{e^+}$, being of
degree zero, commutes with everything, and is just used as a placeholder as if
we had an asymptotic condition at $z_f$. 
\end{proof}

\begin{proof}[Proof of Lemma \ref{ocnuchain}]
    Given that $\check{\oc}$ is already known to be a chain map by
    \cite{Abouzaid:2010kx}*{Lemma 5.4}, repeated as Lemma \ref{occhain}
    above, the new part to check is that the
    $(-1)^n d_{CF} \circ \hat{\oc} = \check{\oc} d_{\wedge \vee} + \hat{\oc} \circ b'$.
    This is the content of Lemma \ref{ochatchainold} above.
\end{proof}

\subsection{An auxiliary operation}\label{ocs1sec}
It will be technically convenient to define an auxiliary operation
\begin{equation}\label{ocS1}
    \oc^{S^1}: \r{CH}_{*-n}(\f) \ra CH^{*+1}(M)
\end{equation}
from the left factor of the non-unital Hochschild complex to Floer co-chains,
in which the asymptotic marker $\tau_{out}$ varies freely around the circle. This operation
is more easily comparable to the BV operator on Floer cohomology, and moreover,
we will show that $\oc^{S^1}$ (and $\hat{\oc}$) can be chosen to satisfy the
following crucial identity:
\begin{prop}\label{ocS1Brelation}
    There is an equality of chain level operations: 
    \begin{equation}
        \oc^{S^1} = \hat{\oc} \circ B^{nu}.
    \end{equation}
\end{prop}

To define \eqref{ocS1}, let
\begin{equation}\label{S1modulispace}
    \mc{R}^{S^1}_d
\end{equation}
be the abstract moduli space of discs with $d$ boundary positive punctures $z_1,
\ldots, z_d$ labeled in counterclockwise order and 1 interior negative puncture
$z_{out}$, with an asymptotic marker $\tau_{out}$ at $z_{out}$ (or choice of
real half line in $T_{z_{out}} D$) which is free to vary. 
Equivalently,
\begin{equation}
    \begin{split}
&\textrm{\eqref{S1modulispace} is the space of discs with $z_1, \ldots, z_d$ and
$z_{out}$ as before, 
 and an extra auxiliary }\\
 &\textrm{interior marked point $p_1$ such
    that, for a representative with $(z_{out}, z_1)$ fixed at $(0,-i)$,  }\\
&\textrm{$|p_1| = \frac{1}{2}$, and the asymptotic marker $\tau_{out}$ points
towards $p_1$.}
\end{split}
\end{equation}

By using a representative with fixed $(z_{out}, z_1)$ as above, the argument of $p_1$ produces an abstract identification
\begin{equation} \label{uncompactifieds1}
    \mc{R}^{S^1}_d =  \check{\mc{R}}^1_d \times S^1;
\end{equation}
using this identification, fix an orientation of \eqref{uncompactifieds1} given by negative the product orientation of \eqref{openclosedcheckorientation} with the standard counterclockwise orientation on $S^1$.
The Deligne-Mumford type compactification can thus be thought of as
\begin{equation} \label{compactifications1}
    \overline{\mc{R}_d^{S^1}} =  \overline{\check{\mc{R}}}^1_d \times S^1
\end{equation}
Given an element $S$ of $\mc{R}_d^{S^1}$ and a choice of marked point $z_i$
on the boundary of $S$, we say say that $\tau_{out}$ {\em points at $z_i$}, if,
when $S$ is reparametrized so that $z_1$ fixed at $-i$ and $z_{out}$ fixed at
0, the vector $\tau_{out}$ is tangent to the straight line from $z_{out}$ to
$z_i$. 
Equivalently, for this representative, $z_{out}$, $p_1$, and $z_i$ are
collinear. For each $i$, the locus where $\tau_{out}$ points at $z_i$ forms a
codimension 1 submanifold, denoted
\begin{equation}
    \mc{R}_d^{S^1_i}.
\end{equation}
The notion compactifies well; if $z_i$ is not on the main component of
\eqref{compactifications1} we say $\tau_{out}$ {\em points at $z_i$} if it
points at the root of the bubble tree $z_i$ is on. This compactified locus
$\overline{\mc{R}_d^{S^1_i}}$ can be identified with
$\overline{\mc{R}}^1_d$ via the map
\begin{equation}\label{pointatimap}
    \tau_i: \overline{\mc{R}_d^{S^1_i}} \ra \overline{\mc{R}}^1_d
\end{equation}
which cyclically permutes the labels of the boundary marked points so that
$z_i$ is now labeled $z_d$. 

In a similar fashion, we have an invariant notion of what it means for
$\tau_{out}$ to point {\it between $z_i$ and $z_{i+1}$}; this is a codimension
0 submanifold with corners of \eqref{uncompactifieds1}, denoted
\begin{equation}
\label{pointii1}
\mc{R}_d^{S^1_{i,i+1}}.
\end{equation}
The compactification
has some components that are codimension 1 submanifolds with corners of
\eqref{compactifications1}, when $z_i$ and $z_{i+1}$ both lie on a bubble tree.

Finally, there is a {\em free} $\Z_{d}$ action generated by the map
\begin{equation} \label{labelpermute}
    \kappa: \overline{\mc{R}_d^{S^1}} \ra \overline{\mc{R}_d^{S^1}}
\end{equation}
which cyclically permutes the labels of the boundary marked points; for
concreteness, $\kappa$ changes the label $z_i$ to $z_{i+1}$ for $i < d$, and
$z_d$ to $z_1$. Note that if, on a given $S$, $\tau_{out}$ points between $z_i$
and $z_{i+1}$, then on $\kappa(S)$, $\tau_{out}$ points between $z_{i+1 \rm{\ mod \ }d}$
and $z_{i+2 \rm{\ mod \ }d}$. 
\begin{lem}\label{freediscontinuous}
    The action generated by \eqref{labelpermute} is free and properly discontinuous.
\end{lem}
\begin{proof}[Sketch]
    The basic observation arises on the level of uncompactified moduli spaces: since any element
of $\mc{R}_d^{S^1}$ has a unit disk representative with
$(z_{out},p_1)$ fixed at $(0,\frac{1}{2})$, the positions of the remaining points identify $\mc{R}_d^{S^1}$ with the space of tuples $(z_1, \ldots, z_d)$ of disjoint (cylically ordered) points on $S^1$ (without any further quotienting by automorphism). The action of $\kappa$, which cyclically permutes the labels $z_1, \ldots, z_d$ in this identification, 
evidently acts freely and properly discontinuously on this locus. Similarly, an element of a boundary stratum consists of an element of $\mc{R}_k^{S^1}$ for some $k \leq d$ with some collection of stable disc bubble trees attached to some or all of the marked points of $\mc{R}_k^{S^1}$, so that there are $d$ leaf (non-nodal) boundary marked points, along with a counterclockwise-ordered labeling of these marked points by  $z_1, \ldots, z_d$ (note that there is a well-defined cyclic counterclockwise ordering of boundary marked points on any such stable configuration). By using a representative of the main component $\mc{R}_k^{S^1}$ with $(z_{out},p_1)$ fixed at $(0,\frac{1}{2})$, an explicit analysis shows that the action of \eqref{labelpermute} remains free and properly discontinuous (e.g., to see free, note there is a well-defined ``first boundary non-nodal marked point at or counterclockwise from the argument of $p_1$''; the action of \eqref{labelpermute} freely permutes the label of this first boundary marked point hence cannot have a fixed point).
\end{proof}
The quotient of the action of $\kappa$ consists of the space of discs with
$z_{out}$ and $p_1$ as before (meaning $z_{out}$ is a negative interior
puncture, and $p_1$ is an auxiliary interior marked point such that for any
representative with $z_{out}$ fixed at 0, $|p_1|= \frac{1}{2}$), equipped with
{\em $d$} {\em cyclically unordered} or {\em unlabeled} boundary marked points.
Note that on the
open-locus $\mathring{\mc{R}}_d^{S^1}$ where $\tau_{out}$ does not point at a
boundary marked point, one can choose a labeling by setting the boundary point
immediately clockwise of where $\tau_{out}$ points to be $z_d$. This induces a
diffeomorphism
\begin{equation}
    \mathring{\mc{R}}_d^{S^1}/\kappa \cong \mc{R}_d^{1, free}
\end{equation}
Similarly, on the complementary locus where $\tau_{out}$ points at a boundary marked
point, we can similarly choose a labeling by declaring this boundary marked
point to be $z_d$, giving a diffeomorphism (of this locus) with
$\check{\mc{R}}_d^1$.

We now choose Floer perturbation data for the family of moduli spaces $\mc{R}_d^{S^1}$; in
fact, it will be helpful to re-choose Floer data for the moduli spaces
appearing in the non-unital open-closed map to have extra compatibility. To
that end, a {\em BV compatible Floer datum for the non-unital open-closed map}
is an inductive choice $(\mathbf{D}_{\check{\oc}}, \mathbf{D}_{\hat{\oc}},
\mathbf{D}_{S^1})$ of Floer
data where $\mathbf{D}_{\check{\oc}}$ and $\mathbf{D}_{\hat{\oc}}$ is a
universal and consistent choice of Floer data for the non-unital open-closed
map as before, and $\mathbf{D}_{S^1}$ consists of, 
for each $d
\geq 1$
and every representative $S \in  \overline{\mc{R}_d^{S^1}}$, a Floer datum for $S$
varying smoothly over the moduli space. Again, these satisfy the usual consistency condition with respect to previously made choices along lower-dimensional strata. 
Moreover, there are two additional inductive constraints on the Floer data chosen:
\begin{align}
            \label{S1occompatibility}&\textrm{On the codimension-1 loci $\overline{\mc{R}_d^{S^1_i}}$ where
            $\tau_{out}$ points at $z_i$, the Floer datum }\\
            \nonumber &\textrm{ should agree with the pullback by $\tau_i$ of the existing Floer datum}\\
            \nonumber &\textrm{ for the
                (check) open-closed map. }\\
            \label{S1tauequivariance}&\textrm{The Floer datum should be $\kappa$-equivariant, where $\kappa$ is the map
                    \eqref{labelpermute}. } 
        \end{align}
Also, there is a final a posteriori constraint on the Floer data for the
non-unital open-closed map $\mathbf{D}_{\hat{\oc}}$; for $S \in
\overline{\hat{\mc{R}}}_d^1$:
\begin{equation}\label{S1compatibilitycondition}
    \begin{split}
    &\textrm{the Floer datum on the main component $S_0$ of $\overline{\pi}_f(S)$ should coincide with the}\\
    &\textrm{existing datum chosen on $S_0 \in \mc{R}_d^{1,free} \subset \mc{R}_d^{S^1}$}.
    \end{split}
\end{equation}
By an inductive argument as before, a BV compatible Floer datum for the non-unital open-closed map exists.

To explain the way choices are made (which ensures both existence at every stage and that the requirements above are satisfied): we choose the data for $\mc{R}_d^{S^1}$ prior to choosing that of $\overline{\hat{\mc{R}}}_d^1$ and note that the condition \eqref{S1compatibilitycondition} specifies the Floer
datum on $\overline{\hat{\mc{R}}}_d^1$ entirely. In particular,  the conditions \eqref{forgottencondition} required on the latter Floer datum are compatibile with consistency and the condition \eqref{S1occompatibility}.
With regards to choosing the data for $\mc{R}_d^{S^1}$, the equivariance
constraint \eqref{S1tauequivariance}, which is compatible with both
\eqref{S1occompatibility} (which is a $\kappa$-equivariant condition) and with
the consistency condition, is also unproblematic in light of Lemma
\ref{freediscontinuous}:
one can pull back a Floer datum from the quotient of
$\overline{\mc{R}_d^{S^1}}$ by $\kappa$.

Fixing a BV compatible Floer datum for the non-unital open-closed map  we obtain, for any
$d$-tuple of Lagrangians $L_0, \ldots, L_{d-1}$, and asympotics 
$\vec{x} = (x_{d}, \ldots, x_1)$ ($x_i \in \chi(L_{i-1},L_{i \mathrm{\ mod\ }d})$), 
   $y_{out} \in \mc{O}$,
a moduli space 
\begin{equation}
    {\mc{R}}_d^{S^1}(y_{out}; \vec{x}),
\end{equation}
of parametrized families of solutions to Floer's equation (with respect to the Floer data chosen)
\begin{equation} \{ (S, u) | S \in \mc{R}_d^{S^1}, u:\pi_f(S) \ra M |   (du - X \otimes \alpha)^{0,1} = 0 \}
\end{equation}
satisfying asymptotic and boundary conditions as in \eqref{floerequationocasymptotics} (again with the modifications of Remarks \ref{liouvilleFloerdatumOC} or \ref{wrappedFloerdatumOC} for compact or wrapped Fukaya categories of Liouville manifolds).  
Generically the Gromov-Floer compactification
\begin{equation}
    \overline{\mc{R}}_d^{S^1}(y_{out}; \vec{x})
\end{equation}
of the components of virtual dimension $\leq 1$ are compact manifolds-with-boundary of the expected dimension; this dimension coincides (mod 2 or exactly depending on whether we are in a $\Z/2$ or $\Z$-graded setting) with
\begin{equation}
    \deg(y_{out}) - n + d - \sum_{k=1}^{d} \deg(x_k).
\end{equation}
Each rigid $u \in \overline{{\mc{R}}}_d^{S^1}(y_{out}; \vec{x})$
gives by the orientation from \eqref{compactifications1} and
\cite{Abouzaid:2010kx}*{Lemma C.4} an isomorphism of orientation lines
\begin{equation}
    (\mc{R}_d^{S^1})_u: o_{x_{d}} \otimes \cdots \otimes o_{x_1} \ra o_{y_{out}},
\end{equation}
which gives the $|o_{y_{out}}|_{\K}$ component of the $S^1$ open-closed map
with $d$ inputs in the lines $|o_{x_d}|_{\K}, \ldots, |o_{x_1}|_{\K}$, up to a sign twist given below: define
\begin{equation}
    \begin{split}
        \oc^{S^1}([x_{d}], \ldots, [x_1]) &:= 
        \sum_{u \in  \overline{\mc{R}}_d^{S^1}(y_{out}; \vec{x})\textrm{ rigid}}  (-1)^{\clubsuit_d} (\mc{R}_d^{S^1})_u([x_d], \ldots, [x_1]),\\
        \clubsuit_d &= \sum_{i=1}^d (i+1) \cdot \deg(x_i) + \deg(x_d) + d-1.
\end{split}
\end{equation}
The proof of Proposition \ref{ocS1Brelation}, which equates $\oc^{S^1}$ with
$\hat{\oc} \circ B^{nu}$, appears below and is composed of two steps. First, we decompose the moduli space
$\mc{R}_d^{S^1}$
into sectors in which $\tau_{out}$ points between a pair of adjacent boundary
marked points. It will follow that the sum of the corresponding ``sector
operations'' is exactly $\oc^{S^1}$. The sector operations in turn can be
compared to $\hat{\oc}$ via cyclically permuting inputs and an orientation
analysis.

We begin by defining the relevant sector operations: 
For $i\in \Z/(d+1)\Z$, define 
\begin{equation}
    \hat{\mc{R}}^{1}_{d,\tau_i}
\end{equation}
to be the abstract moduli space of discs with $d+1$ boundary punctures $z_1,
\ldots, z_{i}, z_f, z_{i+1}, \ldots, z_d$ arranged in counterclockwise order and
interior puncture $z_{out}$ with asymptotic marker pointing towards the
boundary point $z_f$, which is also marked as ``auxiliary.'' There is a 
bijection
\begin{equation}\label{taui}
    \tau_i: \hat{\mc{R}}^{1}_{d,\tau_i} \simeq \hat{\mc{R}}^1_d
\end{equation}
given by cyclically permuting labels, which induces a model for the
compactification $\overline{\hat{\mc{R}}^{1}_{d,\tau_i}}$. 
However, we will use a different orientation than the one induced
by pullback: on a slice with fixed position of $z_d$ and $z_{out}$, we take the
volume form
\begin{equation}\label{differentorientation}
    dz_1 \wedge \cdots \wedge dz_{d-1} \wedge dz_f.
\end{equation}
By construction, the induced ``forgetful map''
\begin{equation}
    \pi_f^{i}: \hat{\mc{R}}^1_{d,\tau_i} \ra \mc{R}^{S^1_{i,i+1}},
\end{equation}
is an oriented diffeomorphism that extends to a map between compactifications
(note as before that strictly speaking this map does not forget any
information, at least on the open locus).  
\begin{rem}\label{inducedorientationslice}
    In the case $i=0$, note that this orientation 
   agrees with
    the previously
    chosen orientation \eqref{openclosedhatorientation} on $\hat{\mc{R}}^1_d$.
    To see this, note that we previously defined the orientation on $\hat{\mc{R}}^1_d$
    in terms of a different slice of the group action. To compare the forms $dz_1
    \wedge \cdots \wedge dz_{d-1} \wedge dz_f$ (coming from the slice with fixed $z_d$
    and $z_{out}$) and $-dz_1 \wedge \cdots \wedge dz_d$ (coming from the slice with
    fixed $z_f$ and $z_{out}$), note that
    either orientation is induced by the following procedure:
    \begin{itemize}
        \item fix an orientation  on the space of discs as above with fixed
            position of $z_{out}$ (but not $z_f$ or $z_d$): we shall fix the
            canonical orientation $dz_1 \wedge \cdots \wedge dz_d \wedge dz_f$; 

    \item fix a choice of trivalizing vector field for the remaining $S^1$
        action on this space of discs with fixed $z_{out}$: we shall fix $S =
        (-\partial_{z_f} -
    \partial_{z_1} - \cdots - \partial_{z_d})$; and 

    \item fix a convention for contracting orientation forms along slices of the action: to determine the
        orientation on a slice of an $S^1$ action, we will contract the orientation on the original space on the
        right by the trivializing vector field. 
    \end{itemize}
    Moreoever, this data induces an orientation on the quotient by the
    $S^1$ action, and also an oriented isomorphism between the induced orientation
    on any slice and that of the quotient.  It follows that on the quotient,
    the orientation
    $-dz_1 \wedge \cdots \wedge dz_d$ (from the slice where $z_f$ is fixed), and the  orientation 
    $dz_1 \wedge \cdots \wedge dz_{d-1} \wedge dz_f$ (from the slice where $z_d$ is fixed) agree.
    We conclude these two orientations agree.
    The author thanks Nick Sheridan for relevant discussions about orientations of moduli spaces.
\end{rem} 
Choose as a Floer datum for each $\overline{\mc{R}^1_{d,\tau_i}}$ the pulled back
Floer datum from $\overline{\hat{\mc{R}}^1_d}$ via \eqref{taui}; this system of choices automatically is inductively consistent with choices made on lower strata (inheriting this property from the Floer data on the collection of $\overline{\hat{\mc{R}}^1_d}$). 
Using this choice, for any $d$-tuple of Lagrangians $L_0, \ldots, L_{d-1}$, and
asympotic conditions
$        \vec{x} = (x_{d}, \ldots, x_1),\ x_i \in \chi(L_{i-1},L_{i \mathrm{\ mod\ }d})$,
        $y_{out} \in \mc{O}$
we obtain a moduli space 
\begin{equation}
    \mc{R}^1_{d,\tau_i}(y_{out}; \vec{x}) = \hat{\mc{R}}^1_d(y_{out}; (x_{i-1}, \ldots, x_1, x_{d}, \ldots, x_i))
\end{equation}
of parametrized families of solutions to Floer's equation
\begin{equation}
    \{ (S, u) | S \in \hat{\mc{R}}_d^1, u:\pi_f(S) \ra M |   (du - X \otimes \alpha)^{0,1} = 0 \textrm{ using the Floer datum for $\pi_f(S)$} \}
\end{equation}
satisfying asymptotic and boundary conditions as in
\eqref{floerequationocasymptotics} (with the modifications as in Remark \ref{liouvilleFloerdatumOC} or \ref{wrappedFloerdatumOC} in the Liouville case),
and its Gromov-Floer compactification 
\begin{equation}
    \overline{\mc{R}}^1_{d,\tau_i}(y_{out}; \vec{x}) := \overline{\hat{\mc{R}}^1_d}(y_{out}; (x_{i}, \ldots, x_1, x_{d}, \ldots, x_{i+1})),
\end{equation}
whose components of virtual dimension $\leq 1$ (at least) are compact manifolds-with-boundary of the correct dimension (which coincides exactly in the graded case and mod 2 in the $\Z/2$ graded case with $\deg(y_{out}) - n + d - \sum_{j=0}^{d} \deg(x_j)$).

    Each rigid element
$u \in \overline{\mc{R}}_{d,\tau_i}^1(y_{out}; \vec{x})$ gives by
\eqref{differentorientation} and \cite{Abouzaid:2010kx}*{Lemma C.4} an
isomorphism of orientation lines
\begin{equation}\label{isomorphismrotated}
    (\mc{R}_{d,\tau_i}^1)_u: o_{x_{d}} \otimes \cdots \otimes o_{x_1} \ra o_{y_{out}},
\end{equation}
which defines the $|o_{y_{out}}|_{\K}$ component of an operation $\hat{\oc}_{d,\tau_i}$
with $d$ inputs in the lines $|o_{x_d}|_{\K}, \ldots, |o_{x_1}|_{\K}$, up to the following sign twist:
\begin{equation}
    \begin{split}
        \hat{\oc}_{d,\tau_i}([x_{d}], \ldots, [x_1]) &:= 
        \sum_{u \in  \overline{\hat{\mc{R}}}_{d,\tau_i}^1(y_{out}; \vec{x})\textrm{ rigid}}  (-1)^{\clubsuit_d} (\hat{\mc{R}}_{d,\tau_i}^1)_u([x_d], \ldots, [x_1]),\\
        \clubsuit_d &= \sum_{i=1}^d (i+1) \cdot \deg(x_i) + \deg(x_d) + d-1.
\end{split}
\end{equation}

\begin{lem}\label{sectordecomposition}
    As chain level operations,
    \begin{equation}
        \oc^{S^1} = \sum_i \hat{\oc}_{d,\tau_i}
    \end{equation}
\end{lem}
\begin{proof}
    For each $d$, there is an embedding of abstract moduli spaces
    \begin{equation}\label{disjointembedding}
        \coprod_{i} \hat{\mc{R}}^1_{d,\tau_i} \stackrel{\coprod_i \pi_f^{i}}{\lra} \coprod_i \mc{R}^{S^1_{i,i+1}}_d 
        \hookrightarrow \mc{R}^{S^1}_d;
    \end{equation} 
    see Figure \ref{fig:oc_s1_new}.
    \begin{figure}[h]
        \caption{ The diffeomorphism between  $\hat{\mc{R}}^1_{2,\tau_0} \cup  \hat{\mc{R}}^1_{2,\tau_1}$ and the open dense part of $\mc{R}^{S^1}_2$ given by $\mc{R}^{S^1_{0,1}}_2 \cup \mc{R}^{S^1_{1,2}}_2$. The former spaces can in turn be compared to $\hat{\mc{R}}^1_2$ via cyclic permutation of labels. \label{fig:oc_s1_new} }
        \centering
        \includegraphics[scale=0.6]{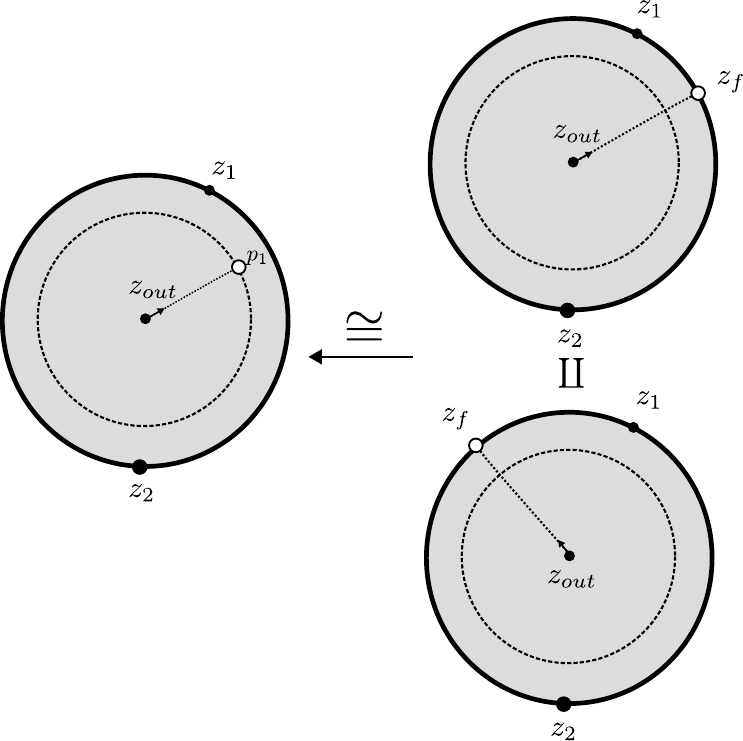}
    \end{figure}

    By construction, this map is compatible with Floer data
    (this uses the fact that the Floer data on $\mc{R}^{S^1_{i,i+1}}$ agrees
    with the data on $\hat{\mc{R}}_d^1$ via the reshuffling map $\kappa^{-i}$
    by \eqref{S1tauequivariance}), and covers all but a codimension 1 locus in
    the target. Since after perturbation zero-dimensional solutions to Floer's
    equation can be chosen to come from the complement of any codimension 1
    locus in the source abstract moduli space, we conclude that the two
    operations in the Lemma, which arise from either side of
    \eqref{disjointembedding}, are identical up to sign. To fix the signs, note
    that \eqref{disjointembedding} is in fact an oriented embedding, and all
    the sign twists defining the operations $\hat{\oc}_{d,\tau_i}$ are chosen
    to be compatible with the sign twist in the operation $\oc^{S^1}$.
\end{proof}
Next, because the Floer data used in the constructions are identical,
$\hat{\oc}_{d,\tau_i}(x_{d} \otimes \cdots \otimes x_1):= \hat{\oc}_{d,\tau_i}(x_{d},  \ldots, x_1)$ (recall the abuse of notation $x_i:=[x_i]$) agrees with
$\hat{\oc}(x_{i} \otimes \cdots \otimes x_1 \otimes x_{d} \otimes \cdots \otimes x_{i+1}):= \hat{\oc}(x_{i}, \ldots, x_1, x_{d}, \ldots, x_{i+1})$ up to a sign difference coming from orientations of abstract
moduli spaces, cyclically reordering inputs, and sign twists. The following
proposition computes the sign difference, and hence completes the proof of
Proposition \ref{ocS1Brelation}: 
\begin{lem}\label{sectorsignorientation}
    There is an equality
    \begin{equation}
        \begin{split}
            \hat{\oc}_{d,\tau_i}(x_d \otimes \cdots \otimes x_1)&= \hat{\oc}^d(s^{nu}( t^i (x_d \otimes \cdots x_1)))
    \end{split}
    \end{equation}
where $s^{nu}$ is the operation \eqref{snu} arising from changing a check term to a hat
term with a sign twist. 
\end{lem}
\begin{proof}
    It is evident that $\hat{\oc}_{d,\tau_i}$ agrees with $\hat{\oc}_d \circ s^{nu} \circ
    t^i$ up to sign, as the Floer data used in the two constructions are
    identical.  By an inductive argument it suffices to verify the following
    equalities of signed operations:
    \begin{align}
        \hat{\oc}_{d,\tau_0} &= \hat{\oc}_d \circ s^{nu}, \label{insertunitsign}\\
        \hat{\oc}_{d, \tau_1} &= \hat{\oc}_{d,\tau_0} \circ t \label{cyclicallyrotatesign};
    \end{align}
    the remaining sign changes are entirely incremental.  For the equality
    \eqref{insertunitsign}, we simply note that the signs appearing in the
    operations $\hat{\oc}_{d,\tau_0}([x_d], \ldots, [x_1])$ and $\hat{\oc}_d ([x_d],
    \ldots, [x_1])$ differ in the following fashions:
    \begin{itemize}
        \item The abstract orientations on the moduli space of domains 
            agree,
            as per Remark \ref{inducedorientationslice}.

        \item The difference in sign twists is given by $\clubsuit_d -
            \hat{\star}_d = \sum_{i=1}^d |x_i| + |x_d| + d-1 = (\sum_{i=1}^d ||x_i||) +
            1 + |x_d| = \maltese_1^d + ||x_d||$.  
    \end{itemize}

    All together, the parity of difference in signs is $\maltese_1^d + ||x_d||$
    which accounts for the sign in the algebraic operation $s^{nu}$ (see
    \eqref{snu}); this verifies \eqref{insertunitsign}.
    
    Next, the sign difference between the two operations in the equality
    \eqref{cyclicallyrotatesign} is a sum of three contributions:
    \begin{itemize}
        \item The two orientations of abstract moduli spaces from (on the slice where $z_f$ and $z_{out}$ are fixed; see Remark \ref{inducedorientationslice}) $-dz_1
            \wedge \cdots \wedge dz_d$ to $dz_2 \wedge \cdots \wedge dz_d
            \wedge dz_1$ differ by a sign change of parity \[d-1.\]

        \item For a given collection of inputs, the change in {\it sign
            twisting data} from $\clubsuit_d = \sum_{i=1}^d (i+1) \cdot |x_i| + |x_d| + d-1$ to $\sum_{i=1}^{d-1} (i+1)
            |x_{i+1}| + (d+1)|x_1| + |x_1| + d-1 = \sum_{i=2}^d i |x_i| + d
            |x_1| + d-1$ ($\clubsuit_d$ for the sequence $(x_2, \ldots, x_d,
            x_1)$) induces a sign change of parity 
            \[
                \sum_{i=2}^d |x_i| + |x_d| + d |x_1| = \sum_{i=1}^d |x_i| + |x_d| + (d-1) |x_1| = \sum_{i=1}^d ||x_i|| + (d-1) || x_1|| + ||x_d|| = \maltese_1^d + (d-1) ||x_1|| + ||x_d||.
            \]

        \item Finally, the re-ordering of determinant lines of the inputs induces a sign
        change of parity 
        \[
            |x_1| \cdot (\sum_{i=2}^d |x_i|) = ||x_1|| \cdot (\sum_{i=2}^d ||x_i|| ) + \sum_{i=2}^d ||x_i|| + (d-1)||x_1|| + (d-1) = ||x_1|| \maltese_2^d + \maltese_1^d + d||x_1|| + (d-1)
        \]
    \end{itemize}
    The cumulative sign parity is congruent mod 2 to
    \[
        ||x_1|| \maltese_2^d + ||x_1|| + ||x_d||,
    \]
    which is precisely the sign appearing in $t$ (see \eqref{toperator}). This
    verifies \eqref{cyclicallyrotatesign}.
\end{proof}

\begin{proof}[Proof of Proposition \ref{ocS1Brelation}]
    Combine Lemmas \ref{sectordecomposition} and \ref{sectorsignorientation} (note the definition of $B^{nu}$ given in \eqref{eq:Bnudef}).
\end{proof}

\subsection{Compatibility of homology-level BV operators}\label{sec:homology}
Before diving into the statement of chain-level equivariance, we prove a
homology-level statement. The below Theorem is insufficient for studying,
say, equivariant homology groups, but may be of independent interest.
\begin{thm}\label{thm:homology}
    The homology level open-closed map $[\oc]$ intertwines the Hochchild and symplectic
    cohomology BV operators, that is 
    \begin{equation}
        [\oc] \circ [B^{nu}] = [\delta_1] \circ [\oc].
    \end{equation}
\end{thm}

Theorem \ref{thm:homology}
is an immediate consequence of the
following chain-level statement: 
\begin{prop}\label{homologylevelchainhomotopy}
    The following diagram homotopy commutes:
\begin{equation}
    \xymatrix{ \r{CH}_{*-n}(\mc{F}, \mc{F}) \ar[d]^{\check{\oc}} \ar@{^{(}->}[r]^{\iota}_{\sim} & \r{CH}_{*-n}^{nu}(\mc{F}, \mc{F}) \ar[r]^{B^{nu}} & \r{CH}_{*-n-1}^{nu}(\mc{F}, \mc{F}) \ar[d]^{\oc} \\
    CF^*(M) \ar[rr]^{\delta_1} & & CF^{*-1}(M).}
\end{equation}
where $\iota$ is the inclusion onto the left factor, which is a
quasi-isomorphism by Lemma \ref{inclusionquasi}.  More precisely, there exists
an operation $\check{\oc}^1: \r{CH}_{*-n}(\mc{F}, \mc{F}) \ra CF^{*-2}(M)$
satisfying 
\begin{equation}
    (-1)^{n+1} d \check{\oc}^1 + \check{\oc}^1 b = \hat{\oc} B^{nu} \iota - (-1)^n \delta_1 \check{\oc}.
\end{equation}
\end{prop}
\begin{proof}[Proof of Theorem \ref{thm:homology}] 
    Proposition \ref{homologylevelchainhomotopy} immediately implies that
    $[\delta_1] \circ [\check{\oc}] = [\oc]\circ [B^{nu}] \circ [\iota]$ where
    $\iota: \r{CH}_{*-n}(\mc{F}, \mc{F}) \ra \r{CH}_{*-n}^{nu}(\mc{F}, \mc{F})$
    is the inclusion of chain complexes. But by Lemma \ref{inclusionquasi},
    $[\iota]$ is an isomorphism and by Corollary \ref{homologylevelnonunitaloc}
    $[\check{\oc}] = [\oc]$.
\end{proof}

To define $\check{\oc}^1$, consider
\begin{equation} 
    \label{check1equivariant}
    { }_1 \check{\mc{R}}_{d}^1,
\end{equation}
the moduli space of discs with with $d$ positive boundary marked points $z_1,
\ldots, z_d$ labeled in counterclockwise order, 1 interior negative puncture
$z_{out}$ equipped with an asymptotic marker, and $1$ additional interior
marked point $p_1$ (without an asymptotic marker), marked as {\it auxiliary}.
Also, with respect to the unit disc representative of any element of this moduli space fixing
$z_d$ at $1$ and $z_{out}$ at 0 on the unit disc, $p_1$ should lie {\it inside
the circle of radius $\frac{1}{2}$}:
\begin{equation}\label{ordered1} 0 < |p_1| < \frac{1}{2}.
\end{equation} 
Using the above representative, one can talk about the {\it
angle}, or {\it argument} of $p_1$
\begin{equation} 
    \theta_1:= \arg(p_1).  
\end{equation} 
We require that with respect to the above representative,
\begin{equation}\label{pointing1condition} \textrm{the asymptotic marker on
        $z_{out}$ points in the direction $\theta_1$. } 
\end{equation} 
For every representative $S \in { }_1 \check{\mc{R}}_{d}^1$,
\begin{equation}\label{check1negcylindricalend} 
    \begin{split} &\textrm{
        fix a negative cylindrical end around $z_{out}$ not containing
     $p_1$, compatible with the }\\ &\textrm{direction of the
        asymptotic marker, or {\it equivalently compatible with the
        angle $\theta_1$}.} \end{split} 
\end{equation}
We orient \eqref{check1equivariant} as follows: pick, on a slice of the
automorphism action which fixes the position of $z_d$ at 1 and $z_{out}$ at 0,
the volume form
\begin{equation}
\label{check1orientation} -r_1 dz_1\wedge dz_2 \wedge \cdots
    \wedge dz_{d-1} \wedge  dr_1 \wedge d\theta_1
\end{equation}
The compactification of \eqref{check1equivariant} is a real blow-up of the
ordinary Deligne-Mumford compactification, in the sense of \cite{Kimura:1995fk}
(see \cite{Seidel:2010uq} for a first discussion in the context of Floer
theory), reviewed in Appendix \ref{modulispaces} (note this is the case $k=1$
of the more general description therein).
The result of this discussion is that the codimension 1 boundary of the
compactified check moduli space ${ }_1\overline{ \check{\mc{R}}}_{d}^1$  is
covered by the images of the natural inclusions of the following strata:
\begin{align} 
&\label{check1stratum1} \overline{\mc{R}}^s \times { }_1 \overline{\check{\mc{R}}}_{d-s + 1}^1 \\ 
&\label{check1stratum2} \overline{ \check{\mc{R}}}_{d}^1 \times \overline{\mc{M}}_{1}\\
&\label{check1stratum3} \overline{\check{\mc{R}}}_{d}^{S^1}
\end{align} 
The stratum \eqref{check1stratum3} describes the locus which $|p_1| =
\frac{1}{2}$, which is exactly the locus we defined to be the auxiliary
moduli space $\mc{R}_{d}^{S^1}$ inducing the operation $\oc^{S^1}$.
The
strata \eqref{check1stratum1}-\eqref{check1stratum2} have manifold with corners
structure given by standard local gluing maps using fixed choices of strip-like
ends near the boundary. For \eqref{check1stratum1} this is standard, and for
\eqref{check1stratum2}, the local gluing map uses the cylindrical ends
\eqref{check1negcylindricalend} and \eqref{posendangles} (in other words, one
rotates the $1$-pointed angle cylinder by an amount commensurate to the angle
of the marked point $z_d$ on the disk before gluing; see \S \ref{modulispaces} particularly \eqref{gluinganglesopenclosed}). See Figure \ref{fig:oc1_check_degenerations} for a schematic of \eqref{check1equivariant} and two out of the three types of strata \eqref{check1stratum2}-\eqref{check1stratum3}.
\begin{figure}[h]
    \caption{\label{fig:oc1_check_degenerations}
    A schematic of an element of \eqref{check1equivariant} on the left and a schematic of two of its three types of degenerations on the right, \eqref{check1stratum3} (above) and \eqref{check1stratum2} (below). The remaining type of degeneration \eqref{check1stratum1}, omitted from the figure, occurs when some boundary marked points coalesce into a disc bubble.}
    \centering
    \includegraphics[scale=0.7]{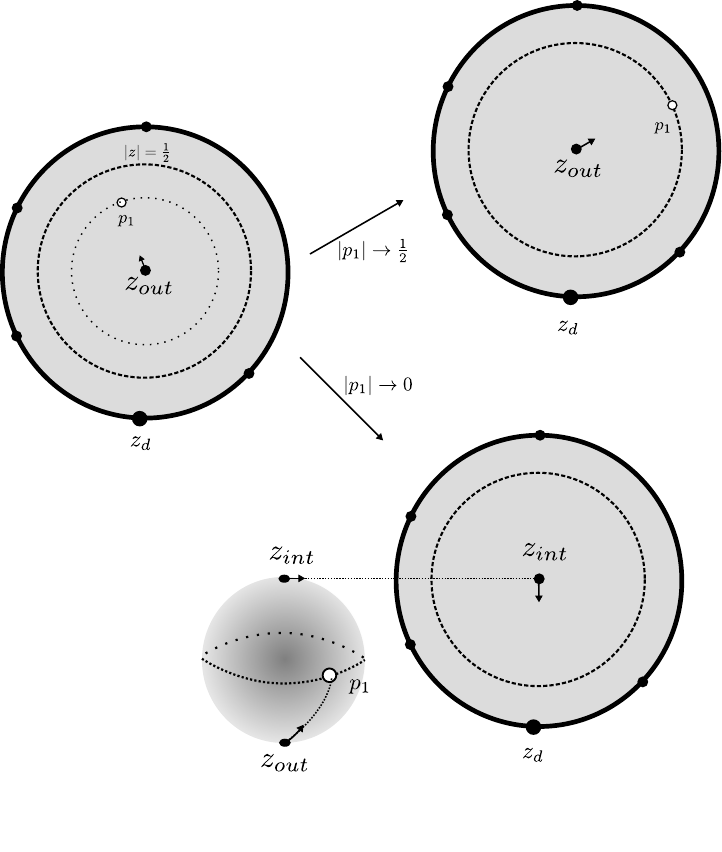}
\end{figure}

   We will as usual fix a {\em Floer datum for the BV homotopy}, meaning an inductive choice, for every $d \geq
    1$, of Floer data for every representative $S \in  { }_1
    \overline{\check{\mc{R}}}_{d}^1$ varying smoothly in $S$,
    which on boundary strata is smoothly equivalent to a product of Floer data
    inductively chosen on lower dimensional moduli spaces.
    Such a system of choices exist again by a contractibility argument, and for any such choice, 
one obtains for any $d$-tuple of Lagrangians $L_0, \ldots,
L_{d-1}$, and asympotic
conditions 
\begin{equation}
    \begin{split} \vec{x} &= (x_{d},
        \ldots, x_1),\ x_i \in
        \chi(L_{i-1},L_{i \mathrm{\ mod\ }d})\\
        y_{out} &\in \mc{O} \end{split}
    \end{equation} 
a compactified moduli space
\begin{equation}
        \label{BVcompactifiedspace1}{ }_1 \overline{\check{\mc{R}}}_{d}^1(y_{out},
        \vec{x})
\end{equation} 
of maps into $M$ with source an arbitrary element $S$ of the moduli space
\eqref{check1equivariant}, satisfying Floer's equation using the Floer datum
chosen for the given $S$ as in \eqref{floerequationOC} with asymptotics and
boundary conditions as in \eqref{floerequationocasymptotics} (with the usual
    modifications in the Liouville case detailed in Remarks \ref{liouvilleFloerdatumOC} and
\ref{wrappedFloerdatumOC}). The virtual dimension of every component of ${ }_1
\overline{\check{\mc{R}}}_{d}^1(y_{out}, \vec{x}) $ coincides (mod 2 or
exactly depending on whether we are in a $\Z/2$ or $\Z$-graded setting)
with
\begin{equation} 
    \label{dimBVcompactifiedspace1} \deg(y_{out}) -
    n + d + 1  - \sum_{i=1}^d \deg(x_i)
 \end{equation} 
 By Assumption \ref{mainassumption}, for generic choices of Floer data, the Gromov-Floer
    compactification of the components of virtual dimension $\leq 1$ of \eqref{BVcompactifiedspace1} are compact manifolds-with-boundary of expected dimension.
 For rigid elements $u$ of the moduli spaces \eqref{BVcompactifiedspace1}, 
the orientations
    \eqref{check1orientation}, 
    and \cite{Abouzaid:2010kx}*{Lemma C.4} induce isomorphisms of
    orientation lines 
    \begin{equation} \label{BVorientationline1} ({ }_1
        \check{\mc{R}}_{d}^1)_u: o_{x_d} \otimes \cdots \otimes o_{x_1}
        \ra o_y\\ 
     \end{equation} 
     As usual ``counting rigid elements $u$,'' i.e., summing application of
     these isomorphisms over all $u$, defines the $|o_{y_{out}}|_\K$ component
     of an operation $\check{\oc}^1$, up to a sign twist which we specify below:
            \begin{equation} 
                \check{\oc}^{1}([x_{d}], \ldots, [x_1]) :=
                \sum_{u \in  { }_k \overline{\check{\mc{R}}}_d^{1}(y_{out};
                \vec{x})\textrm{ rigid}}  (-1)^{\check{\star}_d} ({
                    }_k\check{\mc{R}}_d^{1})_u([x_d], \ldots, [x_1]); \\
            \end{equation} 
where the sign is given by
\begin{equation} 
    \check{\star}_d = \deg(x_d) + \sum_{i} i \cdot \deg(x_i).
\end{equation}
A codimension 1 analysis of the moduli spaces \eqref{BVcompactifiedspace1}
 reveals: 
 \begin{prop} \label{BVcompactifiedequation1} 
     The following equation is satisfied:
    \begin{equation} 
        (-1)^n\delta_1 \check{\oc} + (-1)^n d \check{\oc}^{1}  =
        \oc^{S^1} + \check{\oc}^{1} b 
    \end{equation} 
\end{prop} \begin{proof} The boundary of the 1-dimensional components of
        \eqref{BVcompactifiedspace1} are covered by the rigid components of the
        following types of strata: 
        \begin{itemize} 
            \item spaces of maps with domain lying on the
                    codimension 1 boundary of the moduli space, i.e., in
                    \eqref{check1stratum1}-\eqref{check1stratum3}

 \item semi-stable breakings, namely those of the form 
     \begin{align} 
         &{ }_1\overline{\check{\mc{R}}}_d^1(y_1; \vec{x}) \times \overline{\mc{M}}(y_{out}; y_1) \\ 
         &\overline{\mc{R}}^1(x; x_i) \times { }_1 \overline{\check{\mc{R}}}_d^1(y_{out}; \tilde{\vec{x}}) 
     \end{align}
where $\tilde{\vec{x}}$ denotes the collection of inputs $\vec{x}$ with $x_i$ replaced with $x$.
 \end{itemize}
All together, this implies, up to signs, that 
\begin{equation}
    \label{initialcheck1equation} (-1)^n\delta_1 \check{\oc} + (-1)^n d \check{\oc}^1 =
    \oc^{S^1} + \check{\oc}^{1} b. 
\end{equation} 
\eqref{initialcheck1equation} is of
course a shorthand for saying, for any $d$ and any tuple of $d$ cyclically composable
morphisms $x_d, \ldots, x_1$, that 
\begin{equation} \begin{split} 
        (-1)^n\sum_{i =0}^1 \delta_i &\check{\oc}^{k-i}_d(x_d, \ldots, x_1) = \oc^{S^1}_d(x_d,
    \ldots, x_1) \\
    &+\sum_{i,s}(-1)^{\maltese_1^s} \check{\oc}^1_{d-i+1} (x_d, \ldots,
    x_{s+i+1}, \mu^i(x_{s+i}, \ldots, x_{s+1}), x_s, \ldots, x_1)\\
    &+\sum_{i,j} (-1)^{\sharp_j^i} \check{\oc}^1_{d-i-j} (\mu^{i + j + 1}(x_i, \ldots,
    x_1, x_d, \ldots, x_{d-j}), x_{d-j-1}, \ldots, x_{i+1})  \end{split}
\end{equation}
(recall the abuse of notation $x_i:=[x_i]$). Thus, it suffices to verify that the signs coming from the codimension 1
boundary are exactly those appearing in \eqref{initialcheck1equation}. (in
particular, that the terms in, for instance, $\check{\oc}^1 b$ appear with the
right sign). 

Let us  recall broadly how the signs are
computed.  For any operator $g$ defined above such as $\oc$, $\oc^{S^1}$, $\mu$,
$d$, $\delta_1$ etc., we let $g_{\mathrm{ut}}$ denote the {\em untwisted}
version of the same operator, e.g., the operator whose matrix coefficients come
from the induced isomorphism on orientation lines, without any sign twists by
the degree of the inputs. So for instance $\mu^d(x_d, \ldots, x_1) =
(-1)^{\sum_{i=1}^d i \deg(x_i)} \mu^d_{\mathrm{ut}}(x_d, \ldots, x_1)$ and so on. The
methods described in \cite{Seidel:2008zr}*{Prop. 12.3} and elaborated upon in
\cite{Abouzaid:2010kx}*{\S C.3, Lemma 5.3} and \cite{ganatra1_arxiv}*{\S B},
when applied to the boundary of the 1-dimensional component of the moduli space
of maps, $\overline{\check{\mc{R}}}_{d}^1(y_{out}, \vec{x}))$, 
imply the following signed equality:
\begin{equation}\label{untwistedequation}
    \begin{split}
        0 &= d_{\mathrm{ut}} \check{\oc}^1_{\mathrm{ut}} (x_d, \ldots, x_1) + (\delta_1)_{\mathrm{ut}} \check{\oc}_{\mathrm{ut}}(x_d ,  \ldots,  x_1) \\
    &  - \oc^{S^1}_{\mathrm{ut}} (x_d, \ldots, x_1) + (-1)^{\mathfrak{f}_d} \check{\oc}^{1} b (x_d, \ldots, x_1)
\end{split}
\end{equation}
where
\begin{equation}
    \mathfrak{f}_d := \sum_i (i+1) \deg(x_i) + \deg(x_d) = \check{\star}_d + \maltese_d - d.
\end{equation}
is an auxiliary sign.

To explain this equation \eqref{untwistedequation}, we note first that the
signs appearing in all terms but the last are simply induced by the boundary
orientation on the moduli space of domains. The sign appearing in the first
term also follows from a standard boundary orientation analysis for Floer
cylinders, which we omit (but see e.g., \cite{Seidel:2008zr}*{(12.19-12.20)}
for a version close in spirit). The signs for the first two terms are also
exactly as in Lemma \ref{weaks1actionSHlemma}.  Finally, in the last term, the sign
$(-1)^{\mathfrak{f}_d} \check{\oc}^{1} b (x_d \otimes \cdots \otimes x_1)$
(compare \cite{Seidel:2008zr}*{(12.25)} \cite{ganatra1_arxiv}*{(B.59)}) appears
as a cumulative sum of 
\begin{itemize}
     \item the sign twists which turn the untwisted operations
         $\check{\oc}^1_{\mathrm{ut}}$ and $\mu^s_{ut}$ into the usual
         operations $\check{\oc}^1$ and $\mu^s$;

     \item the Koszul sign appearing in the Hochschild differential $b$; and

    \item the boundary orientation sign appearing
        in the relevant (untwisted) term of $\check{\oc}^1 b$, for instance  $\check{\oc}^1_{\mathrm{ut}}(x_d, \ldots, x_{n+m+1},
        \mu^m_{ut}(x_{n+m}, \ldots, x_{n+1}), x_n, \ldots, x_1)$, which itself
        is as a sum of two different contributions: 
        \begin{enumerate}[(a)]
            \item the comparison
        between the boundary (of the chosen) orientation and the product (of
        the chosen orientation) on the moduli of {\em domains} and 
        \item Koszul
        reordering signs, which measure the signed failure of the method of
        orienting the moduli of maps (in terms of orientations of the domain
        and orientation lines of inputs and outputs) to be compatible with
        passing to boundary strata. 
    \end{enumerate}
    See \cite{Seidel:2008zr}*{(12d)} for more
    details in the case of the $\ainf$ structure, and
    \cite{Abouzaid:2010kx}*{\S C}, \cite{ganatra1_arxiv}*{\S C} for the
    case of these computations for the open-closed map. 
    We note in particular that the forgetful map $F_1: { }_1 \check{\mc{R}}_d^1
    \ra \check{\mc{R}}_d^1$, which forgets the point $p_1$ (and changes the
    direction of the asymptotic marker to point at $z_d$) has {\em complex
    oriented fibers} (in which just the marked point $p_1$ varies). So the
    boundary analysis  of these ``$\check{\oc^1} \circ b$'' strata appearing
    here is identical to the analysis strata appearing in
    \cite{Abouzaid:2010kx} and \cite{ganatra1_arxiv} for the
    ``$\oc\circ b$'' strata, which is why we have not repeated it here.
\end{itemize}

Multiplying all terms of \eqref{untwistedequation} by $(-1)^{\check{\star}_d + \maltese_d -d + 1}$, and noting that,
for instance, $\maltese_d -d + 1 + n -2 = \deg(\check{\oc}^1(x_d \otimes \cdots \otimes
x_1))$, so 
\begin{equation}
    \begin{split}
        (-1)^{\check{\star}_d + \maltese_d - d + 1} (\delta_1)_{ut}\check{\oc}^1_{\mathrm{ut}}(x_d, \ldots, x_1) &= (-1)^{\deg(\check{\oc}^1(x_d, \ldots, x_1)) - n} (\delta_1)_{ut} (-1)^{\check{\star}_d} \check{\oc}^1_{\mathrm{ut}}(x_d, \ldots, x_1)\\ 
    &= \delta_1 \check{\oc}^1(x_d, \ldots, x_1),
\end{split}
\end{equation}
(and similarly for the $d \circ \oc^1$ term), it follows that 
\begin{equation}
    \begin{split}
        0 &= (-1)^{n}\delta_1 \check{\oc}(x_d,  \ldots, x_1) + (-1)^n d \check{\oc}^1 (x_d, \ldots, x_1)\\
        &  - \check{\oc}^{1} b (x_d, \ldots, x_1) - (-1)^{\check{\star}_d + \maltese_d - d + 1}\oc^{S^1}_{\mathrm{ut}} (x_d, \ldots, x_1),
\end{split}
\end{equation}
but $\check{\star}_d + \maltese_d - d + 1 = \clubsuit_d$, and hence the last term above
is $-\oc^{S^1}(x_d, \ldots, x_1)$ as desired.

\end{proof}

\begin{proof}[Proof of Proposition \ref{homologylevelchainhomotopy}]
    The ``sector decomposition'' performed in Proposition \ref{ocS1Brelation}
    which compares $\oc^{S^1}$ to $\hat{\oc} \circ B^{nu} \circ \iota$, along
    with Proposition \ref{BVcompactifiedequation1}, immediately implies the result.
\end{proof}

\subsection{The main construction}\label{sec:equivariance}
We now turn to the definition of the (closed) morphism of
$S^1$-complexes, and the proof of Theorem
\ref{thm:mainresult1} and Corollary \ref{cor:cyclicOC}. The required data takes the form 
\begin{equation} 
    \widetilde{\oc} = \bigoplus_{k\geq 0} 
    \overline{\K[\Lambda]/\Lambda^2}^{\otimes k}
    \otimes \r{CH}_*^{nu}(\f) \ra CF^*(M)[n]
\end{equation} which is equivalent, as recalled in \S
\ref{circleactionsubsection} to defining the collection of maps $\widetilde{\oc} = \{\oc^k\}_{k \geq 0}$, or $u$-linearly (see \S \ref{sec:ulinear}) $\widetilde{\oc} = \sum_{k=0}^\infty \oc^k u^k$, where 
\begin{equation} 
    \oc^k = (\check{\oc}^{k} + \hat{\oc}^{k}):= \widetilde{\oc}^{k|1}(\Lambda, \ldots, \Lambda, -):
    \r{CH}_*^{nu}(\f) \ra CF^{*+n-2k}(M).  
\end{equation} 
(recall from \S \ref{circleactionsubsection} that $\K[\Lambda]/\Lambda^2$ is
our small model for $C_{-*}(S^1)$ and $S^1$-complexes are by definition
strictly unital $\ainf$ modules over $\K[\Lambda]/\Lambda^2$).  
By definition, the case $k=0$ is already covered: 
\begin{equation} 
    \oc^0 = (\check{\oc}^0 \oplus \hat{\oc}^0) =  (\check{\oc} \oplus \hat{\oc}) = \oc
\end{equation}
To handle the general case ($k \geq 0$), we will associate operations to, for
each $d$, compactifications of three moduli spaces of domains, in the following
order: 
\begin{align} 
    \label{checkequivariant} &{ }_k \check{\mc{R}}_{d}^1\\ 
    \label{S1equivariant}   &{ }_k \mc{R}_{d}^{S^1}\\ 
    \label{hatequivariant} &{ }_k \hat{\mc{R}}_d^1;
\end{align} 
The moduli space \eqref{S1equivariant} will induce an auxiliary operation
useful for the proof, whereas \eqref{checkequivariant} and
\eqref{hatequivariant} will lead to the desired operations. For $k = 0$, these
moduli spaces are simply $\check{\mc{R}}_{d}^1$, $\mc{R}_{d}^{S^1}$, and
$\hat{\mc{R}}_d^1$ as defined earlier, and the $k=1$ case of \eqref{checkequivariant} was defined in \eqref{check1equivariant}. 
Inductively, we will construct, and
study operations from \eqref{checkequivariant} and \eqref{S1equivariant}
simultaneously, and then finally construct \eqref{hatequivariant}. 
Using these moduli spaces, we will construct the maps $\check{\oc}^{k}$,
$\hat{\oc}^k$ and an auxiliary operation $\oc^{S^1, k}$
(which we compare to $\hat{\oc}^{k-1} \circ B^{nu}$ in Proposition
\ref{ocks1Brelation} below), and then prove that:
\begin{prop} 
    \label{cyclicOCequations}
    The following equations hold, for each $k \geq 0$: 
    \begin{align} 
        (-1)^n \sum_{i \geq 0}^k \delta_i \check{\oc}^{k-i} &= \hat{\oc}^{k-1} B^{nu} + \check{\oc}^{k} b\\
        (-1)^n \sum_{i \geq 0}^k \delta_i \hat{\oc}^{k-i} &=  \hat{\oc}^{k} b' + \check{\oc}^k (1 - t).
    \end{align} 
    All at once, denoting by $\oc^k =(\check{\oc}^{k} + \hat{\oc}^{k})$ and
    $\widetilde{\oc} = \sum_{i=0}^\infty  \oc^i u^i$, $\delta_{eq} = \sum_{j=0}^\infty \delta_j^{CF}u^j$, and $b_{eq} = b^{nu} + u B^{nu}$ as in \S
    \ref{sec:ulinear}, we have that
    \begin{equation}
        (-1)^n \delta_{eq} \circ \widetilde{\oc} = \widetilde{\oc} \circ b_{eq}.
    \end{equation}
\end{prop}
This will also directly imply our main Theorems, as spelled out at the bottom
of this subsection.  

The space \eqref{checkequivariant} is the moduli space of discs
with with $d$ positive boundary marked points $z_1, \ldots, z_d$ labeled in
counterclockwise
order, 1 interior negative puncture $z_{out}$ equipped with an asymptotic
marker, and $k$ additional interior marked points $p_1, \ldots, p_k$ (without
asymptotic markers), marked as {\it auxiliary}.  Also, on the unit disc
representative of any element of this moduli space which fixes $z_d$ at $1$ and
$z_{out}$ at 0, the $p_i$ should be {\it strictly radially ordered} with norms
in $(0,\frac{1}{2})$; that is,
\begin{equation}
    \label{ordered} 0 < |p_1| < \cdots < |p_k| < \frac{1}{2}.
\end{equation} 
Using the above representative, one can talk about the {\it
angle}, or {\it argument} of each auxiliary interior marked point,
\begin{equation} 
    \theta_i:= \arg(p_i).  
\end{equation} 
We require that with respect to the above representative,
\begin{equation}
    \label{pointingcondition} \textrm{the asymptotic marker on
        $z_{out}$ points in the direction $\theta_1$ (or towards $z_d$ if
    $k=0$).} 
\end{equation} 
(equivalently one could define $\theta_{k+1} = 0$, so that $\theta_1$ is always
defined). See Figure \ref{fig:ock_check} for a depiction.
\begin{figure}[h]
    \caption{\label{fig:ock_check}A representative of an element of the moduli space ${ }_3\check{\mc{R}}^1_{5}$.}
    \centering
    \includegraphics[scale=0.7]{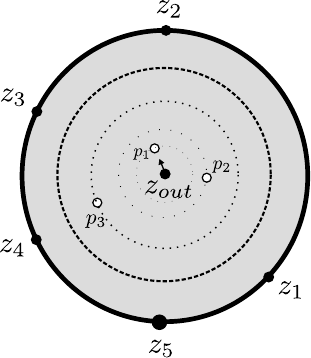}
\end{figure}
For every representative $S \in { }_k \check{\mc{R}}_{d}^1$,
\begin{equation}\label{checknegcylindricalend} 
    \begin{split} 
        &\textrm{ fix a negative cylindrical end around $z_{out}$ not containing any $p_i$, compatible with the }\\ 
        &\textrm{direction of the asymptotic marker, or {\it equivalently compatible with the angle $\theta_1$}.} 
    \end{split} 
\end{equation}
The second moduli space \eqref{S1equivariant} is the moduli space of discs with
with $d$ positive boundary marked points $z_1, \ldots, z_{d}$ labeled in
counterclockwise order, 1 interior negative puncture $z_{out}$ equipped with an
asymptotic marker, and $k+1$ additional interior marked points $p_1, \ldots,
p_k, p_{k+1}$ (without asymptotic markers), marked as {\it auxiliary}.
With respect to the unit disc representative of any element this moduli space fixing $z_d$ at
$1$ and $z_{out}$ at 0, the $p_i$ should again be {\it strictly
radially ordered}, this time with norms lying in $(0, \frac{1}{2}]$ and with {\em $p_{k+1}$ lying on the circle of radius
$\frac{1}{2}$}:
\begin{equation}\label{orderedS1} 0 < |p_1| < \cdots < |p_k| < |p_{k+1}| =
    \frac{1}{2}.  \end{equation} 
The asymptotic marker on $z_{out}$ for this representative again satisfies condition \eqref{pointingcondition}. 
Abstractly we have that $ { }_k \mc{R}_{d}^{S^1} \cong \times { }_k \check{\mc{R}}_{d}^1 \times S^1$, where
the $S^1$ parameter is given by the position of $p_{k+1}$. See \ref{fig:checks1} for a depiction of ${ }_{k-1} \check{\mc{R}}_{d}^{S^1}$.
\begin{figure}[h]
    \caption{\label{fig:checks1}A representative of an element of the moduli space ${ }_{k-1} \check{\mc{R}}_{d}^{S^1}$, which also arises as the boundary stratum  \eqref{checkstratum3} of ${ }_k\overline{ \check{\mc{R}}}_{d}^1$.}
    \centering
    \includegraphics[scale=1.0]{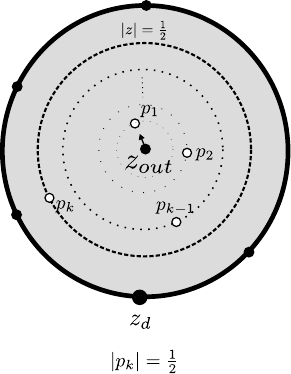}
\end{figure}

The compactification of \eqref{checkequivariant} is a real blow-up of the
ordinary Deligne-Mumford compactification, in the sense of \cite{Kimura:1995fk}
(see \cite{Seidel:2010uq} for a first discussion in the context of Floer
theory), reviewed in more detail in Appendix \ref{modulispaces}.
The result of the discussion there is that the codimension 1 boundary of the
compactified check moduli space ${ }_k\overline{ \check{\mc{R}}}_{d}^1$  is
covered by the images of the natural inclusions of the following strata:
\begin{align} 
    &\label{checkstratum1} \overline{\mc{R}^s} \times { }_k
    \overline{\check{\mc{R}}}_{d-s + 1}^1 \\ 
    &\label{checkstratum2} { }_s\overline{ \check{\mc{R}}}_{d}^1 \times \overline{\mc{M}}_{k-s}\\
    &\label{checkstratum3}{ }_{k-1} \overline{\check{\mc{R}}}_{d}^{S^1}\\
    &\label{checkstratum4}{ }_k^{i,i+1} \overline{\check{\mc{R}}}_{d}^1
\end{align} 
The strata \eqref{checkstratum3}-\eqref{checkstratum4}, in which
$|p_k| = \frac{1}{2}$ (see Figure \ref{fig:checks1}) and $|p_i| = |p_{i+1}|$ (see Figure \ref{fig:checkii1}) respectively, describe the
boundary loci of the ordering condition \eqref{ordered} and hence come equipped
with a natural manifold with corners structure.  
\begin{figure}[h]
    \caption{\label{fig:checkii1} A representative of an element of the stratum \eqref{checkstratum4}. }
    \centering
    \includegraphics[scale=1.0]{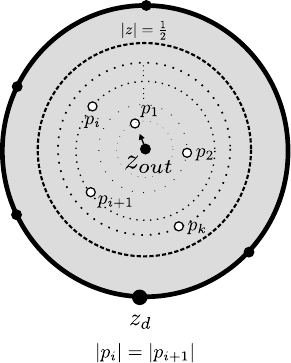}
\end{figure}
The strata
\eqref{checkstratum1}-\eqref{checkstratum2} have manifold with corners
structure given by standard local gluing maps using fixed choices of strip-like
ends near the boundary. For \eqref{checkstratum1} (depicted in Figure \ref{fig:checkainfbubble}) this is standard, and for
\eqref{checkstratum2} (depicted in Figure \ref{fig:checkksv}), the local gluing map uses the cylindrical ends
\eqref{checknegcylindricalend} and \eqref{posendangles} (in other words, one
rotates the $(k-s)$-pointed angle cylinder by an amount commensurate to the angle
of the first marked point $p_{k-s+1}$ on the disk before gluing) as also described in \S \ref{modulispaces}.
\begin{figure}[h]
    \caption{\label{fig:checkainfbubble} A representative of an element of the boundary stratum \eqref{checkstratum1} in which a disc bubble forms (such disc bubble is allowed to include the ``first'' point --- $z_d$ by our convention --- but need not, and does not in the Figure). }
    \centering
    \includegraphics[scale=1.0]{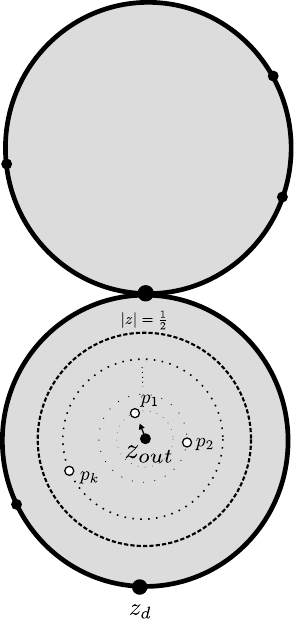}
\end{figure}
\begin{figure}[h]
    \caption{\label{fig:checkksv} A representative of an element of the boundary stratum \eqref{checkstratum2}. }
    \centering
    \includegraphics[scale=1.0]{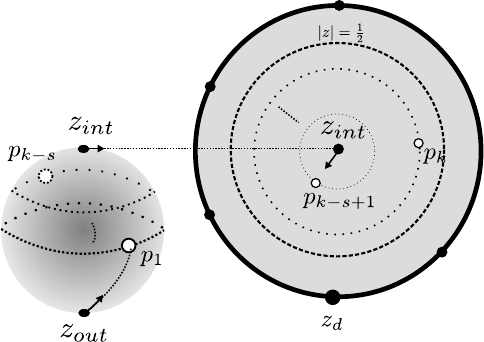}
\end{figure}

Associated to the stratum \eqref{checkstratum4} where $p_i$ and $p_{i+1}$ have
coincident magnitudes, there is a forgetful map 
\begin{equation}
    \label{picheck} \check{\pi}_i: { }_k^{i,i+1}
    \overline{\check{\mc{R}}}_{d}^1 \ra { }_{k-1}
    \overline{\check{\mc{R}}}_{d}^1 
\end{equation} 
which simply forgets the point $p_{i+1}$. Since the norm of $p_{i+1}$ and $p_i$
agree on this locus, this amounts to forgetting the argument of $p_{i+1}$ (in
particular, the fibers of $\check{\pi}_i$ are one-dimensional).

The compactification of the $S^1$-moduli space \eqref{S1equivariant}
can be modeled abstractly by
${ }_k\overline{ \check{\mc{R}}}_{d}^1 \times S^1$. However, it
is again preferable to give an explicit description of the boundary strata, which are
covered in codimension 1 by the following strata: 
\begin{align}
    &\label{s1stratum1} \overline{\mc{R}^s} \times { }_k \overline{\mc{R}}_{d-s + 1}^{S^1} \\ 
    &\label{s1stratum2} { }_s\overline{ \mc{R}}_{d}^{S^1} \times \overline{\mc{M}}_{k-s}\\ 
    &\label{s1stratum4}{ }_k^{i,i+1} \overline{\mc{R}}_{d}^{S^1}.  
\end{align} 
Here,
\eqref{s1stratum1} and \eqref{s1stratum2} are just versions of the
degenerations \eqref{checkstratum1} and \eqref{checkstratum2}, in which a
collection of boundary points bubbles off, or a collection of auxiliary
points convergest to $z_{out}$ and bubbles off (the fact that the latter
occurs in codimension 1 is part of the ``real blow-up phenomenon'' already
discussed).  The stratum \eqref{s1stratum4}, is the locus where $|p_i| =
|p_{i+1}|$, for $i \leq k$
(so when $i=k$, $|p_k| =  |p_{k+1}| = \frac{1}{2}$).

As in \eqref{picheck}, on the stratum 
\eqref{s1stratum4} where
$p_i$ and $p_{i+1}$ have coincident magnitudes, define the map
\begin{equation}\label{piS1} \pi_i^{S^1}: { }_k^{i,i+1}
    \overline{\mc{R}}_{d}^{S^1} \ra { }_{k-1} \overline{\mc{R}}_{d}^{S^1}
\end{equation} to be the one forgetting the point $p_{i+1}$ (so as before, this
map has one-dimensional fibers).

For an element $S \in { }_{k} \overline{\mc{R}}_{d}^{S^1}$, we say that
$p_{k+1}$ {\it points at a boundary point $z_i$} if, for any unit disc
representative of $S$ with $z_{out}$ at the origin, the ray from $z_{out}$ to
$p_{k+1}$ intersects $z_i$. The locus where $p_{k+1}$ points at $z_i$ is
denoted 
\begin{equation} 
    { }_k\overline{\mc{R}}_{d}^{S^1_i}. 
\end{equation}
Similarly, we say that $p_{k+1}$ {\it points between $z_i$ and $z_{i+1}$}
(modulo $d$, so this includes the case of pointing between $z_d$ and $z_1$) if for such a
representative, the ray from $z_{out}$ to $p_{k+1}$ intersects the portion
of $\partial S$ between $z_i$ and $z_{i+1}$. The locus where $p_{k+1}$
points between $z_i$ and $z_{i+1}$ is denoted 
\begin{equation} 
    { }_k\overline{\mc{R}}_{d}^{S^1_{i,i+1}}.  
\end{equation}

As before in \eqref{labelpermute}, there is a free and properly discontinuous $\Z_{d}$-action
\begin{equation} \label{labelpermuteequivariant} \kappa: {
    }_k\overline{(\mc{R}^1_d)^{S^1}} \ra { }_k\overline{(\mc{R}^1_d)^{S^1}}
\end{equation} which cyclically permutes the labels of the boundary marked
points; as before, $\kappa$ changes the label $z_i$ to $z_{i+1}$ for $i < d$, and
$z_d$ to $z_1$ (compare Lemma \ref{freediscontinuous}).

Finally, we come to the third moduli space \eqref{hatequivariant}, the moduli
space of discs with with $d+1$ positive boundary marked points $z_1,
\ldots, z_{d}, z_f$ labeled in counterclockwise order, 1 interior negative
puncture $z_{out}$ equipped with an asymptotic marker, and $k$ additional
interior marked points $p_1, \ldots, p_k$ (without an asymptotic marker),
marked as {\it auxiliary}, staisfying a {\it strict radial ordering} condition
as before: for any representative element with $z_f$ fixed at 1 and $z_{out}$
at 0, we require \eqref{ordered} to hold,
as well as condition \eqref{pointingcondition}. The boundary marked point $z_f$
is also marked as auxiliary, but apart from this designation we see that (identifying $z_f$ with $z_{d+1}$) ${ }_k\hat{\mc{R}}_d^1
\cong { }_k\check{\mc{R}}_{d+1}^1$. See Figure \ref{fig:ock_hat}.
\begin{figure}[h]
    \caption{\label{fig:ock_hat}A representative of an element of the moduli space ${ }_4\hat{\mc{R}}^1_{4}$.}
    \centering
    \includegraphics[scale=1.0]{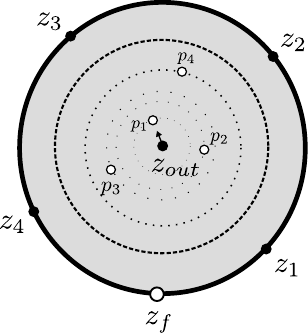}
\end{figure}

In codimension 1, the compactification $ { }_k \overline{\hat{\mc{R}}}_{d}^1$
has boundary covered by inclusions of the following strata: 
\begin{align} 
&\label{hatstratum1} \overline{\mc{R}}^s \times { }_k \overline{\hat{\mc{R}}}_{d-s + 1}^1 \\ 
&\label{hatstratum2} \overline{\mc{R}}^{m,f_k}\times_{d-m+1} { }_k\overline{\check{\mc{R}}}^1_{d-m+1} \ \ \ 1 \leq k \leq m\\
&\label{hatstratum3} { }_s\overline{ \hat{\mc{R}}}_{d}^1 \times \overline{\mc{M}}_{k-s}\\ 
&\label{hatstratum4}{ }_{k-1} \overline{\hat{\mc{R}}}_{d}^{S^1}\\ 
&\label{hatstratum5}{ }_k^{i,i+1} \overline{\hat{\mc{R}}}_{d}^1 
\end{align} 
Once more, on strata \eqref{hatstratum5} where $p_i$ and $p_{i+1}$ have
coincident magnitudes (depicted in Figure \ref{fig:hatii1}), define the map \begin{equation}\label{pihat}
\hat{\pi}_i: { }_k^{i,i+1} \overline{\hat{\mc{R}}}_{d}^1 \ra { }_{k-1}
\overline{\hat{\mc{R}}}_{d}^1.  \end{equation} to be the one forgetting the
point $p_{i+1}$ (so again, this map has one-dimensional fibers).  
\begin{figure}[h]
    \caption{\label{fig:hatii1} A representative of an element of the stratum \eqref{hatstratum5}. }
    \centering
    \includegraphics[scale=1.0]{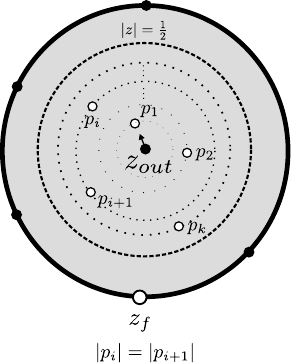}
\end{figure}
On the
stratum \eqref{hatstratum4}, which is the locus where $|p_{k}| =
\frac{1}{2}$ (depicted in Figure \ref{fig:hats1}), there is also a map of interest
\begin{equation}
    \label{hatboundaryforget} 
    \hat{\pi}_{boundary}: { }_{k-1}\overline{\hat{\mc{R}}}_{d}^{S^1} \ra { }_{k-1} \overline{\mc{R}}_d^{S^1}
\end{equation} 
which forgets the position of the auxiliary boundary point $z_f$.
\begin{figure}[h]
    \caption{\label{fig:hats1}A representative of an element of the boundary stratum  \eqref{hatstratum4}.}
    \centering
    \includegraphics[scale=1.0]{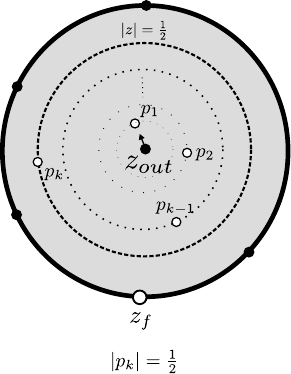}
\end{figure}
The stratum \eqref{hatstratum3}, depicted in Figure \ref{fig:hatksv}, is the locus where some subcollection of interior auxiliary points $p_1, \ldots, p_{k-s}$ tend to zero and split off an angle-decorated cylinder (in the manner again described in \S \ref{modulispaces} for \eqref{checkstratum2}).
\begin{figure}[h]
    \caption{\label{fig:hatksv} A representative of an element of the boundary stratum \eqref{hatstratum3}. }
    \centering
    \includegraphics[scale=1.0]{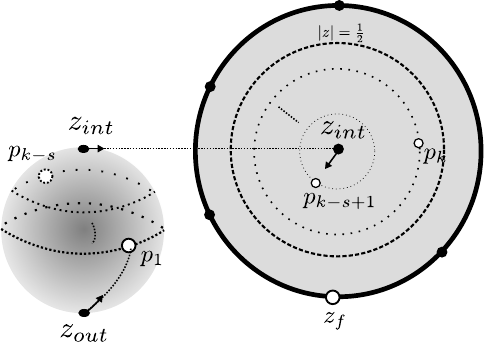}
\end{figure}
The strata \eqref{hatstratum1} and \eqref{hatstratum2} (depicted in Figures \ref{fig:hatainfbubble} and \ref{fig:hatainfbubble2} respectively) are the loci where a disc bubble forms involving some boundary marked points (not including or including $z_f$ respectively).
\begin{figure}[h]
    \caption{\label{fig:hatainfbubble} A representative of an element of the boundary stratum \eqref{hatstratum1} in which a disc bubble forms not including the auxiliary point $z_f$. }
    \centering
    \includegraphics[scale=1.0]{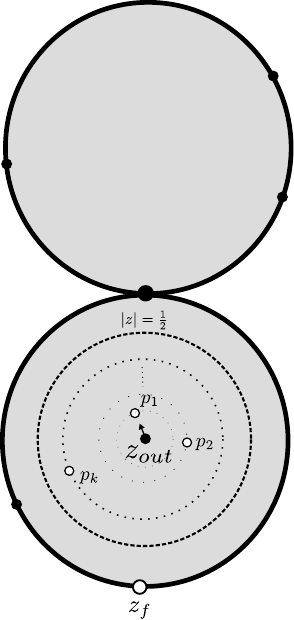}
\end{figure}
\begin{figure}[h]
    \caption{\label{fig:hatainfbubble2} A representative of an element of the boundary stratum \eqref{hatstratum2} in which a disc bubble forms including the auxiliary point $z_f$. }
    \centering
    \includegraphics[scale=1.0]{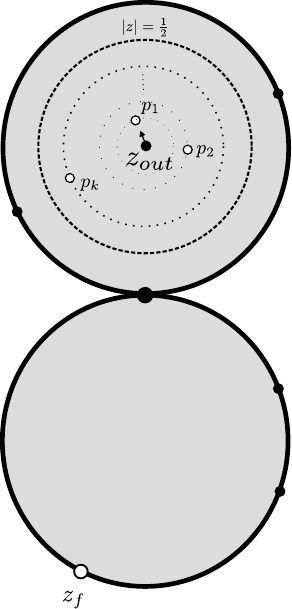}
\end{figure}

Denote by ${ }_k{\mc{R}}_d^{1,free}: = { }_k {\mc{R}}_d^{S^1_{d,1}}$ to be the
sector of the moduli space ${ }_k\mc{R}^{S^1}_d$ where $p_{k+1}$ points between
$z_d$ and $z_1$.
The {\it auxiliary-rescaling map} 
\begin{equation}\label{equivariantforgetful}
    \pi_f: { }_k \hat{\mc{R}}_d^1 \ra { }_k{\mc{R}}_d^{1,free},
\end{equation}
(our replacement of the ``forgetful map'') can be described as follows: given a representative $S$ in $ { }_k
\hat{\mc{R}}_d^1$ with $z_{out}$ fixed at the origin, there is a unique
point $p$ with $|p|= \frac{1}{2}$ between $z_{out}$ and $z_f$. $\pi_f(S)$
is the element of ${ }_k\mc{R}^{S^1}_d$ with $p_{k+1}$ equal to this point
$p$ and with $z_f$ deleted. Of course, $z_f$ is not actually forgotten,
because it is determined by the position of $p_{k+1}$. In particular
\eqref{equivariantforgetful} is a diffeomorphism. We extend this map to a map $\overline{\pi}_f$
from the compactification ${ }_k \overline{\hat{\mc{R}}}_d^1$ as in \S
\ref{sec:nonunital}, by putting the auxiliary point $z_f$ back in, eliminating
any component which is not main or secondary which has only one (non-auxiliary)
boundary marked point $q$, and by labeling the positive marked point below this
component by $q$.

We orient the moduli spaces \eqref{checkequivariant}-\eqref{hatequivariant} as
follows: picking, on a slice of the automorphism action which fixes the
position of $z_d$ at 1 and $z_{out}$ at 0, the volume forms 
\begin{align}
\label{checkorientation} &-r_1\cdots r_k dz_1\wedge dz_2 \wedge \cdots
    \wedge dz_{d-1} \wedge  dr_1 \wedge d\theta_1 \wedge \cdots \wedge dr_{k}
    \wedge d\theta_k \\ 
\label{s1orientation} & r_1 \cdots r_k dz_1 \wedge dz_2
    \wedge \cdots \wedge dz_{d-1} \wedge d\theta_{k+1} \wedge  dr_1 \wedge
    d\theta_1 \wedge \cdots \wedge dr_k \wedge d\theta_k  \\
\label{hatorientation} &r_1 \cdots r_k dz_1 \wedge dz_2 \wedge \cdots
    \wedge dz_{d-1} \wedge dz_f \wedge dr_1 \wedge d\theta_1 \wedge \cdots
    \wedge dr_k \wedge d\theta_k.  
\end{align} 
Above, $(r_i,\theta_i)$ denote the polar coordinate positions of the point
$p_i$ (we could equivalently use Cartesian coordinates $(x_i, y_i)$ and
substitute $dx_i \wedge dy_i$ for every instance of $r_i dr_i \wedge
d\theta_i$, but polar coordinates are straightforwardly compatible with the
boundary stratum where $|p_k| = \frac{1}{2}$).

    A {\em Floer datum} on a stable disc $S$ in ${ }_k
    \overline{\check{\mc{R}}}_{d}^1$ or a stable disc $S$ in ${
    }_k\overline{\mc{R}}_{d}^{S^1}$ is simply a Floer datum for $S$ in the
    sense of 
    \S \ref{floeropenclosed}. 
    A {\em Floer datum} on a stable
    disc $S \in { }_k \overline{\hat{\mc{R}}}_d^1$ is a Floer datum for
    $\overline{\pi}_f(S)$.  

Again we will make a system of choices of Floer data for the above moduli spaces. A {\em Floer datum for the cyclic open-closed map} is an inductive sequence of choices, for every $k \geq 0$ and $d \geq 1$, of Floer data for every representative 
$S_0 \in  { }_k
\overline{\check{\mc{R}}}_{d}^1$, $S_1 \in {
}_k\overline{\mc{R}}_{d}^{S^1}$, and $S_2 \in { }_k
\overline{\hat{\mc{R}}}_d^1$, varying smoothly in $S_0$, $S_1$ and $S_2$, 
which satisfies the usual consistency condition: the choice of Floer datum on any boundary stratum should agree with the previously inductively chosen datum along any boundary stratum for which (is possibly a product of moduli spaces for) we have already inductively picked data.
Moreover, this choice should satisfy a series of additional requirements: 
First, for $S_0 \in  { }_k \overline{\check{\mc{R}}}_{d}^1$, 
    \begin{align}
        \label{checkcompat1}  &\textrm{At a boundary stratum of the form
        \eqref{checkstratum4}, the Floer datum for $S_0$ is} \\
        \nonumber &\textrm{equivalent to the one pulled back from ${
        }_{k-1} \overline{\check{\mc{R}}}_{d}^1$ via the forgetful map
        $\check{\pi}_i$}.  
    \end{align} 
    Next, for $S_1 \in { }_k\overline{\mc{R}}_{d}^{S^1}$, 
    \begin{align} 
        \label{s1compat1} &\textrm{On the codimension-1 loci 
        ${ }_k\overline{\mc{R}}_{d}^{S^1_i}$, 
            where $p_{k+1}$ points at $z_i$, the Floer datum }\\ 
            \nonumber &\textrm{should agree with the pullback by $\tau_i$ of the existing Floer datum for the}\\ 
            \nonumber &\textrm{open-closed map. }\\ 
        \label{s1compat2} &\textrm{The Floer datum should be
            $\kappa$-equivariant, where $\kappa$ is the map
            \eqref{labelpermuteequivariant}. } \\ 
        &\textrm{At a boundary stratum of the form 
        \eqref{s1stratum4}, the Floer datum for $S_1$ is} \\ 
        \nonumber &\textrm{conformally equivalent to the one pulled back from
            ${ }_{k-1} \overline{\mc{R}}_{d}^{S^1}$ via the forgetful map
            ${\pi}_i^{S^1}$ }.  
    \end{align}
        Finally, for $S_2 \in  { }_k \overline{\hat{\mc{R}}}_d^1$, 
        \begin{align}
        \label{hatcompat1} &\textrm{The choice of Floer
            datum on strata containing $\mc{R}^{d,f_i}$
        components should be}\\ 
        &\nonumber \textrm{constant along fibers of the forgetful map
            $\mc{R}^{d,f_i} \ra \mc{R}^{d-1}$. }\\
        \label{hatcompat2} &\textrm{The Floer datum on
            the main component $(S_2)_0$ of $\overline{\pi}_f(S_2)$
            should coincide with the}\\ 
            \nonumber &\textrm{Floer datum chosen on $(S_2)_0 \in { }_k
            \mc{R}_d^{1,free} \subset { }_k \mc{R}_d^{S^1}$}.\\
            \label{hatcompat3} &\textrm{At a boundary stratum
                of the form \eqref{hatstratum4}, the Floer
            datum on the main component} \\ 
            \nonumber &\textrm{of $S_2$ is conformally equivalent to the
                one pulled back from ${ }_{k} \overline{\mc{R}}_{d}^{S^1}$ via the forgetful} \\ 
            \nonumber &\textrm{map $\hat{\pi}_{boundary}$ }.\\
        \label{hatcompat4} &\textrm{At a boundary stratum of the form
            \eqref{hatstratum5}, the Floer datum for $S_2$ is conformally} \\
            \nonumber &\textrm{equivalent to the one pulled back from ${
            }_{k-1} \overline{\hat{\mc{R}}}_{d}^{1}$ via the forgetful map
            $\hat{\pi}_i$ }.  
        \end{align} 
The above system of requirements can be split into three broad categories: the first type concerns the compatibility with forgetful maps of Floer data along the lower strata which were not previously constrained, the second type concerns the equivariance (under a free properly discontinuous action) of the Floer data on ${ }_k\overline{\mc{R}}_{d}^{S^1}$ as well as the relationship between the Floer datum chosen here and the ones chosen on $\overline{\check{\mc{R}}}_{d}^1$ and ${ }_k \overline{\hat{\mc{R}}}_d^1$. 

\begin{prop}
A Floer datum for the cyclic open-closed map exists.
\end{prop} 
\begin{proof} 
    Since the choices of Floer data at each stage are contractible,
    this follows from the straightforward verification that, for a suitably
    chosen inductive order on strata, the conditions satisfied by the Floer
    data at various strata do not contradict each other. We use the following
    inductive order: first, say we've chosen a Floer datum for the
    $\ainf$ structure as in \S \ref{subsec:fukaya} along with a BV compatible
    Floer datum for the non-unital open-closed map following \S \ref{ocs1sec}. In particular, 
    we have chosen Floer data for the moduli spaces $\overline{\mc{R}}^{d,f_i}$ (per Appendix
    \ref{discsforgotten}), ${ }_0 \overline{\check{\mc{R}}}_d$, 
    the auxiliary moduli space ${ }_0 \overline{\mc{R}}_d^{S^1}$ and (using the conditions above) induced
    a particular choice of Floer datum on ${ }_0 \overline{\hat{\mc{R}}}_d$. Next, inductively assuming we
    have made all choices at level $k-1$ $(k>0)$, we first choose Floer data
    for ${ }_k \overline{\check{\mc{R}}}_d$ for each $d$, then ${
    }_k\overline{\mc{R}}_d^{S^1}$ for each $d$ (by pulling back a choice of
    Floer datum on the quotient by $\kappa$ in order to satisfy the
    equivariance condition), and finally note that a choice is fixed for ${
    }_k \overline{\hat{\mc{R}}}_d$ by the above constraints. 
\end{proof} 
Fixing a Floer datum for the cyclic open-closed map, we obtain, for any
$d$-tuple of Lagrangians $L_0, \ldots, L_{d-1}$, and asympotic
conditions 
\begin{equation}
    \begin{split} \vec{x} &= (x_{d},
        \ldots, x_1),\ x_i \in
        \chi(L_{i-1},L_{i \mathrm{\ mod\ }d}) \\
        y_{out} &\in \mc{O} \end{split}
    \end{equation} 
Gromov-Floer compactified moduli spaces 
\begin{align}
        \label{cycliccompactifiedspace1}{ }_k \overline{\check{\mc{R}}}_{d}^1(y_{out},
        \vec{x})\\
        \label{cycliccompactifiedspace2}{ }_k\overline{\mc{R}}_{d}^{S^1}(y_{out},
        \vec{x})\\
        \label{cycliccompactifiedspace3}{ }_k
        \overline{\hat{\mc{R}}}_d^1(y_{out},
        \vec{x}) 
\end{align} 
of maps into $M$ from an arbitrary element $S$ of the moduli spaces
\eqref{checkequivariant}, \eqref{S1equivariant}, and \eqref{hatequivariant}
respectively (or rather from $\pi_f(S)$ in the case of \eqref{hatequivariant}) 
satisfying Floer's equation using the Floer datum chosen for the
given $S$ as in \eqref{floerequationOC} with asymptotics and Lagrangian boundary conditions as in \eqref{floerequationocasymptotics} (again with the modifications as in Remarks \ref{liouvilleFloerdatumOC} or \ref{wrappedFloerdatumOC} for compact or wrapped Fukaya categories of Liouville manifolds).
The virtual dimension of each component of these moduli spaces coincides (mod 2 or exactly, depending on whether we are $\Z/2$ or $\Z$-graded) with
\begin{align} 
        \label{dimcycliccompactifiedspace1}
            \textrm{(for ${ }_k\overline{\check{\mc{R}}}_{d}^1(y_{out}, \vec{x})$) }&  
            \deg(y_{out}) - n + d -1 - \sum_{i=1}^d \deg(x_i) + 2k;\\
            \label{dimcycliccompactifiedspace2} \textrm{(for ${ }_k\overline{\mc{R}}_{d}^{S^1}(y_{out}, \vec{x})$) }& \deg(y_{out}) - n
        + d - \sum_{i=1}^d \deg(x_i) + 2k; \\
        \label{dimcycliccompactifiedspace3}\textrm{(for ${ }_k
        \overline{\hat{\mc{R}}}_d^1(y_{out}, \vec{x})$) }& \deg(y_{out}) - n + d
        - \sum_{i=1}^d \deg(x_i) + 2k.  
\end{align} 
By Assumption \ref{mainassumption}, for generic choices of Floer data, the Gromov-Floer 
compactifications of the components of virtual dimension $\leq 1$ of \eqref{cycliccompactifiedspace1} -
\eqref{cycliccompactifiedspace3} are compact manifolds-with-boundary of the expected dimension.
For rigid elements $u$ in the moduli spaces \eqref{cycliccompactifiedspace1} -
\eqref{cycliccompactifiedspace3}, (which occur for asymptotics $(y,
\vec{x})$ satisfying \eqref{dimcycliccompactifiedspace1} = 0,
\eqref{dimcycliccompactifiedspace2} = 0, or
\eqref{dimcycliccompactifiedspace3} = 0 respectively), the orientations
\eqref{checkorientation}, \eqref{s1orientation}, \eqref{hatorientation}
and \cite{Abouzaid:2010kx}*{Lemma C.4} induce isomorphisms of
orientation lines 
\begin{align} \label{cyclicorientationline1} 
    ({ }_k \check{\mc{R}}_{d}^1)_u: o_{x_d} \otimes \cdots \otimes o_{x_1}
            \ra o_y\\ 
    \label{cyclicorientationline2}({ }_k\mc{R}_{d}^{S^1})_u:
            o_{x_d} \otimes \cdots \otimes o_{x_1} \ra o_y\\
    \label{cyclicorientationline3}({ }_k \hat{\mc{R}}_d^1)_u: o_{x_d}
            \otimes \cdots \otimes o_{x_1} \ra o_y.  
\end{align} 
Summing the application of these isomorphisms over all rigid $u$ (or ``counting rigid elements'') defines the
$|o_{y_{out}}|_\K$ component of three families of operations $\check{\oc}^k$,
${\oc}^{S^1,k}$, $\hat{\oc}^k$ up to a sign twist specified below: define
\begin{align} 
    \check{\oc}^{k}([x_{d}], \ldots, [x_1]) &:=
    \sum_{u \in  { }_k \overline{\check{\mc{R}}}_d^{1}(y_{out};
    \vec{x}) \textrm{ rigid}}  (-1)^{\check{\star}_d} ({
        }_k\check{\mc{R}}_d^{1})_u([x_d], \ldots, [x_1]); \\
    \oc^{S^1,k}([x_{d}], \ldots, [x_1]) &:=
    \sum_{u \in  { }_k \overline{\mc{R}}_d^{S^1}(y_{out}; \vec{x}) \textrm{ rigid}}
            (-1)^{\star_d^{S^1}} ({ }_k\mc{R}_d^{S^1})_u([x_d],
            \ldots, [x_1]);\\ 
    \hat{\oc}^{k}([x_{d}], \ldots, [x_1]) &:= 
    \sum_{u \in  { }_k \overline{\hat{\mc{R}}}_d^{1}(y_{out}; \vec{x}) \textrm{ rigid}}
    (-1)^{\hat{\star}_d} ({ }_k\hat{\mc{R}}_d^{1})_u([x_d], \ldots,
    [x_1]).  
\end{align} 
where the signs are given by
\begin{align} \check{\star}_d &= \deg(x_d) + \sum_{i} i
    \cdot \deg(x_i).\\ \star_d^{S^1} = \clubsuit_d &=
    \sum_{i=1}^d (i+1) \cdot \deg(x_i) + \deg(x_d) + d-1 = \check{\star}_d + \maltese_d - 1\\
    \hat{\star}_d &= \sum_{i} i \cdot \deg(x_i).
\end{align}
A codimension 1 analysis of the moduli spaces \eqref{cycliccompactifiedspace1}
and \eqref{cycliccompactifiedspace3} reveals: 
\begin{prop} The following equations hold for each $k\geq 0$: 
\begin{align} 
    (-1)^n \sum_{i = 0}^k \delta_i \check{\oc}^{k-i} &= \oc^{S^1,k-1} + \check{\oc}^{k} b\\ 
   (-1)^n \sum_{i = 0}^k \delta_i \hat{\oc}^{k-i} &=  \hat{\oc}^{k} b' + \check{\oc}^k (1 - t).
\end{align} 
\end{prop} 
\begin{proof} 
The boundary of the 1-dimensional components of \eqref{cycliccompactifiedspace1} are covered
by the (rigid components of) the following types of strata: 
\begin{itemize} 
    \item spaces of maps with domain lying on the codimension 1 boundary of the
            moduli space, i.e., in \eqref{checkstratum1}-\eqref{checkstratum4}

    \item semi-stable breakings, namely those of the form 
         \begin{align} 
             &{ }_k\overline{\check{\mc{R}}}_d^1(y_1; \vec{x}) \times
             \overline{\mc{M}}(y_{out}; y_1) \\ 
             &\overline{\mc{R}}^1(x; x_i) \times
             { }_k\overline{\check{\mc{R}}}_d^1(y_{out}; \tilde{\vec{x}})
        \end{align} 
where again $\tilde{\vec{x}}$ denotes the collection of inputs $\vec{x}$ with $x_i$ replaced with $x$.
\end{itemize}

All together, this implies, up to sign, that 
\begin{equation}
    \label{initialcheckequation} (-1)^n \sum_{i = 0}^k \delta_i \check{\oc}^{k-i} =
    \oc^{S^1,k-1} + \check{\oc}^{k} b + \sum_{i=1}^{k-1}\check{\oc}^{k,i,i+1},
\end{equation} 
where $\check{\oc}^{k,i,i+1}$ is an operation corresponding with
some sign twist to \eqref{checkstratum4}. \eqref{initialcheckequation} is of
course a shorthand for saying, for a tuple of $d$ cyclically composable
morphisms $x_d, \ldots, x_1$ (recalling the abuse of notation $x_i:=[x_i]$), that 
\begin{equation} 
    \begin{split} 
        (-1)^n \sum_{i = 0}^k \delta_i \check{\oc}^{k-i}_d(x_d, \ldots, x_1) &=
        \oc^{S^1,k-1}_d(x_d, \ldots, x_1) +   \sum_{i=1}^{k-1}
        \check{\oc}^{k,i,i+1}_d (x_d, \ldots, x_1)\\ 
        &+\sum_{i,s}(-1)^{\maltese_1^s} \check{\oc}^k_{d-i+1} (x_d, \ldots,
        x_{s+i+1}, \mu^i(x_{s+i}, \ldots, x_{s+1}), x_s, \ldots, x_1)\\
        &+\sum_{i,j} (-1)^{\sharp_j^i} \check{\oc}^k (\mu^{i + j + 1}(x_i,
        \ldots, x_1, x_d, \ldots, x_{d-j}), x_{d-j-1}, \ldots, x_{i+1}).  
    \end{split}
\end{equation}
 
 We first note that in fact the operation $\check{\oc}^{k,i,i+1} = \sum_d
 \check{\oc}_d^{k,i,i+1}$ is zero, because by condition \eqref{checkcompat1},
 the Floer datum chosen for elements $S$ in \eqref{checkstratum4} are constant
 along the one-dimensional fibers of $\check{\pi}_i$. Hence, elements of the
 moduli space with source in \eqref{checkstratum4} are never rigid (see Lemma
 \ref{weaks1actionSHlemma} for an analogous and more detailed explanation).

Thus, it suffices to verify that the signs coming from the codimension 1
boundary are exactly those appearing in \eqref{initialcheckequation}.  We can
safely ignore studying any signs for the vanishing operations such as
$\hat{\oc}^{k,i,i+1}$.  The remaining sign analysis is exactly as in
Proposition \ref{BVcompactifiedequation1}; more precisely note that the
forgetful map $\check{F}_k: { }_k \check{\mc{R}}_d^1 \ra { }_1
\check{\mc{R}}_d^1$ which forgets $p_1, \ldots, p_{k-1}$ has complex oriented
fibers, and in particular (since the marked points $p_i$ contribute complex
domain orientations and do not introduce any new orientation lines) the sign
computations sketched in Proposition \ref{BVcompactifiedequation1} carry over
for any strata whose domain is pulled back from a boundary stratum of ${ }_1
\check{\mc{R}}_d^1$ (in turn, as described in Proposition
\ref{BVcompactifiedequation1}, the sign computations for ${ }_1
\check{\mc{R}}_d^1$ largely reduce to those for ${ }_0 \check{\mc{R}}_d^1$.
This verifies \eqref{initialcheckequation}.

Similarly, for the hat moduli space, an analysis of the boundary of
1-dimensional moduli spaces of maps tells us, up to sign verification: 
 \begin{equation}\label{initialhatequation} 
     (-1)^n\sum_{i = 0}^k \delta_i \hat{\oc}^{k-i} =  \hat{\oc}^{k} b' +
     \check{\oc}^k (1 - t) + \hat{\oc}^{k,k,k+1} +
     \sum_{i=1}^{k-1}\hat{\oc}^{k,i,i+1}, 
 \end{equation} 
where $\hat{\oc}^{k,k,k+1}$ is an operation corresponding with
some sign twist to \eqref{hatstratum4} and $\hat{\oc}^{k,i,i+1}$ is an
operation corresponding with some sign twist to \eqref{hatstratum5}.  The
conditions \eqref{hatcompat3}-\eqref{hatcompat4} similarly imply that
$\hat{\oc}^{k,k,k+1}$ and $\hat{\oc}^{k,i,i+1}$ are zero, so it is not
necessary to even establish what the signs for these terms are.

To verify signs for \eqref{initialhatequation}, we apply the principle
discussed in the proof of Lemma \ref{ochatchainold}, in which by treating the
auxiliary boundary marked point $z_f$ as possessing a ``formal unit element
asymptotic constraint $e_+$,'' therefore viewing 
$\hat{\oc}^k(x_d \otimes \cdots \otimes x_1 ):= \hat{\oc}^k(x_d,  \ldots, x_1 )$ formally as ``$\hat{\oc}^k(e^+ \otimes x_d
\otimes \cdots \otimes x_1)$'' the signs for the equations
\eqref{initialhatequation} applied to strings $(x_d \otimes \cdots \otimes x_1)$ of length $d$ follow from the sign computations for $\check{\oc}$
applied to strings $(e^+ \otimes x_d \cdots \otimes \cdots
x_1)$ of length $d+1$  (we note that this analysis applies to the term $\hat{\oc}^{k,k,k+1}$ as
well, which is the hat version of $\oc^{S^1,k}$; however, the former operation
happens to be zero because extra symmetries imply the moduli space controlling
this operation is never rigid).
\end{proof}

Next, by decomposing the moduli space ${ }_k \mc{R}_d^{S^1}$ into sectors, we
can write the auxiliary operation $\oc^{S^1,j}$ in terms of $\hat{\oc}^j$ and
Connes' $B$ operator: 
\begin{prop} \label{ocks1Brelation} As chain-level operations,
    \begin{equation} \oc^{S^1,k} = \hat{\oc}^k \circ B^{nu}.  \end{equation}
    \end{prop} 
    \begin{proof} The proof directly emulates Proposition \ref{ocS1Brelation},
        and as such we will give fewer details. We begin by defining operations
        \begin{equation}
           \hat{\oc}^{k}_{d,\tau_i} 
        \end{equation}
        associated to various ``sectors'' of the $k+1$st marked point $p_{k+1}$
        of ${ }_k \mc{R}_d^{S^1}$, for $i \in \Z/d\Z$. Once more, to gain
        better control of the geometry of these sectors in the compactification
        (when the sector size
        can shrink to zero), we pass to an alternate model for the
        compactification: define 
        \begin{equation}
            { }_k \mc{R}_{d,\tau_i}^{1}
        \end{equation}
        to be the abstract moduli space of discs with $d+1$ boundary punctures,
        $z_1, \ldots, z_i, z_f, z_{i+1}, \ldots, z_d$ arranged in counterclockwise
        order, one interior negative puncture $z_{out}$ with asymptotic marker,
        and $k$ additional interior auxiliary marked points $p_1, \ldots, p_k$
        which are {\it strictly radially ordered} with norms
in $(0,\frac{1}{2})$
        for a representative fixing $z_0$ at
        1 and $z_{out}$ at 0:
        \begin{equation}
            0 < |p_1| < \ldots < |p_k| < \frac{1}{2}.
        \end{equation}
        Moreover, as before,
        \begin{equation}
                \textrm{the asymptotic marker on $z_{out}$ points in the direction
            $\theta_1$ (or towards $z_f$ if $k=0$).} 
        \end{equation}
        There is a bijection 
        \begin{equation}\label{cyclicallypermutesector}
            \tau_i: { }_k \mc{R}_{d,\tau_i}^{1} \ra { }_k\hat{\mc{R}}_{d}^1
        \end{equation} 
        given by cyclically permuting boundary labels, and in particular we
        also have an {\it auxiliary-rescaling map} as in
        \eqref{equivariantforgetful}
        \begin{equation}\label{equivariantforgetfulsector}
            { }_k \mc{R}_{d,\tau_i}^{1} \ra { }_k \mc{R}_{d}^{S^1_{i,i+1}},
        \end{equation}
        which, for a representative with $|z_{out}| = 0$, adds a point
        $p_{k+1}$ on the line between $z_{out}$ and $z_f$ with
        $|p_{k+1}|=\frac{1}{2}$ and deletes $z_f$. We choose orientations on 
        ${ }_k \mc{R}_{d,\tau_i}^{1}$ to be compatible with
        \eqref{equivariantforgetfulsector}; more concretely, for a slice fixing
        the positions of $z_{out}$ and $z_d$, consider the top form 
        \begin{equation}
            r_1 \cdots r_k dz_1 \wedge dz_2 \wedge \cdots \wedge dz_{d-1} \wedge dz_d \wedge dz_f \wedge dr_1 \wedge d\theta_1 \wedge \cdots \wedge dr_k \wedge d\theta_k
        \end{equation} 
        The compactification ${ }_k \overline{\mc{R}}_{d,\tau_i}^{1}$ is
        inherited from the identification \eqref{cyclicallypermutesector}; the
        salient point is that we treat bubbled off boundary strata containing
        the point $z_f$ as coming from $\mc{R}^{d,f_i}$, the moduli space of
        discs with $i$th marked point forgotten (where the $i$th marked point
        is $z_f$), constructed in Appendix \ref{discsforgotten}.

        We choose as a Floer datum for ${ }_k
        \overline{\mc{R}}_{d,\tau_i}^{1}$ the pulled back Floer datum from
        \eqref{cyclicallypermutesector}, it automatically then exists and is
        universal and consistent as desired. Moreover we have chosen
        orientations as in the case $k=0$ so that the auxiliary rescaling map
        \eqref{equivariantforgetfulsector} is an oriented diffeomorphism
        extending to a map between compactifications.

        Thus, given a Lagrangian labeling $\{L_0, \ldots, L_{d-1}\}$ and
        compatible asymptotics $\{x_1, \ldots, x_d; y_{out}\}$ we obtain a
        moduli space of maps satisfying Floer's equation with the chosen
        boundary and asymptotics
        \begin{equation}\label{sectormoduli2}
            { }_k \overline{\mc{R}}_{d,\tau_i}^{1}(y_{out}; \vec{x}) := { }_k \overline{\hat{\mc{R}}}_d^1(y_{out}; x_{i-1}, \ldots, x_1, x_d, \ldots, x_i),
        \end{equation}
        which is (for components of virtual dimension $\leq 1$) is a manifold of dimension equal to the virtual dimension of the right hand
        side, which is $\deg(y_{out}) - n + d- \sum_{j=1}^{d} \deg(x_j) + 2k$ (with $\Z$-gradings or mod 2 if working with $\Z/2$-gradings). The
        isomorphisms of orientation lines 
        \begin{equation}
            ({ }_k \mc{R}_{d,\tau_i}^{1})_u: o_{x_d} \otimes \cdots \otimes o_{x_1} \ra o_{y_{out}}
        \end{equation}
        induced by elements $u$ of the zero-dimensional components of
        \eqref{sectormoduli2} define the $|o_{y_{out}}|_\K$ component of the
        operation $\hat{\oc}^{k}_{d,\tau_i}$, up to the following sign twist:
        \begin{equation} 
            \begin{split}
                \hat{\oc}^{k}_{d,\tau_i}([x_{d}], \ldots, [x_1]) &:= 
                \sum_{u \in  { }_k\overline{\hat{\mc{R}}}_{d,\tau_i}^{1}(y_{out}; \vec{x}) \textrm{ rigid}}  (-1)^{\hat{\star}_d} ({ }_k {\mc{R}}_{d,\tau_i}^{1})_u([x_d], \ldots, [x_1]),\\
                \star_d^{S^1} &= \sum_{i=1}^d (i+1) \cdot \deg(x_i) + \deg(x_d) + d - 1.
            \end{split}
        \end{equation}
        Now, exactly as in Lemma \ref{sectordecomposition}, there is a
        chain-level equality of signed operations
        \begin{equation}\label{signedequalityk}
            \oc^{S^1,k}_d = \sum_{i=0}^{d-1} \hat{\oc}^{k}_{d,\tau_i}.
        \end{equation} 
        We recall the geometric statement underlying this; the
        point is by construction there
        is an oriented embedding
        \begin{equation}\label{disjointembeddingksector}
            \coprod_{i} { }_k{\mc{R}}^1_{d,\tau_i} \stackrel{\coprod_i \pi_f^{i}}{\lra} \coprod_i { }_k \mc{R}^{S^1_{i,i+1}}_d \hookrightarrow { }_k\mc{R}^{S^1}_d,
        \end{equation}
    compatible with Floer data, covering all but a codimension 1 locus in the
    target, and moreover all the sign twists defining the operations
    $\oc^{k}_{d,\tau_i}$ are chosen to be compatible with the sign twist in
    the operation $\oc^{S^1,k}$ (this uses the fact that the Floer data on
    ${ }_k\mc{R}^{S^1_{i,i+1}}$ agrees with the data on ${ }_k\hat{\mc{R}}_d^1$
    via the cyclic permutation map $\kappa^{-i}$ by \eqref{S1tauequivariance}).
    After perturbation zero-dimensional
    solutions to Floer's equation can be chosen to come from the complement of
    any codimension 1 locus in the source abstract moduli space, implying the
    equality \eqref{signedequalityk}.

Finally, all that remains is a sign analysis whose conclusion is that, 
\begin{equation}
    \hat{\oc}^{k}_{d,\tau_i} = \hat{\oc}^k_d \circ s^{nu} \circ t^i
\end{equation}
where $s^{nu}$ is the operation arising from changing a check term to a hat
term with a sign twist \eqref{snu}.  (The equality up to comparing signs is
immediate, as the operations are constructed with identical Floer data and
hence involve counts of identical moduli spaces).
The details of this sign comparison are exactly the
same as in Proposition \ref{sectorsignorientation}, including with signs, as
when orienting the moduli of maps, the additional marked points $p_1, \ldots
p_{k}$ only contribute complex orientations to the moduli spaces of domains
(and no additional orientation line terms).
\end{proof}

\begin{proof}[Proof of Proposition \ref{cyclicOCequations}]
    This is an immediate corollary of the previous two propositions.
\end{proof}

We now collect all of this information to finish the proof of our main result.

\begin{proof}[Proof of Theorem \ref{thm:mainresult1}]
    The pre-morphism $\widetilde{\oc} \in \r{Rhom}_{S^1}^n(\r{CH}_*^{nu}(\f), CF^*(M))$, written $u$-linearly as $\sum_i \oc^k u^k$, where $\oc^k = \check{\oc}^k
    \oplus \hat{\oc}^k$ are as constructed above, satisfies $\partial
    \widetilde{\oc} = 0$ by Prop. \ref{cyclicOCequations}, hence
    $\widetilde{\oc}$ is closed, or an $S^1$-complex homomorphism, also known
    as an $\ainf$ $C_{-*}(S^1)$ module homomorphism (see \S
    \ref{circleactionsubsection}). Note that $[\oc^0] =
    [\oc] = [\check{\oc}]$ where the first equality is by definition and the
    second is by Corollary \ref{homologylevelnonunitaloc}, hence
    $\widetilde{\oc}$ is an enhancement of $\check{\oc}$ (as defined in \S
    \ref{circleactionsubsection}).
\end{proof}

\begin{proof}[Proof of Corollary \ref{cor:cyclicOC}]
    This is an immediate consequence of Theorem \ref{thm:mainresult1} and the
    induced homotopy-invariance properties for equivariant homology groups
    discussed in \S \ref{section:s1action}, particularly Cor.
    \ref{S1homotopyinvariance} and Prop. \ref{functorialitysequences}.
\end{proof}

\subsection{Variants of the cyclic open-closed map} \label{sec:OCvariants}
\subsubsection{Using singular (pseudo-)cycles instead of Morse cycles}\label{pseudocycles}
Let $M$ be Liouville or compact and admissible (in which case by our convention $\bar{M} = M$ and
$\partial \bar{M} = \emptyset$),\footnote{Technically we should write $\r{QH}^*(M)$ in the latter case, but additively $\r{QH}^*(M) = H^*(M)$ and correspondingly no sphere bubbling occurs in the moduli spaces we define here, so there is no difference for the purposes of this discussion.} 
and let us consider the version of $\widetilde{\oc}$ with target the relative
cohomology $H^*(\bar{M}, \partial \bar{M})$ as in \S \ref{subsubsec:relh}.
Instead of using a $C^2$ small Hamiltonian to define the Floer complex
computing $H^{*+n}(\bar{M}, \partial \bar{M})$ (which we only did for simultaneous
compatibility with the symplectic cohomology case), we can pass to a geometric
cycle model for the group, and then build a version of the map
$\widetilde{\oc}$ with such a target,  which simplifies many of the
constructions in the previous section (in the sense that the  codimension 1
boundary strata of moduli spaces, and hence the equations satisfied by
$\widetilde{\oc}$, are strictly a subset of the terms appearing above). As
such, it will be sufficient to fix some notation for the relevant moduli
spaces, and state the relevant simplified results.

We denote by 
\begin{align} 
    \label{checkequivariant_cycle} { }_k \check{\mc{P}}_{d}^1\\ 
    \label{S1equivariant_cycle}   { }_k \mc{P}_{d}^{S^1}\\ 
    \label{hatequivariant_cycle} { }_k \hat{\mc{P}}_d^1;
\end{align} 
copies of the abstract moduli spaces
\eqref{checkequivariant}-\eqref{hatequivariant} where the interior puncture
$z_{out}$ is filled in and replaced by a marked point $\bar{z}_{out}$, {\em
without any asymptotic marker}. The compactifications of these moduli spaces
are exactly as before, except that the auxiliary points $p_1, \ldots, p_k$ are
now allowed to coincide with $\bar{z}_{out}$, without breaking off an
angle-decorated cylinder / element of $\mc{M}_r$ (in the language of \S
\ref{sec:angledeccylinder}). In other words, the real blow-up of
Deligne-Mumford compactifications at $z_{out}$ described in \S
\ref{modulispaces}, which was responsible for the boundary strata containing
$\mc{M}_r$ factors, {\em no longer occurs} (but all other degenerations do
occur). Correspondingly the codimension-1 boundaries of compactified moduli
spaces have all of the strata as before except for strata containing
$\mc{M}_r$'s.

Inductively choose smoothly varying families Floer data as before on these moduli spaces of
domains, satisfying all of the requirements and consistency conditions as before (except for any
consistency conditions involving $\mc{M}_r$ moduli spaces, which no longer
occur on the boundary).  For a basis $\beta_1, \ldots, \beta_s$ of smooth
(pseudo)-cycles in homology $H_*(M)$ whose Poincar\'{e} duals $[\beta_i^\vee]$ generate the
cohomology $H^*(\bar{M}, \partial \bar{M})$, one obtains moduli spaces
\begin{align} 
    \label{checkequivariant_cycle_maps} { }_k \check{\mc{P}}_{d}^1(\beta_i; \vec{x})\\ 
    \label{S1equivariant_cycle_maps}   { }_k \mc{P}_{d}^{S^1}(\beta_i,\vec{x}) \\ 
    \label{hatequivariant_cycle_maps} { }_k \hat{\mc{P}}_d^1(\beta_i, \vec{x}) ;
\end{align} 
of moduli spaces of maps into $M$ with source an arbitrary element of the
relevant domain moduli space, satisfying Floer's equation as before, with
Lagrangian boundary and asymptotics $\vec{x}$ as before, {\em with the
additional point constraint that $\bar{z}_{out}$ lie on the cycle $\beta_i$}. As
before, standard methods ensure that 0 and 1-dimensional moduli spaces are (for generic
choices of perturbation data and/or $\beta_i$) transversely cut out manifolds
of the ``right'' dimension and boundary, which is all that we need.

Then, define the coefficient of $[\beta_i^{\vee}] \in H^*(\bar{M}, \partial
\bar{M})$ in $\check{\oc}^k(x_d \otimes \cdots \otimes x_1)$ to be given by signed counts
(with the same sign twists as before) of the moduli spaces
\eqref{checkequivariant_cycle_maps}; similarly for $\hat{\oc}^k$ and
$\oc^{S^1,k}$ using the moduli spaces \eqref{hatequivariant_cycle_maps} and \eqref{S1equivariant_cycle_maps}.  
A simplification of the arguments already given (in which the $\delta_k$
operations no longer occur, but every other part of the argument carries
through) implies that:
\begin{prop}
    The pre-morphism $\widetilde{\oc} = \sum_{i=0}^{\infty} \oc^k u^k \in
    \r{Rhom}_{S^1}^n(\r{CH}_*^{nu}(\mc{F}, \mc{F}), H^*(\bar{M}, \partial
    \bar{M}))$ satisfies 
    \begin{equation}
        \widetilde{\oc} \circ b_{eq} = 0,
    \end{equation}
    where $b_{eq} = b^{nu} + u B^{nu}$. In other words, $\widetilde{\oc}$ is a
    homomorphism of $S^1$-complexes between $\r{CH}_*^{nu}(\mc{F}, \mc{F})$
    with its strict $S^1$ action and $H^*(\bar{M}, \partial \bar{M})$ with its
    trivial $S^1$ action.
\end{prop}

As usual this model of $\widetilde{\oc}$ again induces maps $\widetilde{\oc}^{+/-/\infty}$ between homotopy orbit complexes, homotopy fixed point complexes etc.; note the relevant equivariant homology chain complexes are particularly simple for the latter $H^*(\bar{M}, \partial \bar{M})$, seeing as there is no differential and trivial circle action
(e.g., $H^*(\bar{M}, \partial \bar{M})_{hS^1} = (H^*(\bar{M}, \partial \bar{M}) ( ( u ) ) / u H^*(\bar{M}, \partial \bar{M}) [ [ u ] ], \delta_{eq} = 0)$).

\subsubsection{Compact Lagrangians in non-compact manifolds}\label{sec:compactopenclosed}

Now let us explicitly restrict to the case of $M$ a Liouville manifold, and denote
by $\mc{F} \subset \mc{W}$ the full-subcategory consisting of a finite
collection of compact exact Lagrangian branes contained in the compact region
$\overline{M}$. By Poincar\'{e} duality we may think of the map 
$\oc$ (and its cyclic analogue, $\widetilde{\oc}$) with target $H^*(\bar{M}, \partial \bar{M})$ as a
pairing $\r{CH}_*^{nu}(\mc{F},\mc{F}) \otimes C^*(M) \to \K[n]$. In this case,
there is a non-trivial refinement of this pairing to 
\begin{equation}
    \label{ocpairing} \oc_{cpct}: \r{CH}_{*}(\mc{F}) \otimes SC^*(M) \ra \K[-n].
\end{equation}
where $SC^*(M)$ is the {\em symplectic cohomology} cochain complex. 
\begin{rem}
    The refinement \eqref{ocpairing} relies on extra flexibility in Floer
    theory for compact Lagrangians compared to non-compact Lagrangians (compare
    Remarks \ref{liouvilleFloerdatum} and \ref{wrappedFloerdatum}), first
    alluded to in this form in \cite{SeidelEquivariant}. This extra flexibility allows us to
    define operations without
    outputs --- and in particular study a version of the open-closed map where
    the interior marked point and boundary marked points are all inputs--- for instance by
    Poincar\'{e} dually treating some boundary inputs as outputs with ``negative
    weight''.

    One way to implement such operations, using the type of Floer data
    discussed in Remark \ref{wrappedFloerdatum}, is by allowing the {\em
    sub-closed one-form $\alpha_S$} used in Floer theoretic perturbations to
    have complete freedom along boundary conditions corresponding to compact
    Lagrangians; in contrast along possibly non-compact Lagrangian boundary
    conditions, $\alpha_S$ is required to vanish in order to appeal to the
    integrated maximum principle.  In particular, if we allow $\alpha_S$ to be
    non-vanishing along boundary components, Stokes' theorem no longer implies
    that $\alpha_S$ being sub-closed implies that the total ``output'' weights
    must be greater than the total ``input'' weights.  
\end{rem}
\begin{rem}
    The existence of a map $SC^*(M) \to \r{CH}_*(\mc{F})^{\vee}[-n]$
    is well known. Namely, categories $\cc$ with a {\em weak proper Calabi-Yau
    structure} of dimension $n$ (such as the Fukaya category of compact
    Lagrangians; see e.g., \cite{Seidel:2008zr}*{(12j)},
\cite{Seidel:2010aa}*{Proof of Prop. 5.1, Step 1} and \cite{Sheridan:2016}*{\S
2.8})
    come equipped with isomorphisms between the dual of Hochschild chains and
    Hochschild co-chains $\r{CH}_*(\cc)^{\vee}[-n] \simeq \r{CH}^*(\cc)$, and
    the existence of a map $SC^*(M) \to \r{CH}^*(\mc{F})$ was observed
    in \cite{Seidel:2002ys}.  
\end{rem}
The geometric moduli spaces used to establish our main result apply verbatim in
this case (with the interior marked point changed to an input, and the ordering
of the auxiliary marked points $p_1, \ldots, p_k$ appearing in the cyclic
open-closed map reversed). In this case, the operations associated to such
moduli spaces imply:
\begin{prop}
    Consider $\r{CH}_*^{nu}(\f) \otimes SC^*(M)$ as an $S^1$-complex with its
    diagonal $S^1$ action (see Lemma \ref{diagonalaction} in \S
    \ref{circleactionsubsection}), and $\K = \underline{\K}^{triv} \in S^1\mod$ with its
    trivial $S^1$-complex structure. Then, the map from $\r{CH}_*(\f)
    \otimes SC^*(M)$ to $\K$ can be enhanced to a homomorphism of
    $S^1$-complexes
    \[
        \widetilde{\oc}_{cpct} \in \r{Rhom}_{S^1}^n(\r{CH}_*^{nu}(\f) \otimes SC^*(M), \K)
    \]
    e.g., $\widetilde{\oc}_{cpct}$ satisfies $\partial
\widetilde{\oc}_{cpct} = 0$. In other words, in the notation of \S \ref{sec:ulinear} there exists a map $\widetilde{\oc}_{cpct, eq} = \sum_{i=0}^\infty \oc_{cpct, i} u^i: \r{CH}_*^{nu}(\f) \otimes SC^*(M) \to \K[ [ u] ])$ of pure degree $n$, with $[\oc_{cpct, 0}] = [\oc_{cpct}]$ such that 
    \[
        \widetilde{\oc}_{cpct, eq} \circ ( (-1)^{\deg(y)} b_{eq}(\sigma) \otimes y + \sigma \otimes \delta^{SC}_{eq}(y)) = 0.
    \]
\end{prop}

To clarify the relevant moduli spaces used, we define the spaces
\begin{align} 
    \label{checkequivariant_pairing} { }_k \check{\mc{R}}_{d,cpct}^1\\ 
    \label{S1equivariant_pairing}   { }_k \mc{R}_{d,cpct}^{S^1}\\ 
    \label{hatequivariant_pairing} { }_k \hat{\mc{R}}_{d,cpct}^1
\end{align} 
to be copies of the abstract moduli spaces
\eqref{checkequivariant}-\eqref{hatequivariant} where the interior puncture
$z_{out}$ is now a {\em positive} puncture (still equipped with an asymptotic
marker); and all of the other inputs and auxiliary points are as before, except we've reversed the order of the labelings $p_1, \ldots, p_k$ (for notational convenience), so the ordering constraints all now read as $0 < |p_k| < \cdots < |p_1| < \frac{1}{2}$.
    The
compactified moduli spaces have boundary strata agreeing with the boundary
strata of the compactified \eqref{checkequivariant}-\eqref{hatequivariant},
except now the $\mc{M}_r$ cylinders break ``above'' the ${ }_{k-r}
\mc{R}_{d,cpct}^1$ (equipped with $\check{ }$, $\hat{ }$, or $S^1$
decoration) discs instead of ``below.'' The reversal of the ordering of auxiliary marked points is designed to be compatible with the ordering of the auxiliary marked points on the $\mc{M}_r$ moduli spaces when it breaks ``above'' (as in $\mc{M}_r$ the label numbers of the auxiliary marked points increase from top to bottom).

Equipping these moduli spaces with perturbation data satisfying the same consistency conditions and other requirements as before, and counting solutions with
sign twists as before defines the terms of the pre-morphism
exactly as in the previous subsections (with identical analysis to show that,
for instance, the operation corresponding to  ${ }_k \mc{R}_{d,cpct}^{S^1}$ is
the operation corresponding to ${ }_{k-1} \hat{\mc{R}}_{d,cpct}^1$ composed
with Connes' B operator, the boundary strata in which $|p_i|$ and
$|p_{i+1}|$ are coincident contributes trivially, and so on).

\section{Calabi-Yau structures}\label{section:cystructures}
\subsection{The proper Calabi-Yau structure on the Fukaya category}\label{secproperCY}
Here we review the notion of a {\em proper Calabi-Yau structure}, following
Kontsevich-Soibelman \cite{Kontsevich:2009ab}, and construct proper
Calabi-Yau structures on Fukaya categories of compact Lagrangians (in a compact
admissible or Liouville manifold).  A proper Calabi-Yau structure induces 
chain-level topological field theoretic operations on the Hochschild chain complex
of the given category, controlled by the open moduli space of curves with
marked points equipped with asymptotic markers, at least one of which is an
input
\cite{Costello:2006vn, Kontsevich:2009ab}. Note that Costello's work
\cite{Costello:2006vn} constructing field-theoretic operations has the (a
priori stronger) requirement that the underlying $\ainf$ category be {\em
cyclic}, but in characteristic zero any proper Calabi-Yau structure determines
a unique quasi-isomorphism between the underlying $\ainf$ category and a cyclic
$\ainf$ category \cite[Thm.  10.7]{Kontsevich:2009ab}; see Remark
\ref{rem:cyclic} for more discussion.

We say an $\ainf$ category $\mc{A}$ is \emph{proper} (sometimes called
\emph{compact}) if its cohomological morphism spaces $H^*(\hom_{\mc{A}}(X,Y))$ have total finite rank over
$\K$, for each $X,Y$.  Recall that for any object $X \in \mc{A}$, there is an inclusion of chain
complexes $\hom(X,X) \ra \r{CH}_*(\mc{A})$ inducing a map $[i]:
H^*(\hom(X,X)) \ra \r{HH}_*(\mc{A})$.  
\begin{defn}\label{weakproperCY}
    Let $\mc{A}$ be a proper category.  A chain map $tr:
    \r{CH}_{*+n}(\mc{A}) \ra \K$ is called a \emph{weak proper
    Calabi-Yau structure}, or \emph{non-degenerate trace} of dimension $n$ if,
    for any two objects $X, Y \in \ob \mc{A}$, the composition
    \begin{equation}
        H^*(\hom_{\mc{A}} (X,Y)) \otimes H^{n-*}(\hom_{\mc{A}} (Y,X)) \stackrel{[\mu^2_{\mc{A}}]}{\ra} H^n(\hom_{\mc{A}} (Y,Y)) \stackrel{[i]}{\ra} \r{HH}_n(\mc{A}) \stackrel{[tr]}{\ra \K}
    \end{equation}
    is a perfect pairing (this non-degeneracy property evidently only depends on the homology class $[tr]$). A chain map from the non-unital Hochschild complex $tr: \r{CH}_{*+n}^{nu}(\mc{A}) \ra \K$ is called a weak proper Calabi-Yau structure if composition with the inclusion $\r{CH}_{*+n}(\mc{A}) \subset \r{CH}_{*+n}^{nu}(\mc{A})$ is a weak proper Calabi-Yau structure in the sense above.
\end{defn}
\begin{rem}\label{ccybimodule}
    In the symplectic literature, weak proper Calabi-Yau
    structures of dimension $n$ are sometimes defined as bimodule quasi-isomorphisms
    $\mc{A}_{\Delta} \stackrel{\sim}{\ra} \mc{A}^{\vee}[n]$, where
    $\mc{A}_{\Delta}$ denotes the {\it diagonal bimodule} and $\mc{A}^{\vee}$
    the {\it linear dual diagonal bimodule} --- see
    \cite{Seidel:2008zr}*{(12j)} and \S \ref{subsec:smoothCY} for brief
    conventions on $\ainf$ bimodules
    (see also \cite{Tradler:2008fk}).  To explain the relationship between this
    definition and the one above, which has sometimes been called a {\em weakly
    cyclic structure} or {\em $\infty$-inner product} \cite{Tradler:2008fk,
    Sheridan:2016}, note that for any compact $\ainf$ category
    $\mc{A}$, there are quasi-isomorphisms (with explicit chain-level models)
    \begin{equation} 
        (\r{CH}_*(\mc{A}))^{\vee} = \r{CH}^*(\mc{A}, \mc{A}^{\vee}) \stackrel{\sim}{\leftarrow} \hom_{\mc{A}\!-\!\mc{A}}(\mc{A}_{\Delta}, \mc{A}^{\vee})
    \end{equation}
    where $\hom_{\mc{A}\!-\!\mc{A}}$ denotes morphisms in the category of
    $\ainf$ bimodules (see e.g., \cite{Seidel:2008cr} or \cite{ganatra1_arxiv}). 
    Under this correspondence, non-degenerate morphisms from $\r{HH}_*(\mc{A})
    \ra \K$ as defined above correspond precisely (cohomologically) to weak
    Calabi-Yau structures, e.g., those bimodule morphisms from
    $\mc{A}_{\Delta}$ to $\mc{A}^{\vee}$ which are cohomology isomorphisms.
\end{rem}
Remember that the Hochschild chain complex of an $\ainf$ category $\mc{A}$ comes
equipped with a natural chain map to the (positive) cyclic homology chain complex, the {\em projection to homotopy orbits} \eqref{projectionhomotopyorbits}
\[
    pr: \r{CH}_*^{nu}(\mc{A}) \ra \r{CC}^+_*(\mc{A})
\]
modeled on the chain level by the map that sends $\alpha \mapsto \alpha \cdot
u^0$, for $\alpha \in \r{CH}^{nu}(\mc{A})$ (compare \eqref{projectionhomotopyorbits_explicit}).
\begin{defn}[c.f. Kontsevich-Soibelman \cite{Kontsevich:2009ab}]
A \emph{(strong) proper Calabi-Yau structure of degree $n$} is a chain map
\begin{equation}
    \tilde{tr}: \r{CC}^+_*(\mc{A})\ra \K[-n]
\end{equation}
from the (positive) cyclic homology chain complex of $\mc{A}$ to $\K$ of degree $-n$, such that the
induced map $tr = \tilde{tr} \circ pr: \r{CH}_*^{nu}(\mc{A}) \ra \K[-n]$ (or equivalently the composition $\check{tr}$ of $tr$ with the inclusion $\r{CH}_*(\mc{A}) \subset \r{CH}_*^{nu}(\mc{A})$)
is a weak proper Calabi-Yau structure.
\end{defn}
Via the model for cyclic chains given as $\r{CC}_*^+(\mc{A}) :=
(\r{CH}_*^{nu}(\mc{A}) (( u ) ) / u \r{CH}_*^{nu}(\mc{A}) [ [ u ] ], b +
u B^{nu})$, such an element $\tilde{tr}$ takes the form
\begin{equation}
    \tilde{tr} := \sum_{i=0}^{\infty} tr^k u^k
\end{equation}
where 
\begin{equation}
    tr^k := (\check{tr}^k \oplus \hat{tr}^k): \r{CH}_*^{nu}(\mc{A}) \ra \K[-n-2k].
\end{equation}
We now complete the proof of Theorem \ref{mainthm3} described and sketched in the Introduction:
first, define the putative proper Calabi-Yau structure as the composition:
\begin{equation}\label{trcomposition}
    \tilde{tr}: \r{CC}^+_*(\mc{F}) \stackrel{\widetilde{\oc}^+}{\ra} C^{*+n}(\bar{M}, \partial \bar{M}) \otimes_{\K} \K( ( u ) ) / u \K [ [ u ] ]  \ra \K
\end{equation}
where the last map (cohomologically) sends $PD(pt) \cdot u^0 \in H^{2n}(\bar{M}, \partial \bar{M})$ to $1$,
and other elements to zero (i.e., it projects to the $u^0$ factor then integrates over $[M]$). Instead of using a $C^2$ small Hamiltonian to
define the Floer complex computing $H^{*+n}(\bar{M}, \partial \bar{M})$ (which we
only did for simultaneous compatibility with the symplectic
cohomology case), we can pass to a geometric cycle model for $\widetilde{\oc}^+$
(and therefore $\tilde{tr}$) which, as described in \S \ref{pseudocycles} directly maps (on the chain level) to $H^{*+n}(\bar{M}, \partial \bar{M}) \otimes_{\K} \K ( ( u ) ) / u \K [ [ u ]]$.
With respect to this model, the map
$\tilde{tr}$ involves counts of the moduli spaces described there
where the interior marked point $\bar{z}_{out}$ is {\em unconstrained}, e.g.,
${ }_k \check{\mc{P}}_{d}^1([M]; \vec{x})$, ${ }_k \hat{\mc{P}}_{d}^1([M];
\vec{x})$, and ${ }_k \mc{P}_{d}^{S^1}([M]; \vec{x})$; see Figure \ref{fig:trk_checkhat}.

The following well-known Lemma verifies the non-degeneracy property of the map $\tilde{tr}$:
\begin{lem}[see e.g.,  \cite{Seidel:2008zr}*{(12j)}, \cite{Sheridan:2016}*{Lemma 2.4})]
    \label{poincareduality}
The corresponding morphism $[tr]:\r{HH}_{*+n}(\f) \ra \K$ is a
non-degenerate trace (or weak proper Calabi-Yau structure).  
\end{lem}

\begin{proof}[Sketch]
    This is an immediate consequence of Poincar\'{e} duality in Lagrangian
    Floer cohomology, see the references cited above.
    As a brief sketch, note that $\check{tr}^0 \circ \mu^2 = \check{tr} \circ \mu^2: \hom(X,Y) \otimes \hom(Y,X)
    \ra \K$ is chain homotopic (and hence equal in cohomology) to a chain map
    which counts holomorphic discs with an interior marked point satisfying an
    empty constraint, and two (positive) boundary asymptotics on $p$, $q$, with
    corresponding Lagrangian boundary on $x$ and $y$. Via a further homotopy of
    Floer data, one can arrange that the generators of $\hom(X,Y)$ and
    $\hom(Y,X)$ are in bijection (for instance if one is built out of time-1
    flowlines of $H$ and one out of time-1 flowlines of $-H$), and the only
    such rigid discs are constant discs between $p$ and the corresponding
    $p^{\vee}$.
\end{proof}

\begin{proof}[Proof of Theorem \ref{mainthm3}]
    The above discussion constructs $\tilde{tr}$ and Lemma
    \ref{poincareduality} verifies non-degeneracy.  \end{proof}

\begin{figure}[h]
    \caption{\label{fig:trk_checkhat} An image of representatives of moduli spaces ${ }_3 \check{\mc{P}}_{3}^1([M]; \vec{x})$ and ${ }_2 \hat{\mc{P}}_{4}^1([M]; \vec{x})$, which appear in the map $\tilde{tr}$.}
    \includegraphics[scale=0.9]{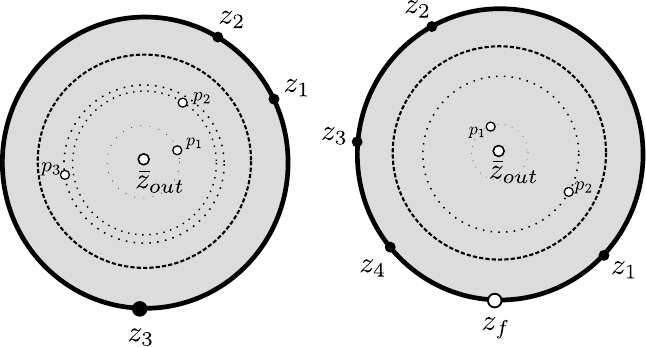}
\end{figure}

\subsection{The smooth Calabi-Yau structure on the Fukaya category}\label{subsec:smoothCY}
\def\eup{\mathcal{P}}
\def\euq{\mathcal{Q}}
\def\bimodhomc{\hom^*_{\cc\!-\!\cc}}
\def\bimodcategoryc{\cc\!-\!\mathrm{mod}\!-\!\cc}
\def\seq#1#2#3{\hom_{\cc}( {#2}_{#3-1}, {#2}_{#3}) \otimes \cdots \otimes \hom_{\cc}( {#2}_0, {#2}_1)}
\def\seqc#1#2{ \seq{\cc}{#1}{#2}}
\def\diagc{\cc_{\Delta}}
\def\biyon#1#2{\mc{Y}^l_{#1} \otimes_{\K} \mc{Y}^r_{#2}}
We give a brief overview of (a categorical version of) the notion of a {\em
(strong) smooth Calabi-Yau structure}, and construct such smooth Calabi-Yau structures on (wrapped or compact) Fukaya categories under the ``non-degeneracy'' hypotheses of \cite{ganatra1_arxiv}.  Smooth Calabi-Yau structures were proposed in the work-in-progress \cite{Kontsevich:uq} (see also \cite{Kontsevich:CY}; other expositions appear e.g., in \cite{GPS1:2015, brav_dyckerhoff_2019} --- in the latter work the terminology ``left'' is used instead of ``smooth'' and ``right'' instead of ``proper''). 
A smooth Calabi-Yau structure is expected (analogously to the proper case) by \cite{Kontsevich:uq} to induce chain-level topological field
theory operations on the Hochschild chain complex of the given category,
controlled by the open moduli space of curves with marked points equipped with asymptotic markers, at least one
of which is an output.\footnote{In contrast, note that in the proper case all operations should have at least one {\em input}.}

To state the relevant definitions, we make use of some of the theory of {\em
$\ainf$ bimodules over a category $\cc$}. We do so without much explanation,
instead referring readers to existing references \cite{Seidel:2008cr,
Tradler:2008fk, ganatra1_arxiv}.
An $\ainf$ bimodule $\eup$ over $\cc$ is a bilinear $\ainf$ functor from $\cc^{op} \times \cc$ to chain complexes, which is roughly the data, of, for every pair of objects
$\cc$, a chain complex $(\eup(X,Y), \mu^{0|1|0})$, along with `higher
multiplication maps' $\mu^{s|1|t}: \seqc{X}{s} \otimes \eup(X_0, Y_{t}) \otimes
\seqc{Y}{t} \ra \eup(X_s, Y_0)$ satisfying a generalization of the $\ainf$
equations. $\ainf$ bimodules over $\cc$ form a dg category $\bimodcategoryc$
with morphisms  denoted
$\bimodhomc( \eup, \euq)$ (for dg bimodules over a dg category, this chain
complex corresponds to a particular chain model for the `derived morphism
space' using the bar resolution).
The basic examples of bimodules we require are:
\begin{itemize}
    \item the {\em diagonal bimodule} $\diagc$ which associates to a pair of
        objects $(K,L)$ the chain complex $\diagc(K,L):= \hom_{\cc}(L,K)$.

    \item for any pair of objects $A, B$, there is a {\em Yoneda bimodule}
        $\biyon{A}{B}$ which associates to a pair of objects $(K,L)$ the chain
        complex $\biyon{A}{B}(K,L) := \hom_{\cc}(A, K) \otimes \hom_{\cc}(L,
        B)$.
\end{itemize}
Yoneda bimodules are the analogues of the free bimodule $A \otimes A^{op}$ in
the category of bimodules over an associative algebra $A$ (which are the same
as $A \otimes A^{op}$ modules). Accordingly, we say a bimodule $\eup$ is {\em
perfect} if, in the category $\bimodcategoryc$, it is split-generated by (i.e.,
isomorphic to a retract of a finite complex of) Yoneda bimodules.  We say that
a category $\cc$ is {\em (homologically) smooth} if $\diagc$ is a perfect
$\cc$-bimodule.

Recall for what follows that for any bimodule $\eup$ there is a {\em cap
product action} 
\begin{equation}
    \cap: \r{HH}^*(\cc, \eup) \otimes \r{HH}_*(\cc, \cc) \ra \r{HH}_*(\cc, \eup),
    \label{capproducthh}
\end{equation}
and hence for any class $[\sigma] \in \r{HH}_*(\cc, \cc)$ there is an induced map
\begin{equation}
    [\cap \sigma]: \r{HH}^*(\cc, \eup)  \ra \r{HH}_{*+\deg(\sigma)}(\cc, \eup).
    \label{capproducthhsigma}
\end{equation}
More generally, the cap products acts as $ \r{HH}^*(\cc, \eup) \otimes
\r{HH}_*(\cc, \euq) \ra \r{HH}_*(\cc, \eup \otimes_{\cc} \euq)$; here we are
considering $\euq = \diagc$, and then composing with the equivalence $\eup
\otimes_{\cc} \diagc \cong \eup$. See e.g., \cite{ganatra1_arxiv}*{\S 2.10} for
explicit chain level formulae in the variant case that $\eup = \diagc$, which
can be straightforwardly adapted to the general case and then specialized to
the case here.
\begin{defn}\label{weakscy}
    Let $\cc$ be a homologically smooth $\ainf$ category. A cycle $\sigma \in
    \r{CH}_{-n}(\cc, \cc)$ is said to be a 
    {\em weak smooth Calabi-Yau structure}, or a {\em non-degenerate co-trace} if, for any objects $K, L$, the
    operation of capping with $\sigma$ induces a homological isomorphism
    \begin{equation}\label{capproduct}
        [\cap \sigma]: \r{HH}^*(\cc, \biyon{K}{L}) \stackrel{\cong}{\ra} \r{HH}_{*-n}(\cc, \biyon{K}{L}) \simeq H^*(\hom_{\cc}(K,L)).
    \end{equation}
    (this non-degeneracy property only depends on the homology class $[\sigma]$).
    A cycle in the non-unital Hochschild complex $\sigma \in \r{CH}_{-n}^{nu}(\cc)$
    is said to be a weak smooth Calabi-Yau structure if again $[\sigma] \in
    H^*(\r{CH}_{-n}^{nu}(\cc)) \cong \r{HH}_{-n}(\cc)$ is non-degenerate in the
    sense of \eqref{capproduct}.
    \end{defn}
\begin{rem}
    The second isomorphism $\r{HH}_{*-n}(\cc, \biyon{K}{L}) \simeq
    H^*(\hom_{\cc}(K,L))$ always holds for cohomologically unital
    categories (such as the Fukaya category); the content is in the first.
\end{rem}
\begin{rem} 
    Continuing Remark \ref{ccybimodule}, there is an alternate perspective on
    Definition \ref{weakscy} using bimodules. Namely, for any bimodule $\eup$,
    there is a naturally associated {\em bimodule dual} $\eup^!$, defined for a
    pair of objects $(K,L)$ as the chain complex $\eup^!(K,L) :=
    \bimodhomc(\eup, \biyon{K}{L})$. The higher bimodule structure is defined
    in \cite{ganatra1_arxiv}*{Def. 2.40}; it is an $\ainf$ analogue of defining, for
    an $A$ bimodule $B$, $B^!:= \mathrm{RHom}_{A \otimes A^{op}}(B, A \otimes
    A^{op})$ where $\mathrm{RHom}$ is taken with respect to the outer bimodule
    structure on $A \otimes A^{op}$ and the bimodule structure on $B^!$ comes
    from the inner bimodule structure; see e.g., \cite{Ginzburg:2005aa}*{\S 20.5}.

    We abbreviate $\cc^!:= \cc_{\Delta}^!$, and call $\cc^!$ the {\em inverse
    dualizing bimodule}, following \cite{Kontsevich:2009ab} (observe $H^*(\cc^!(K,L))
    \cong \r{HH}^*(\cc, \biyon{K}{L})$).  For a homologically smooth
    category $\cc$ one notes that there is a
    quasi-isomorphism $\r{CH}_{*-n}(\cc) \simeq \bimodhomc(\cc_{\Delta}^![n],\cc_{\Delta})$
    (see \cite{Kontsevich:2009ab}*{Rmk. 8.11} for the case of $\ainf$
    algebras), where the equivalence associates to any element, the bimodule
    morphism whose cohomology level map  
    is the cap product operation \eqref{capproduct}. Non-degenerate
    cotraces in $\r{CH}_{-n}(\cc)$ then correspond precisely then to bimodule quasi-isomorphisms $\cc^![n]
    \stackrel{\sim}{\ra} \cc_{\Delta}$. Further discussion of these structures
    in the $\ainf$ categorical setting will appear as part of forthcoming work
    with Cohen \cite{CohenGanatra:2015}.
\end{rem}
Let $\iota: \r{CC}^-_*(\cc) \ra \r{CH}_*^{nu}(\cc)$ denote the inclusion of homotopy fixed points chain map from \eqref{inclusionhomotopyfixedpoints};
concretely as described in \eqref{inclusionhomotopyfixedpoints_explicit} this is the chain map sending $\sum_{i=0}^{\infty} \alpha_i u^i
\mapsto \alpha_0$.  
\begin{defn}
    Let $\cc$ be a homologically smooth $\ainf$ category. A {\em (strong)
    smooth Calabi-Yau structure} is a cycle $\tilde{\sigma} \in
    \r{CC}^-_{-n}(\cc)$ such that the corresponding element
    $\iota(\tilde{\sigma}) \in \r{CH}_{-n}^{nu}(\cc)$ is a weak smooth Calabi-Yau
    structure.
\end{defn}
Using these definitions and the cyclic open-closed map, we restate and prove
Theorem \ref{thm:smoothCY}. We adopt the notation of wrapped Fukaya categories
in the below result, using $\w$ and $SC^*(M)$ in place of $\mc{F}$ and
$CF^*(M)$ (with the understanding that for a compact symplectic manifold, these
are the same).
\begin{thm}[Theorem \ref{thm:smoothCY}
    above]
    Suppose a Liouville (or compact admissible symplectic) manifold is {\em
    non-degenerate} in the sense of \cite{ganatra1_arxiv}, meaning that the map
    $[\oc]: \r{HH}_{*-n}(\w) \ra SH^*(M)$ hits 1.  Then, the Fukaya category
    $\w$ possesses a (cohomologically) canonical geometrically defined {\em
    smooth Calabi-Yau structure}.  
\end{thm}
\begin{proof}
    In \cite{ganatra1_arxiv} it was proven that, assuming non-degeneracy of $M$, the map
    $[\oc]: \r{HH}_{*-n}(\w) \ra SH^*(M)$ is an isomorphism,  $\w$ is
    homologically smooth, and moreover that the pre-image $[\sigma]$ of 1 gives
    a {\em weak smooth Calabi-Yau structure} in the sense described above (see
    \cite{Gautogen_arxiv, GPS1:2015, GPS2:2015} for a proof of some of
    these facts specifically tailored to the case of compact Lagrangians in
    compact symplectic manifolds).
    Let us briefly recall how
    the non-degeneracy condition \eqref{capproduct} is proven (which is left
    slightly implicit in \cite{ganatra1_arxiv}): First,
    a geometric morphism of bimodules $\mathcal{CY}: \w_{\Delta} \to \w^![n]$
    is constructed and shown in \cite{ganatra1_arxiv}*{Thm. 1.3} to be a
    quasi-isomorphism under the given non-degeneracy hypotheses. Then, it is
    shown that capping
    with $[\sigma]$ is a
    one-sided inverse to the homological map $[\mathcal{CY}]$, and thus an
    isomorphism also, by the following argument: the following commutative
    (up to an overall sign of $(-1)^{n(n+1)/2}$) diagram is established (which
        can be thought of as coming from the compatibility of $\oc$ with module
        structures for Hochschild (co)homology with
    coefficients in $\biyon{K}{L}$, and which can be extracted from
    the holomorphic curve theory appearing in
\cite{ganatra1_arxiv}*{Thm. 13.1}):
    \begin{equation}
        \xymatrix@=1.5cm{ \r{HH}_{*-n}(\w, \w) \otimes H^*(\hom_{\w}(K,L)) \ar[r]^{(id, [\mathcal{CY}])\ \ \ \ } \ar[d]^{([\oc], id)}& \r{HH}_{*-n}(\w, \w) \otimes HH^{*+n}(\w, \biyon{K}{L}) \ar[d]^{\cap}\\
        SH^*(M) \otimes H^*(\hom_{\w}(K,L)) \ar[r]^{[\mu^2(\co_0(-), -)]\ \ \ \ \ \ \ \ } \ar[r]& H^*(\hom_{\w}(K,L)) = \r{HH}_*(\w, \biyon{K}{L})}
    \end{equation}
        (here $[\co_0]$ is the length-zero part of the closed open map for the object $L$, mapping $SH^*(M)$ to $H^*(\hom_{\mc{W}}(L,L))$).
    Plugging $[\sigma]$ into $\r{HH}_*(\w,\w)$ and noting
    $[\oc]([\sigma]) = 1$ and $[\mu^2(\co_0(1),-)] = [\mu^2]([e_L], -)$ is the
    identity map establishes as desired that $[\sigma \cap (\mathcal{CY}(y))] =
    [y]$.

    To lift the weak smooth Calabi-Yau structure to a (strong) smooth Calabi-Yau structure, first we note that, because $[\oc]$ is an isomorphism, Corollary \ref{cor:cyclicOC} implies that there is a
    commutative diagram of isomorphisms:
    \begin{equation}\label{inclusionhomotopyfixedpointsOC}
        \xymatrix{
            \r{HC}^-_{*-n}(\w) \ar[r]_{[\iota]} \ar[d]_{[\widetilde{\oc}^-]} & \r{HH}_{*-n}(\w) \ar[d]^{[\oc]}\\
        H^*(SC^*(M)^{hS^1}) \ar[r]_{[\iota]} & SH^*(M)},
        \end{equation}
        where the horizontal maps $\iota$ are the inclusion of
        homotopy fixed points maps $\iota: P^{hS^1} \ra P$ defined for any
        $S^1$-complex $P$, sending $\sum_{i=0}^{\infty} \alpha_i u^i \mapsto
        \alpha_0$. 

        In \S \ref{sec:interior}, and specifically \eqref{tilde1}, it was shown
        that there is a canonical geometrically defined element $\tilde{1} \in
        H^*(SC^*(M)^{hS^1})$ lifting the unit $1 \in SH^*(M)$--- essentially
        this is because the map $1$ is in the image of the map $H^*(M) \to
        SH^*(M)$, which on the chain level (as this map comes from ``the
        inclusion of constant loops into the free loop space'' and ``constant
        loops are acted on by $S^1$ trivially'') can be canonically lifted to a map
        $C^*(M) \to C^*(M)^{hS^1} = C^*(M)[ [ u] ] \to SC^*(M)^{hS^1}$.

        Since $[\widetilde{\oc}^-]$ is an isomorphism, it follows that there is
        a unique (cohomological) element $[\widetilde{\sigma}] \in
        \r{HC}^-_{*-n}(\w)$ hitting
        $\tilde{1}$ via $[\widetilde{\oc}]$. By \eqref{inclusionhomotopyfixedpointsOC},
        $[\iota]([\widetilde{\sigma}])] = [\sigma]$, establishing that (any cycle
        representing) $[\widetilde{\sigma}]$ is a smooth Calabi-Yau structure. 
    \end{proof}

\appendix
\section{Moduli spaces and operations}
\subsection{A real blow-up of Deligne-Mumford space}\label{modulispaces}

We review, in a special case, the compactifications of moduli spaces of
surfaces where some interior marked points are equipped with asymptotic
markers, which are a real blow-up of Deligne-Mumford moduli space as
constructed in \cite{Kimura:1995fk}. In particular, we show how boundary strata of the abstract
compactifications in the sense of \cite{Kimura:1995fk} can be identified with
the specific models of the moduli spaces we use in Section
\ref{section:openclosed1} (the appearance of the compactifications \cite{Kimura:1995fk} in Floer theory are not new, see e.g., \cite{Seidel:2010uq}).

To begin, let 
\begin{equation} \mc{M}_{2,0} \end{equation}
denote the space of spheres with $2$ marked points $z_1, z_2$ removed and
asymptotic markers $\tau_1$, $\tau_2$ around the $z_1$ and $z_2$, modulo
automorphism. 
Fixing the position of $z_1$ and $z_2$ and one of $\tau_1$ or $\tau_2$ gives a
diffeomorphism \[ \mc{M}_{2,0} \cong S^1. \] On an arbitrary representative in
$\mc{M}_{2,0}$, we can think of the map to $S^1$ as coming from the {\it
difference in angles} between $\tau_1$ and $\tau_2$ (after, say, parallel
transporting one tangent space to the other along a geodesic path).

It is convenient to parametrize this difference by a point on the sphere
itself, in the following manner (though this will break symmetry between $z_1$
and $z_2$). Let
\begin{equation} \label{m21}
\mc{M}_{2,1} 
\end{equation}
be the space of spheres with $2$ marked points $z_1, z_2$ removed, an extra
marked point $p$ and asymptotic markers $\tau_1$, $\tau_2$ around the $z_1$
and $z_2$, modulo automorphism such that, for any representative with position
of $z_1$, $z_2$, and $p$ fixed, $\tau_2$ is pointing towards $p$. The remaining
freedom in $\tau_1$ once more gives a diffeomorphism $\mc{M}_{2,1} \cong S^1$. 

We can take a different representative for elements of $\mc{M}_{2,1}$: up to
biholomorphism any element of \eqref{m21} is equal to a cylinder sending $z_1$
to $+\infty$, $z_2$ to $-\infty$) with fixed asymptotic direction around
$+\infty$ and an extra marked point $p$ at fixed height freely varying around $S^1$, 
such that the asymptotic marker at $-\infty$ coincides with the $S^1$
coordinate of $p$.  Thus, we obtain an identification 
\begin{equation} \label{cylindermodel}
    \mc{M}_{2,1} \cong \mc{M}_1
\end{equation}
where $\mc{M}_1$ is the space in Definition \ref{rpointedmodulispace} (with
$p_1$ corresponding to $p$ here).

Denote by 
\begin{equation}\label{equivmoduli}
    { }_k\mc{R}_{d}^1
\end{equation}
the moduli space of discs $(S, z_1, \ldots, z_d, z_{out}, \tau_{z_{out}}, p_1, \ldots, p_k)$ with
$d$ boundary marked points $z_1, \ldots, z_d$ arranged in counterclockwise order, an
interior marked point with asymptotic marker $(z_{out},\tau_{z_{out}})$, and interior marked points
with no asymptotic markers $p_1, \ldots, p_k$ satisfying two constraints to be described below, modulo automorphism. Up to automorphism, every equivalence class of the unconstrained moduli space of such $(S, z_1, \ldots, z_d, z_{out}, \tau_{z_{out}}, p_1, \ldots, p_k)$ 
admits a unique unit disc representative with $z_d$ fixed
at $1$ and $z_{out}$ at 0; call this the {\em $(z_d, z_{out})$ standard representative} (or simply the standard representative). The positions
of the asymptotic marker, remaining marked points, and interior marked points
identify this unconstrained moduli space
with an open subset of $S^1 \times \R^{2k} \times
\R^d$. With respect to this identification, the space \eqref{equivmoduli} consists of those discs satisfying the following (open) ``ordering constraint'' on
the positions of the interior marked points:
\begin{equation}\label{orderingappa}
    \textrm{On the standard representative, } 0 < |p_1| < |p_2| < \cdots <  |p_k| < \frac{1}{2}
\end{equation}
along with a (codimension 1) condition on the asymptotic marker:
\begin{equation}
    \textrm{On the standard representative, $\tau_{z_{out}}$ points at $p_1$}.
\end{equation}
The condition \eqref{orderingappa}, which cuts out a manifold with corners of
the larger space in which the $p_i$ are unconstrained, is technically
convenient, as it reduces the types of bubbles that can occur with $z_{out}$).
The compactification of interest, denoted
\begin{equation}\label{ksvcompact}
    { }_k\overline{\mc{R}}_{d}^1
\end{equation}
differs from the Deligne-Mumford compactification in a couple respects: firstly,
we allow points $p_i$ and $p_{i+1}$ to be coincident without bubbling off
(alternatively, we can Deligne-Mumford compactify and collapse the relevant
strata).

More interestingly, \eqref{ksvcompact} is a real blow-up of the usual
Deligne-Mumford compactification along any strata in which $z_{out}$ and $p_i$ points
bubble off (as in \cite{Kimura:1995fk}). We will proceed to describe the codimension-1
boundary strata of \eqref{ksvcompact} along with (after identification with the
moduli spaces we introduce in this paper) the boundary chart gluing maps.
Let $\Sigma = S_0 \cup_{z_{int}^+ = z_{int}^-} S_1$ denote a nodal surface, where 
\begin{itemize}
    \item $S_0$ is a sphere containing interior marked points $(z_{out}, \tau_{z_{out}})$,
        $p_1, \ldots, p_j$, and another marked point $z_{int}^+$, and
\item $S_1$ is a disc with
    $d$ boundary marked points $z_1, \ldots, z_d$, and interior marked points $z_{int}^-$,
$p_{j+1}$, \dots, $p_{k}$. 
\end{itemize}
To occur as a possible degenerate limit of \eqref{equivmoduli}, the relevant
points $p_i$ on $S_0$ and $S_1$ must satisfy an ordering condition:
\begin{align}
    \label{orderingappa1}&\textrm{For any biholomorphic $S_0'$ to $S_0$ with $z_{out}$ and $z_{int}^+$ at opposite poles, }\\
    \nonumber &\textrm{$0 < |p_1| < \cdots < |p_j| < |z_{int}^+|$, where $|p|$ denotes the geodesic distance from $z_{out}$ to $p$ on $S_0'$}.\\
    \label{orderingappa2}&\textrm{For the $(z_d, z_{int}^-)$ standard representative of $S_1$, $0 < |p_{j+1}| < \cdots < |p_k| < \frac{1}{2}$}.
\end{align}
Also, 
\begin{equation}
    \label{pointingappa3}\textrm{For $S_0'$ as in \eqref{orderingappa1}, the asymptotic marker $\tau_{z_{out}}$ should point (geodesically) towards $p_1$}.
\end{equation}
The relevant codimension-1 stratum of \eqref{ksvcompact} consists of all (automorphism classes of) such
broken configurations $S_0 \cup_{z_{int}^+ = z_{int}^-} S_1$ as above
equipped additionally with a {\em gluing angle} at the node, which is a
real positive line $\tau_{z_{int}^+,z_{int}^-}$ in $T_{z_{int}^+} S_0 \otimes
T_{z_{int}^-} S_1$, or equivalently, a pair of asymptotic markers $(\tau_{z_{int}^+}, \tau_{z_{int}^-})$ around each of
$z_{int}^+$ and $z_{int}^-$, modulo the diagonal $S^1$ rotation action. Note 
that the set of gluing angles (which is allowed to vary) is $S^1$, making this stratum codimenion-1
(the corresponding stratum in Deligne-Mumford space does not have gluing angles, and hence
has real codimension 2). The gluing map takes, for a fixed pair of cylindrical ends around $z_{int}^+$ and $z_{int}^-$ compatible with the pair of asymptotic markers in the sense of \eqref{compatiblewithmarker}, the usual gluing with respect to the chosen cylindrical ends. Note first that for a given gluing parameter, if the cylindrical ends are chosen to simply rotate as $(\tau_{z_{int}^+}, \tau_{z_{int}^-})$ vary, the result of gluing after rotating $\tau_{z_{int}^+}$ by $\theta_1$ and $\tau_{z_{int}^-}$ by $\theta_2$ differs from the initial gluing by a rotation of the bottom component by $\theta_2-\theta_1$.  In particular, the glued surface only indeed depends on the gluing angle associated to $(\tau_{z_{int}^+}, \tau_{z_{int}^-})$ i.e., is unchanged by simultaneously rotating $(\tau_{z_{int}^+}, \tau_{z_{int}^-})$.

We can recast this stratum by taking a slice of the quotient by the diagonal
$S^1$ action appearing in the definition of gluing angle:
First, note that 
$z_{int}^-$ on $S_1$ possesses a canonical asymptotic marker $(\tau_{z_{int}}^-)_{canon}$ which (on the standard representative) points towards
$p_{j+1}$ (our convention is that $p_{s+1}= z_d$ so $\tau_{z_{int}}^-$ points at $z_d$ if $j=s$). Choosing the representative $(\tau_{z_{int}^+}, \tau_{z_{int}^-})$ of each gluing angle for which $\tau_{z_{int}^-}$ is the canonical asymptotic marker $(\tau_{z_{int}}^-)_{canon}$, we see that the stratum described above can be identified with the space of broken configurations $S_0 \cup_{z_{int}^+ = z_{int}^-} S_1$ (up to automorphism) of the form: 
\begin{itemize}
    \item $S_1$ is as above (i.e., satisfies \eqref{orderingappa2}) but additionally equipped with $(\tau_{z_{int}}^-)_{canon}$, i.e., $S_1 \in { }_{k-j}\mc{R}^d_1$; and 

    \item $S_0$ is equipped with interior marked points with asymptotic markers $(z_{out}, \tau_{z_{out}})$, $(z_{int}^+, \tau_{z_{int}^+})$ and additional marked points $p_1, \ldots, p_j$ satisfying \eqref{orderingappa1} and \eqref{pointingappa3}.
\end{itemize}
In a manner as in \eqref{cylindermodel}, the space of such $S_0$ up to
biholomorphism is precisely $\mc{M}_j$ as in Definition \ref{rpointedmodulispace} (i.e., given any $S_0$, there is a 1-dimensional space of biholomorphisms to a cylinder sending $z_{int}$ and $z_{out}$ to $\infty$ and $-\infty$ while fixing the angle of $\tau_{z_{int}^+}$ to $1$; any two such biholomorphisms differ by translation).

Hence, we've identified this stratum with
\begin{equation}
     { }_{k-j}\mc{R}^{d}_1 \times \mc{M}_{j} 
\end{equation}
which will be useful in defining the relevant pseudoholomorphic curve counts. From this perspective, the boundary chart gluing maps (defined with respect to the cylindrical ends \eqref{posendangles} and \eqref{negendangles} on $\mc{M}_j$ and with respect to a smoothly varying choice of cylindrical end over elements of ${ }_{k-j}\mc{R}^{d}_1$ compatible with $(\tau_{z_{int}}^-)_{canon}$) just as in \eqref{gluingangles}, rotate the (standard representative of the) angle-decorated cylinder $S_0$ to match the angle of its top asymptotic marker with the angle of $(\tau_{z_{int}}^-)_{canon}$ (which coincides with the argument of $p_{j+1}$ on the standard representative). In other words, if we denote by $\theta_i$ the angle of $p_i$ in $S_1$ for $j+1 \leq i \leq k$ (with respect to any standard representative of $S_1$, with the usual convention that $\theta_{k+1}$ is the argument of $z_d$ on the standard representative, so in particular $\theta_{j+1}$ is well-defined even if $j=k$) and $\bar{\theta}_s$ the angle of $p_s$ in $S_0$ for $1 \leq s \leq j$, the gluing of $S_0$ and $S_1$ for small gluing parameter has (on its standard representative) marked points $p_1, \ldots, p_k$ with the following angles:
\begin{equation}\label{gluinganglesopenclosed}
    (\mathrm{arg}(p_1), \ldots, \mathrm{arg}(p_k))  
    = \left( \bar{\theta}_1 + \theta_{j+1}, \bar{\theta}_2 + \theta_{j+1}, \ldots, \bar{\theta}_{j} + \theta_{j+1}, \theta_{j+1}, \theta_{j+2}, \ldots, \theta_{k} \right).
\end{equation}

\subsection{Operations with a forgotten marked point}\label{discsforgotten}
We introduce auxiliary degenerate operations that will arise as the
codimension 1 boundary of the open-closed map and equivariant structure. This
subsection is a very special case of the general discussion in \cite{ganatra1_arxiv}.

Let $d \geq 2$ and $i \in \{1, \ldots, d\}$. The {\em moduli space of discs
with $d$ marked points with $i$th boundary point forgotten}
\begin{equation}
    \mc{R}^{d,f_i}
\end{equation}
is exactly the moduli space of discs $\mc{R}^d$, with $i$th boundary marked
point labeled as auxiliary. 

The Deligne-Mumford compactification
\begin{equation}\label{dmforgotten}
    \overline{\mc{R}}^{d,f_i}
\end{equation}
is exactly the usual Deligne-Mumford compactification, along with the data of
an {\it auxiliary label} at the relevant boundary marked point.

For $d > 2$, the {\it $i$-forgetful map}
\begin{equation}\label{iforget}
    \mc{F}_{d,i}: \mc{R}^{d,f_i} \ra \mc{R}^{d-1}.
\end{equation}
associates to a surface $S$ the surface obtained by putting the $i$th point
back in and forgetting it.  
This map admits an extension to the Deligne-Mumford
compactification 
\begin{equation}
    \overline{\mc{F}}_{d,i}: \overline{\mc{R}}^{d,f_i} \ra \overline{\mc{R}}^{d-1}
\end{equation} 
as follows: eliminate any non-main components with only one non-auxiliary
marked point $p$, and label the positive marked point below this component by
$p$. We say that any component not eliminated is {\it f-stable} and any
component eliminated is {\it f-semistable}.

The above map is only well-defined for $d > 2$.  In the semi-stable case $d =
2$, $\mc{R}^{2,f_i}$ is a point so one can define an ad hoc map 
\begin{equation}
    \mc{F}^{ss}_i: \mc{R}^{2,f_i} \ra pt.
\end{equation}
which associates to a surface $S$ the (unstable) strip $\Sigma_1 =
(-\infty,\infty) \times [0,1]$ as follows: take the unique representative of $S$
which, after its three marked points are removed, is biholomorphic to the strip
$\Sigma_1$ with an additional puncture  
$(0, 0)$. Then, forget/put back in the point $(0,0)$.

\begin{defn}
    A {\em forgotten Floer datum} for a stable disc with $i$th point auxiliary
    $S \in \overline{\mc{R}}^{d,f_i}$ consists of, for every component $T$ of $S$,
    \begin{itemize}
        \item a Floer datum for $T$, if $T$ does not contain the auxiliary point,
        \item a Floer datum for $\mc{F}_j(T)$, if $T$ is $f$-stable and
            contains the auxiliary point as its $j$th input.
        \item 
            A Floer datum on $\mc{F}^{ss}_i(T)$ which is {\em translation invariant}, if $T$ is {\it
            f-semistable}.
    \end{itemize}
    (by {\em translation invariant}, we mean the following: note that
    $\Sigma_1$ has a canonical $\R$-action given by linear translation in the
    $s$ coordinate. We require $H$, $J$, and the time-shifting map/weights to
    be invariant under this $\R$ action, and in particular should only depend
    on $t \in [0, 1]$ at most).
\end{defn}
In particular, this Floer datum should only depend on the point
$\overline{\mc{F}}_{d,i}(S)$.
\begin{prop}
    Let $i \in \{1,\ldots, d\}$ with $d > 1$.  Then the operation associated to
    $\overline{\mc{R}}^{d,f_i}$ is zero if $d > 2$ and the identity operation
    $I(\cdot)$ (up to a sign) when $d=2$.
\end{prop}
\begin{proof}[Sketch]

    Suppose first that $d>2$, and let $u$ be any solution to Floer's equation
    over the space $\mc{R}^{d,f_i}$ with domain $S$. 
    Since the Floer data on $S$
    only depends on $\mc{F}_{d,i}(S)$, we see that maps from $S'$ with $S' \in
    \mc{F}_{d,i}^{-1}(\mc{F}_{d,i}(S))$ also give solutions to Floer's equation with
    the same asymptotics. Moreover, the fibers of the map $\mc{F}_{d,i}$ are
    one-dimensional, implying that $u$ cannot be rigid, and thus the associated
    operation is zero.

    Now suppose that $d = 2$.  Then the forgetful map associates to the single
    point $[S] \in \mc{R}^{2,f_i}$ the unstable strip with its translation
    invariant Floer datum. Since non-constant solutions can never be rigid (as,
    by translating, one can obtain other non-constant solutions), it follows
    that the only solutions are constant ones, and  
    that the resulting operation is therefore the identity.
\end{proof}

\bibliography{math_bib}
\bibliographystyle{alpha}

\end{document}